\documentclass[12pt,oneside]{amsart}

\usepackage{geometry}                
\usepackage{graphicx}
\usepackage{amssymb}
\usepackage{color}
\usepackage{bbm, dsfont}
\usepackage[hidelinks]{hyperref}

\hypersetup{
	colorlinks=false,
	pdfborder={0 0 0},
	pdfborderstyle={/S/U/W 0},
}

\numberwithin{equation}{section}

\usepackage{verbatim}

\usepackage[dvipsnames]{xcolor}

\usepackage{ulem}
\newtheorem{theorem}{Theorem}[section]
\newtheorem{lemma}[theorem]{Lemma}
\newtheorem{corollary}[theorem]{Corollary}
\newtheorem{proposition}[theorem]{Proposition}

\theoremstyle{definition}

\newtheorem*{conjecture}{Conjecture}

\theoremstyle{remark}
\newtheorem{remark}[theorem]{Remark}
\newtheorem*{remark*}{Note}

\numberwithin{equation}{section}

\usepackage{amsmath}
\usepackage{mathrsfs}

\usepackage{verbatim}
\usepackage{amsmath,amssymb,amsthm}              
\usepackage{amssymb}      
  \usepackage{amsfonts}           
  \usepackage{mathrsfs}            
\usepackage{amsthm}
\usepackage{algorithm}  
\usepackage{algorithmicx}  
\usepackage{algpseudocode}  
\usepackage{mathtools}
\usepackage{commath}
\usepackage{physics}
\usepackage{bm}
\usepackage{graphicx}

\usepackage{float}
\usepackage{listings}
\usepackage{subfigure}
\usepackage{multirow}
\usepackage{color}
\usepackage{bbm}
\usepackage[export]{adjustbox}
\usepackage{enumerate}
\usepackage{bbm}

\usepackage{amsmath}
\usepackage{mathrsfs}

\newcommand{\RNum}[1]{\uppercase\expandafter{\romannumeral #1\relax}}

\newcommand{\specificthanks}[1]{\@fnsymbol{#1}}

\DeclareFontFamily{OML}{rsfs}{\skewchar\font'177}
\DeclareFontShape{OML}{rsfs}{m}{n}{ <5> <6> rsfs5 <7> <8> <9>
	rsfs7 <10> <10.95> <12> <14.4> <17.28> <20.74> <24.88> rsfs10 }{}
\DeclareMathAlphabet{\mathfs}{OML}{rsfs}{m}{n}

\newcounter{cnstcnt}
\newcommand{\cl}{%
	\refstepcounter{cnstcnt}%
	\ensuremath{c_{\thecnstcnt}}}
\newcommand{\cref}[1]{\ensuremath{c_{\ref*{#1}}}}

\newcounter{newcnstcnt}
\newcommand{\Cl}{%
	\refstepcounter{newcnstcnt}%
	\ensuremath{C_{\thenewcnstcnt}}}
\newcommand{\Cref}[1]{\ensuremath{C_{\ref*{#1}}}}

\DeclareFontFamily{U}{mathx}{}
\DeclareFontShape{U}{mathx}{m}{n}{<-> mathx10}{}
\DeclareSymbolFont{mathx}{U}{mathx}{m}{n}
\DeclareMathAccent{\widehat}{0}{mathx}{"70}
\DeclareMathAccent{\widecheck}{0}{mathx}{"71}

\begin{document}

	\title{Quasi-multiplicativity and regularity for metric graph Gaussian free fields}
	

	
	

		\author{Zhenhao Cai$^1$}
		\address[Zhenhao Cai]{Faculty of Mathematics and Computer Science, Weizmann Institute of Science}
		\email{zhenhao.cai@weizmann.ac.il}
		\thanks{$^1$Faculty of Mathematics and Computer Science, Weizmann Institute of Science}

		\author{Jian Ding$^2$}
		\address[Jian Ding]{School of Mathematical Sciences, Peking University}
		\email{dingjian@math.pku.edu.cn}
		\thanks{$^2$School of Mathematical Sciences, Peking University. Partially supported by NSFC Key Program Project No.12231002 and by New Cornerstone Science Foundation through the XPLORER PRIZE}
	\maketitle
	%
	%
	
	 	\begin{abstract}
	 	We prove quasi-multiplicativity for critical level-sets of Gaussian free fields (GFF) on the metric graphs $\widetilde{\mathbb{Z}}^d$ ($d\ge 3$). Specifically, we study the probability of connecting two general sets located on opposite sides of an annulus with inner and outer radii both of order $N$, where additional constraints are imposed on the distance of each set to the annulus. We show that for all $d \ge 3$ except the critical dimension $d=6$, this probability is of the same order as $N^{(6-d)\land 0}$ (serving as a correction factor) times the product of the two probabilities of connecting each set to the closer boundary of this annulus. The analogue for $d=6$ is also derived, although the upper and lower bounds differ by a divergent factor of $N^{o(1)}$. Notably, it was conjectured by Basu and Sapozhnikov (2017) that quasi-multiplicativity without correction factor holds for Bernoulli percolation on $\mathbb{Z}^d$ when $3\le d<6$ and fails when $d>6$. In high dimensions (i.e., $d>6$), taking into account the similarity between the metric graph GFF and Bernoulli percolation (which was  proposed by Werner (2016) and later partly confirmed by the authors (2024)), our result provides support to the conjecture that Bernoulli percolation exhibits quasi-multiplicativity with correction factor $N^{6-d}$.

	 	   During our proof of quasi-multiplicativity, numerous regularity properties, which are interesting on their own right, were also established. E.g., the probability of connecting a point and a general set, as a function of the point, exhibits behaviors similar to harmonic functions, including Harnack's inequality, decay rate, stability on boundary conditions, etc.

	 	   A crucial application of quasi-multiplicativity is proving the existence of the  incipient infinite cluster (IIC), which has been completed by Basu and Sapozhnikov (2017) for Bernoulli percolation for $3\le d<6$. Inspired by their work, in a companion paper, we also utilize quasi-multiplicativity to establish the IIC for metric graph GFFs, for all $d\ge 3$ except the critical dimension $6$.

	 	\end{abstract}

\section{Introduction}\label{section_intro}

The aim of this paper is to establish several regularity properties of connecting probabilities for critical level-sets of Gaussian free fields (GFF) on metric graphs. For clarity of exposition, we first review some definitions. We assume the dimension $d\ge 3$ throughout this paper. For the $d$-dimensional integer lattice $\mathbb{Z}^d$, we denote its edge set by $\mathbb{L}^d:= \{\{x,y\}:x,y\in \mathbb{Z}^d, |x-y|=1\}$, where $|\cdot|$ represents the  Euclidean norm. For each $e=\{x,y\}\in \mathbb{L}^d$, consider a compact interval $I_e$ of length $d$  (here the length $d$ is chosen only for convenience and does not cause any difference on the geometry of the GFF) whose endpoints are identical to the lattice points $x$ and $y$ respectively. The union of all these intervals forms the metric graph of $\mathbb{Z}^d$, denoted by $\widetilde{\mathbb{Z}}^d:=\cup_{e\in \mathbb{L}^d}I_e$. The main focus of this paper $\{\widetilde{\phi}_v\}_{v\in \widetilde{\mathbb{Z}}^d}$, i.e., the GFF on $\widetilde{\mathbb{Z}}^d$, can be sampled by the following two steps.

\begin{enumerate}
	\item Restricted to the lattice $\mathbb{Z}^d$, $\{\widetilde{\phi}_x\}_{x\in \mathbb{Z}^d}$ is distributed as a discrete GFF on $\mathbb{Z}^d$, i.e., mean-zero Gaussian random variables with covariance 
\begin{equation}
		\mathbb{E}[\phi_x\phi_y]=G(x,y),\ \ \forall x,y\in \mathbb{Z}^d,
	\end{equation}
	where $G(x,y)$ is the Green's function and equals to the expected number of visits to $y$ by a simple random walk on $\mathbb{Z}^d$ starting from $x$.

	\item  Given all values of $\widetilde{\phi}_\cdot$ on $\mathbb{Z}^d$, $\{\widetilde{\phi}_v\}_{v\in I_e}$ for $e\in \mathbb{L}^d$ are independent Brownian bridges generated by Brownian motions with variance $2$ at time $1$, whose boundary values match $\{\widetilde{\phi}_x\}_{x\in \mathbb{Z}^d}$.

\end{enumerate}
It was proved in \cite{lupu2016loop} that the GFF level-set $\widetilde{E}^{\ge h}:=\big\{v\in \widetilde{\mathbb{Z}}^d:\widetilde{\phi}_v\ge h\big\}$ ($h\in \mathbb{R}$) a.s. percolates (i.e., includes an infinite connected component) for all $h<0$, and a.s. does not percolate for all $h\ge 0$. As a natural way to understand the clusters of the GFF level-set at the critical level $\widetilde{h}_*=0$, the (critical) connecting probabilities have been extensively studied. Precisely, for any disjoint subsets $A_1,A_2\subset \widetilde{\mathbb{Z}}^d$, let $A_1\xleftrightarrow{\ge h} A_2$ be the event that there exists a path in $\widetilde{E}^{\ge h}$ connecting $A_1$ and $A_2$ (when $A_i=\{v\}$ for some $v\in \widetilde{\mathbb{Z}}^d$, the braces may be omitted). It was proved in \cite[Proposition 5.2]{lupu2016loop} that the two-point function 
\begin{equation}\label{two_point}
	\mathbb{P}\big(x\xleftrightarrow{\ge 0} y\big)= \pi^{-1}\arcsin\bigg(\frac{G(x,y)}{\sqrt{G(x,x)G(y,y)}}\bigg) \asymp |x-y|^{2-d},
	\end{equation}
for all $x\neq y\in \mathbb{Z}^d$, where $f\asymp g$ means that $cg\le f\le Cg$ holds for some constants $c$ and $C$ depending only on $d$. Moreover, the (critical) one-arm probability $\theta_d(N):= \mathbb{P}(\bm{0}\xleftrightarrow{\ge 0}\partial B(N))$ is also of great interest, where $\bm{0}:=(0,0,...,0)$ is the origin of $\mathbb{Z}^d$, $B(N):=[-N,N]^d\cap \mathbb{Z}^d$, and $\partial A:= \{x\in A: \exists y\in \mathbb{Z}^d\setminus A\ \text{such that}\ \{x,y\}\in \mathbb{L}^d\}$. Through the collective efforts from \cite{ding2020percolation,cai2024high,drewitz2023arm,cai2024one,drewitz2024critical1} (also see \cite{drewitz2023critical} on extensions to more general transient graphs), it was finally established that (see \cite{cai2024high} for $d>6$; see \cite{cai2024one} for $3\le d\le 6$ and also see a concurrent work \cite{drewitz2024critical1} for $d=3$) 
\begin{align}
	&\text{when}\ 3\le d<6,\ \ \ \ \ \ \ \ \  \theta_d(N) \asymp N^{-\frac{d}{2}+1};\label{one_arm_low} \\
	&\text{when}\ d=6,\ \ \ \ \  N^{-2}\lesssim \theta_6(N) \lesssim N^{-2+\varsigma(N)}, \  \text{where}\ \varsigma(N):= \tfrac{\ln\ln(N)}{\ln^{1/2}(N)}\ll 1; \label{one_arm_6} \\
	&\text{when}\ d>6,\ \ \ \ \ \ \ \ \ \ \ \ \ \ \ \theta_d(N) \asymp N^{-2}.\label{one_arm_high}
\end{align}
Here $f\lesssim g$ means $f\le Cg$ for some constant $C$ depending only on $d$. In words, the exact order of $\theta_d(N)$ for all dimensions except $6$, as well as the exponent of $\theta_6(N)$, have been ascertained. In addition to the two-point function and the one-arm probability, the crossing probability $\rho_d(n,N):=\mathbb{P}\big(B(n)\xleftrightarrow{\ge 0} \partial B(N) \big)$ was also studied in \cite{cai2024one}. Specifically, it was shown in \cite[Theorem 1.2]{cai2024one} that
\begin{align}
	&\text{when}\ 3\le d<6, \ \ \ \ \ \ \ \rho_d(n,N) \asymp \big(\frac{n}{N}\big)^{\frac{d}{2}+1};\label{crossing_low} \\
	&\text{when}\ d=6,\  \big(\frac{n}{N}\big)^2  \lesssim  	\rho_6(n,N)  \lesssim   \big(\frac{n}{N^{1-\varsigma(N)}}\big)^2; \label{crossing_6} \\
	&\text{when}\ d>6, \ \ \ \ \   \ \ \ \ \  \ \    \rho_d(n,N) \asymp  (n^{d-4}N^{-2})\land 1. \label{crossing_high}
\end{align}
Readers may refer to \cite[Section 1]{cai2024one} for a detailed account on related results.

\subsection{Quasi-multiplicativity}

Instead of considering connecting probabilities between specific sets (say points or boundaries of boxes, as in the two-point function, the one-arm probability $\theta_d$ and the crossing probability $\rho_d$), in this paper we explore the ones between general sets (even with certain zero boundary condition). In \cite{basu2017kesten}, a fundamental property called \textit{quasi-multiplicativity} concerning such general connecting probabilities was conjectured to hold for Bernoulli percolation on $\mathbb{Z}^d$ with $3\le d\le 5$. Precisely, let $\mathbb{P}_p$ denote the law of Bernoulli percolation on $\mathbb{Z}^d$ with parameter $p$, and let $p_c(d)$ be the critical threshold.

\begin{conjecture}[{\cite[Assumption (A2)]{basu2017kesten}}] For any $3\le d\le 5$ and $\delta>0$, there exists a constant $c=c(d,\delta)>0$ such that for any $p\in [p_c(d),p_c(d)+\delta]$, $N\ge 1$, $Z\subset \mathbb{Z}^d$ with $B(4N)\setminus B(N) \subset Z$, $X\subset Z\cap B(N)$ and $Y\subset Z\setminus B(4N)$,  
	\begin{equation}\label{conj_QM}
		\mathbb{P}_p\big(X\xleftrightarrow{} Y\ \text{in}\ Z\big)\ge c\cdot  \mathbb{P}_p\big(X\xleftrightarrow{} \partial B(2N)\ \text{in}\ Z\big)\mathbb{P}_p\big(Y\xleftrightarrow{} \partial B(2N)\ \text{in}\ Z\big).
	\end{equation}	
	Here ``$A\xleftrightarrow{} A'$ in $Z$'' means that there exists a path within $Z$ that connects $A$ and $A'$ and consists of open edges.
\end{conjecture}

Note that the independence of Bernoulli percolation ensures that the left-hand side of (\ref{conj_QM}) is bounded from above by the product of the two probabilities on the right-hand side of (\ref{conj_QM}) (note that this argument does not work for the GFF). Thus, this conjecture implies that the two sides of (\ref{conj_QM}) are of the same order:
\begin{equation}\label{formula_QM}
	\mathbb{P}_p\big(X\xleftrightarrow{} Y\ \text{in}\ Z\big)\asymp  \mathbb{P}_p\big(X\xleftrightarrow{} \partial B(2N)\ \text{in}\ Z\big)\mathbb{P}_p\big(Y\xleftrightarrow{} \partial B(2N)\ \text{in}\ Z\big).
\end{equation}
Such a property is called quasi-multiplicativity. It was also conjectured in \cite{basu2017kesten} that (\ref{formula_QM}) does not hold in high-dimensional cases $d\ge 7$. Our main result shows that quasi-multiplicativity holds for $\{\widetilde{\phi}_v\}_{v\in \widetilde{\mathbb{Z}}^d}$ with $3\le d\le 5$; moreover, when $d\ge 7$, it also holds after multiplying a correction factor $N^{6-d}$ to the right-hand side.

We use $D$ (sometimes with superscripts or subscripts) to denote a subset of $\widetilde{\mathbb{Z}}^d$ consisting of finitely many compact connected components. Let $\mathbb{P}^D$ denote the conditional law of $\{\widetilde{\phi}_v\}_{v\in \widetilde{\mathbb{Z}}^d}$ given that all the GFF values on $D$ equal zero. For any $n\ge 1$, we denote the box of the metric graph $\widetilde{B}(n)$ as the union of $I_{e}$ for all $e=\{x,y\}\in \mathbb{L}^d$ with $\{x,y\}\cap B(n-1)\neq \emptyset$. For convenience, we set $\varsigma(N)=\tfrac{\ln\ln(N)}{\ln^{1/2}(N)}=0$ for  $N<e^2$, and set $\frac{0}{0}:=1$. 


\begin{theorem}\label{prop_new_QM}
	For any $d\ge 3$, there exist constants $\Cl\label{const_new_QM1}(d)>\cl\label{const_new_QM2}(d)>0$ such that the following hold. When $3\le d\le 5$, for any $N\ge 1$, $A_1,D_1\subset  \widetilde{B}(\cref{const_new_QM2}N)$ and $A_2,D_2\subset [\widetilde{B}(\Cref{const_new_QM1}N)]^c$,
		\begin{equation}\label{QM_ineq_1}
	\begin{split}
			\mathbb{P}^{D_1\cup D_2}\big( A_1\xleftrightarrow{\ge 0} A_2\big) 
			 \asymp  \mathbb{P}^{D_1}\big( A_1 \xleftrightarrow{\ge 0} \partial B(N) \big) \mathbb{P}^{D_2}\big( A_2 \xleftrightarrow{\ge 0} \partial B(N) \big).	
			 \end{split}
	\end{equation}
When $d=6$, for any $N\ge 1$, $A_1,D_1\subset  \widetilde{B}(\cref{const_new_QM2}N^{1-\varsigma(N)})$ and $A_2,D_2\subset [\widetilde{B}(\Cref{const_new_QM1}N^{1+\varsigma(N)})]^c$,
\begin{equation}\label{QM_ineq_2}
N^{-10\varsigma(N)}	\lesssim \frac{\mathbb{P}^{D_1\cup D_2}\big( A_1\xleftrightarrow{\ge 0} A_2\big)}{\mathbb{P}^{D_1}\big( A_1 \xleftrightarrow{\ge 0} \partial B(N) \big) \mathbb{P}^{D_2}\big( A_2 \xleftrightarrow{\ge 0} \partial B(N) \big)}\lesssim N^{\varsigma(N)}.
\end{equation}
When $d\ge 7$, for any $N\ge 1$, $A_1,D_1\subset  \widetilde{B}(\cref{const_new_QM2}N^{\frac{2}{d-4}})$ and $A_2,D_2\subset [\widetilde{B}(\Cref{const_new_QM1}N^{\frac{d-4}{2}})]^c$,
	\begin{equation}\label{QM_ineq_3}
	\begin{split}
			\mathbb{P}^{D_1\cup D_2}\big( A_1\xleftrightarrow{\ge 0} A_2\big)
			 \asymp  N^{6-d} \mathbb{P}^{D_1}\big( A_1 \xleftrightarrow{\ge 0} \partial B(N) \big) \mathbb{P}^{D_2}\big( A_2 \xleftrightarrow{\ge 0} \partial B(N) \big).	
			 \end{split}
	\end{equation}
\end{theorem}

\begin{remark}[application of quasi-multiplicativity]
For Bernoulli percolation on $\mathbb{Z}^d$ with $d\ge 3$, assuming that quasi-multiplicativity in (\ref{formula_QM}) holds for all $p\in [p_c(d), p_c(d) + \delta]$ with some $\delta>0$, \cite[Theorem 1.1]{basu2017kesten} established the existence and equivalence of two versions of the incipient infinite cluster (IIC). I.e., the following two limiting probability measures exist and are equivalent: 
	\begin{equation}\label{iic_bernoulli}
		\lim\limits_{N\to \infty} \mathbb{P}_{p_c}\big(\cdot \mid \bm{0}\xleftrightarrow{} \partial B(N)\big)\ \  \text{and} \ \ \lim\limits_{p\downarrow  p_c} \mathbb{P}_{p}(\cdot \mid \bm{0}\xleftrightarrow{} \infty).
	\end{equation}
In a companion paper \cite{iic}, we use quasi-multiplicativity in Theorem \ref{prop_new_QM} to prove the existence and equivalence of several versions of the IIC for $\{\widetilde{\phi}_v\}_{v\in \widetilde{\mathbb{Z}}^d}$ (including the analogues of IICs in (\ref{iic_bernoulli})) for all $d\ge 3$ except the critical dimension $6$. Notably, since the critical level-set and the slightly supercritical level-set of $\{\widetilde{\phi}_v\}_{v\in \widetilde{\mathbb{Z}}^d}$ can be related using the isomorphism theorem involving random interlacements (see \cite[Proposition 6.3]{lupu2016loop}), in \cite{iic} we only need quasi-multiplicativity at the critical level $\widetilde{h}_*=0$, as presented in Theorem \ref{prop_new_QM}. 
\end{remark}


\begin{remark}[correction factor at the critical dimension $6$]\label{remark_cf_cd6}
	When $d=6$, unlike in other dimensions, the order of the ratio in (\ref{QM_ineq_2}) is still unclear. As shown by the proofs in later sections, this issue arises because the order of the one-arm probability 	$\theta_6(N)$ is not derived yet (see (\ref{one_arm_6})). Particularly, it was conjectured in \cite[Remark 1.5]{cai2024one} that $\theta_d(N)\asymp N^{-2}\ln^{\delta}(N)$ for some constant $\delta>0$. Based on this conjecture, we expect that the critical level-set $\widetilde{E}^{\ge 0}$ on $\widetilde{\mathbb{Z}}^6$ behaves more similarly to that in high dimensions $d\ge 7$ than in low dimensions $3\le d\le 5$, and that the ratio in (\ref{QM_ineq_2}) should vanish in a poly-logarithmic rate as $N\to \infty$. 
	\end{remark}

\begin{remark}[sharpness of the conditions for high dimensions]\label{remark_sharp_condition}
	In the conditions for (\ref{QM_ineq_3}), we require that the target set $A_1$ and the zero boundary $D_1$ are contained in a box of radius $cN^{\frac{2}{d-4}}$($\ll N$), and that $A_2$ and $D_2$ are outside the box of radius $CN^{\frac{d-4}{2}}$($\gg N$). In fact, these conditions are necessary for (\ref{QM_ineq_3}) to hold. To see this, we consider a special case when $A_1=B(R_1)$, $A_2=\partial B(R_2)$ (where $R_1<N<R_2$) and $D_1=D_2=\emptyset$. By (\ref{crossing_high}), the connecting probability between $A_1$ and $A_2$ on the left-hand side of (\ref{QM_ineq_3}), which now equals $\rho_d(R_1,R_2)$, is of order $(R_1^{d-4}R_2^{-2})\land 1$. Meanwhile, (\ref{crossing_high}) also shows that the right-hand side of (\ref{QM_ineq_3}) is of order $N^{6-d}\cdot [(R_1^{d-4}N^{-2})\land 1]\cdot [(N^{d-4}R_2^{-2})\land 1]$. In this case, ``$R_1\lesssim N^{\frac{2}{d-4}}, R_2\gtrsim N^{\frac{d-4}{2}}$'' is indeed necessary for matching these two orders (since the polynomials in $N$ on the right-hand side need to be canceled). As for $d=6$, applying the same argument and taking the conjecture $\theta_6(N)\asymp N^{-2}\ln^{\delta}(N)$ into consideration, we expect that the $\varsigma(N)$ term in the conditions of (\ref{QM_ineq_2}) can be improved to $\delta\ln\ln(N)$. 
\end{remark}

\subsection{Comparison inequalities for connecting probabilities}
During the proof of Theorem \ref{prop_new_QM}, we derive some regularity properties of the connecting probability between a point and a general set (referred to as the ``point-to-set connecting probability''), showing that this type of probability, as a function of the point, exhibits behaviors similar to harmonic functions. The first property establishes that as the distance between a set and a point increases, the connecting probability between them decreases at the same rate as a harmonic function, such as the hitting probability for a simple random walk on $\mathbb{Z}^d$ (see e.g. \cite[Proposition 2.2.2]{lawler2012intersections}). For brevity, we denote $\varsigma_d(N):=\varsigma(N)\cdot \mathbbm{1}_{d=6}$ (recall $\varsigma(N)$ in (\ref{one_arm_6})). As mentioned before, all implicit constants for $\asymp$ and $\lesssim$ in the following results depend only on the dimension $d$.

\begin{proposition}\label{prop_pointtoset_speed}
	For any $d\ge 3$, there exists a constant $\Cl\label{const_pointtoset_speed}(d)>0$ such that for any $N\ge 1$, $A\subset \widetilde{B}(N)$, $x_1,x_2\in [B(\Cref{const_pointtoset_speed}N^{1+\varsigma_d(N)})]^c$, $M\ge 2(|x_1|\vee |x_2|)$ and $D\subset \widetilde{B}(N)\cup [\widetilde{B}(M)]^c$, 
	\begin{equation}\label{addnew3.25}
		|x_1|^{d-2}	\mathbb{P}^D\big( A \xleftrightarrow{\ge 0}  x_1 \big) \asymp |x_2|^{d-2}	\mathbb{P}^{D}\big( A \xleftrightarrow{\ge 0}  x_2 \big). 
	\end{equation}
\end{proposition}

The second property shows that the point-to-set connecting probability satisfies Harnack's inequality, meaning that this connecting probability remains of the same order when the point varies within a region whose diameter is bounded from above by a constant times its distance to the set.

\begin{proposition}\label{prop_pointtoset_harnack}
	For any $d\ge 3$, there exists a constant $\cl\label{const_pointtoset_harnack}(d)>0$ such that for any $N\ge 1$, $A\subset [\widetilde{B}(N)]^{c}$, $x_1,x_2\in B(\cref{const_pointtoset_harnack}N^{1-\varsigma_d(N)})$, $M\le \frac{1}{2}(|x_1|\land |x_2|)$ and $D\subset \widetilde{B}(M)\cup [\widetilde{B}(N)]^{c}$,
		\begin{equation}\label{ineqnew_3.10}
		\mathbb{P}^D\big( A \xleftrightarrow{\ge 0}  x_1 \big) \asymp \mathbb{P}^{D}\big( A \xleftrightarrow{\ge 0}  x_2 \big). 
	\end{equation}
\end{proposition}

\begin{remark}[extra conditions at the critical dimension $6$] \label{remark_extra_conditon}
	In the conditions of Propositions \ref{prop_pointtoset_speed} and \ref{prop_pointtoset_harnack}, the $\varsigma_d(\cdot)$ term appears several times. In fact, we expect that these unfavorable terms are essentially avoidable. The only reason for retaining them is that the order of the one-arm probability $\theta_6(N)$ remains unclear.
\end{remark}

The next property states that the connecting probability between a box boundary and a general set (referred to as the ``boundary-to-set connecting probability'') is stable under the addition of certain zero boundary conditions. Specifically, we show that by imposing zero boundary conditions in a region whose distances to the boundary and the set are both comparable to their mutual distance, the boundary-to-set connecting probability will change only by a constant factor. 


\begin{proposition}\label{prop_boundtoset_zero}
	For any $d\ge 3$, there exist constants $\Cl\label{prop_boundtoset_zero1}(d),\cl\label{prop_boundtoset_zero2}(d)>0$ such that 
		\begin{equation}
			\mathbb{P}^D\big( A\xleftrightarrow{\ge 0}  \partial B(M) \big) \asymp \mathbb{P}^{D\cup D'}\big( A\xleftrightarrow{\ge 0}  \partial B(M) \big) 
	\end{equation}
	if either of the following two conditions holds:
	\begin{enumerate}
		\item[(a)]  $N\ge 1$, $A,D\subset \widetilde{B}(N)$, $M\ge \Cref{prop_boundtoset_zero1}N$ and $D' \subset [\widetilde{B}( \Cref{prop_boundtoset_zero1}M)]^c$;

		\item[(b)] $N\ge 1$, $A,D\subset [\widetilde{B}(N)]^c$, $M\le \cref{prop_boundtoset_zero2}N$ and $D' \subset \widetilde{B}( \cref{prop_boundtoset_zero2}M)$.
	\end{enumerate}

\end{proposition}

In addition to studying the point-to-set and boundary-to-set connecting probabilities on their own, we also establish the quantitative relation between these two types of connecting probabilities.

For any two functions $f,g$ depending on the dimension $d$, we denote 
\begin{equation}
	(f \boxdot g) (d):= f(d)\cdot \mathbbm{1}_{d\le 6}+ g(d)\cdot \mathbbm{1}_{d>6}. 
\end{equation}

\begin{proposition}\label{prop_relation_pointandboundary}
	For any $d\ge 3$, there exist constants $\Cl\label{const_relation_pointandboundary1}(d),\cl\label{const_relation_pointandboundary2}(d)>0$ such that 
	\begin{equation}\label{ineq_relation_pointandboundary}
	M^{-\varsigma_d(N)}	\lesssim 	\frac{ \mathbb{P}^D\big(A\xleftrightarrow{\ge 0} x\big)}{\mathbb{P}^D\big(A\xleftrightarrow{\ge 0}  B(M)\big)}\cdot M^{(\frac{d}{2}-1)\boxdot (d-4)}	 \lesssim   M^{3\varsigma_d(M)}
	\end{equation}
for all $N\ge 1$ and $x\in \partial B(M)$ if either of the following two conditions holds:
\begin{enumerate}
	\item[(a)] $A\subset \widetilde{B}(N)$, $M\ge \Cref{const_relation_pointandboundary1}N^{( 1 \boxdot \frac{d-4}{2})+\varsigma_d(N)}$ and $D\subset  \widetilde{B}(N)\cup [\widetilde{B}(2M)]^c$;

	\item[(b)] $A\subset [\widetilde{B}(N)]^c$, $M\le \cref{const_relation_pointandboundary2}N^{(1 \boxdot \frac{2}{d-4})-\varsigma_d(N)}$ and $D\subset  [\widetilde{B}(N)]^c \cup \widetilde{B}(\frac{1}{2}M)$.

	\end{enumerate}

\end{proposition}

By combining the properties of point-to-set connecting probabilities in Propositions \ref{prop_pointtoset_speed} and \ref{prop_pointtoset_harnack} with the relation between point-to-set and boundary-to-set connecting probabilities in Proposition \ref{prop_relation_pointandboundary}, we directly obtain the decay rates of boundary-to-set connecting probabilities as follows.

\begin{corollary}\label{corollary_compare_boundtoset}
	(1) For any $d\ge 3$, there exists a constant $\Cl\label{const_compare_boundtoset1}(d)>0$ such that for any $N\ge 1$, $A\subset \widetilde{B}(N)$, $M_2\ge M_1 \ge \Cref{const_compare_boundtoset1}N^{(1 \boxdot \frac{d-4}{2})+\varsigma_d(N)}$ and $D\subset  \widetilde{B}(N)\cup [\widetilde{B}(2M_2]^c$, 
	\begin{equation}\label{ineq1_compare_boundtoset}
M_2^{-4\varsigma_d(M_2)}   \lesssim	\frac{\mathbb{P}^D\big(A\xleftrightarrow{\ge 0} \partial B(M_1)\big)}{\mathbb{P}^D\big(A\xleftrightarrow{\ge 0} \partial B(M_2)\big)}\cdot \big(\frac{M_1}{M_2}\big)^{(\frac{d}{2}-1)\boxdot 2}\lesssim  M_2^{4\varsigma_d(M_2)}.
\end{equation}

	(2) For any $d\ge 3$, there exists a constant $\cl\label{const_compare_boundtoset2}(d)>0$ such that for any $N\ge 1$, $A\subset [\widetilde{B}(N)]^c$, $M_1\le M_2\le  \cref{const_compare_boundtoset2}N^{(1\boxdot \frac{2}{d-4})-\varsigma_d(N)}$ and $D\subset  [\widetilde{B}(N)]^c \cup \widetilde{B}(\frac{1}{2}M_1)$, 
		\begin{equation}\label{ineq2_compare_boundtoset}
M_2^{-4\varsigma_d(M_2)}	\lesssim 	\frac{\mathbb{P}^D\big(A\xleftrightarrow{\ge 0}  B(M_1)\big)}{ \mathbb{P}^D\big(A\xleftrightarrow{\ge 0} B(M_2)\big)}\cdot \big(\frac{M_2}{M_1}\big)^{(\frac{d}{2}-1)\boxdot (d-4)}	 \lesssim  M_2^{4\varsigma_d(M_2)}.
	\end{equation}
	
\end{corollary}

As noted in Remark \ref{remark_sharp_condition}, for all $d\ge 3$ except for $d=6$, the conditions and estimates in Proposition \ref{prop_relation_pointandboundary} and Corollary \ref{corollary_compare_boundtoset} are sharp up to a constant factor. For $d=6$, based on the conjecture that $\theta_6(N)\asymp N^{-2}\ln^{\delta}(N)$ (recall Remark \ref{remark_cf_cd6}), we expect the intermediate term in (\ref{ineq_relation_pointandboundary}) to be of order $\ln^{-\delta}(M)$, and those in (\ref{ineq1_compare_boundtoset}) and (\ref{ineq2_compare_boundtoset}) to be of order $(\frac{\ln(M_1)}{\ln(M_2)})^{\delta}$, where $\delta$ refers to the same constant throughout this sentence. Moreover, we expect that the $\varsigma_d(\cdot)$ term in the conditions of Proposition \ref{prop_relation_pointandboundary} and Corollary \ref{corollary_compare_boundtoset} can be improved to $\delta \ln\ln(N)$ for $d=6$, as suggested in Remark \ref{remark_sharp_condition}.

\subsection{Volume exponent}

For $h\in \mathbb{R}$ and $v\in \widetilde{\mathbb{Z}}^d$, let $\mathcal{C}^{\ge h}(v):=\{w\in \widetilde{\mathbb{Z}}^d:w \xleftrightarrow{\ge h} v \}$ be the cluster of $\widetilde{E}^{\ge h}$ containing $v$. For any $A\subset \widetilde{\mathbb{Z}}^d$, the volume of $A$, denoted by $\mathrm{vol}(A)$, is defined as the number of lattice points in $\mathbb{Z}^d$ contained in $A$.

Volumes of critical clusters (i.e. clusters of the critical level-sets) are also natural and interesting quantities, and are highly related to the connecting probabilities. For example, the expectation of $\mathrm{vol}(\mathcal{C}^{\ge 0}\cap \widetilde{B}(N))$ can be given by the sum of connecting probabilities between $\bm{0}$ and $x$ for all $x\in B(N)$. For $d\ge 3$ and $M\ge 1$, let  $\nu_d(M):=\mathbb{P}\big(
\mathrm{vol}(\mathcal{C}^{\ge 0}(\bm{0}))\ge M \big)$. In \cite[Corollary 1.6]{drewitz2023critical}, it was derived from the bound $\theta_3(N)\lesssim N^{-\frac{1}{2}}\ln^{\frac{1}{2}}(N) $ that 
\begin{equation}\label{ineq_nu_3d}
\nu_3(M) \lesssim  M^{-\frac{1}{5}}\ln^{\frac{1}{2}}(M). 
\end{equation}
In fact, referring to \cite[Section 8]{drewitz2023critical} (see also \cite[Corollary 1.4]{drewitz2023arm}), the proof of (\ref{ineq_nu_3d}) can be easily extended to the following inequality: for all $d\ge 3$, 
\begin{equation}\label{volume_extend}
		\nu_d(M) \le  \theta_d(M^{\frac{2}{d+2}})+ CM^{-\frac{d-2}{d+2}}, 
\end{equation}
for some constant $C=C(d)>0$. For completeness, we will provide the proof of (\ref{volume_extend}) in Section \ref{section_volume}. Plugging (\ref{one_arm_low}) and (\ref{one_arm_6}) into (\ref{volume_extend}), one has
\begin{align}
	&\text{when}\ 3\le d<6,\ \ \nu_d(M) \lesssim  M^{-\frac{d-2}{d+2}};\label{upper_volume} \\
	&\text{when}\ d=6,\ \ \ \ \ \ \ \ \nu_6(M)  \lesssim M^{-\frac{1}{2}+\varsigma(M)}=M^{-\frac{1}{2}+o(1)}. \label{upper_volume_6d}
\end{align}
	In high-dimensional cases, it was established in \cite[Theorem 1.2 and (1.9)]{cai2024high} that 
	\begin{equation}
		\hspace{-3.1cm} \text{when}\ d>6,\ \ \ \ \ \ \ \ \nu_d(M) \asymp  M^{-\frac{1}{2}}.
			\end{equation}
			Referring to \cite[Inequality (1.9)]{cai2024high}, one has $\nu_d(M)\gtrsim M^{-\frac{1}{2}}$ for all $d\ge 3$, implying that (\ref{upper_volume_6d}) provides the exact exponent of $\nu_6(M)$.

As the by-product of a key step in our proof of Theorem \ref{prop_new_QM}, we establish the following bounds for $\nu_d(M)$, showing that (\ref{upper_volume}) is sharp up to a constant factor.

\begin{theorem}\label{thm_volume}
	For any $3\le d\le 5$, there exists a constant $\cl\label{const_volume_1}>0$ such that for any $M\ge 1$,
	\begin{equation}
		\nu_d(M) \ge \cref{const_volume_1}M^{-\frac{d-2}{d+2}}.
	\end{equation}
	\end{theorem}

To summarize, the exact order of $\nu_d(M)$ is now known for all dimensions except the critical dimension $6$, and the exponent of $\nu_6(M)$ has also been determined. In addition, we remark that these estimates for $3\le d\le 6$ are also derived in a concurrent work \cite{drewitz2024cluster}. Furthermore, as in \cite[Remark 1.5]{cai2024one}, we conjecture that $\nu_6(M)\asymp M^{-\frac{1}{2}}\ln^{\delta'}(M)$ for some constant $\delta'>0$.

\textbf{Statements about constants.} We use the letters $C$ and $c$ (sometimes with superscripts or subscripts) to represent constants. The uppercase letter $C$ denotes large constants and the lowercase $c$ represents small constants. Moreover, constants with numerical labels, such as $C_1,C_2,c_1,c_2,...$, are fixed throughout the paper; other constants may vary depending on the context. Unless otherwise stated, a constant can only depend on the dimension $d$. If a constant depends on parameters other than $d$, all these parameters will be specified in parentheses.

\subsection{Intuition behind quasi-multiplicativity} 

In this subsection, we provide heuristics on quasi-multiplicativity of $\{\widetilde{\phi}_{v}\}_{v\in \widetilde{\mathbb{Z}}^d}$ and explain how these heuristics are developed into rigorous proofs. Thanks to the isomorphism theorem, we can equivalently study connecting probabilities from the perspective of clusters of the critical loop soup (referred to as ``loop clusters''; see Section \ref{subsection_loopsoup} for the precise definition). Notably, these two perspectives (from GFF level-sets and loop clusters) have distinct advantages and limitations for proving the different directions of the inequality involved in quasi-multiplicativity; they complement each other and jointly establish Theorem \ref{prop_new_QM}. To facilitate the exposition, we assume throughout this subsection that $N\ge 1$ and $A_1,A_2,D_1,D_2\subset \widetilde{\mathbb{Z}}^d$ satisfy all conditions in Theorem \ref{prop_new_QM}. Moreover, for clarity, we intentionally skip the critical dimension $d=6$ due to numerous unsolved problems, as noted in Remarks \ref{remark_cf_cd6} and \ref{remark_sharp_condition}.

We first discuss the heuristic on the lower bound. I.e., 
\begin{equation}\label{newadd_1.29}
	\mathbb{P}^{D_1\cup D_2}\big( A_1\xleftrightarrow{\ge 0} A_2\big)
			 \gtrsim   N^{0\boxdot (6-d)} \mathbb{P}^{D_1}\big( A_1 \xleftrightarrow{\ge 0} \partial B(N) \big) \mathbb{P}^{D_2}\big( A_2 \xleftrightarrow{\ge 0} \partial B(N) \big).
\end{equation}
This inequality suggests the following plausible two-step scheme for the critical level-set $\widetilde{E}^{\ge 0}$ to connect $A_1$ and $A_2$:
\begin{enumerate}

\item We first explore the sign cluster containing $A_1$ (i.e., the collection of points connected to $A_1$ by $\widetilde{E}^{\ge 0}$ or $\widetilde{E}^{\le 0}$) and the one containing $A_2$, stopped upon reaching $\partial B(cN)$ and $\partial B(CN)$ respectively.  At this point, each cluster will contain a positive part intersecting $\partial B(cN)$ or $\partial B(CN)$.

	\item  The conditional probability of connecting the positive parts of these two (partial) sign clusters, given the boundary conditions from the exploration in Step (1), is typically at least of order $N^{0\boxdot (6-d) }$.

	 \textbf{P.S.} When we scale the radius of the intermediate boundary by a constant factor, the order of the boundary-to-set connecting probability remains the same, as shown in Corollary \ref{corollary_compare_boundtoset}. Therefore, there is no essential difference between considering $\partial B(cN)$ and $\partial B(CN)$ instead of $\partial B(N)$ as in Theorem \ref{prop_new_QM}. The only reason we choose $\partial B(cN)$ and $\partial B(CN)$ here is to create an annulus between the two explored clusters, making the connecting event in Step (2) appear more natural.

\end{enumerate}
To justify this scheme, we demonstrate (see Lemma \ref{lemma_newHr}) that for each $i\in \{1,2\}$, with probability at least $c\cdot \mathbb{P}^{D_i}\big( A_i \xleftrightarrow{\ge 0} \partial B(N) \big)$, the boundary condition after exploring the sign cluster containing $A_i$ in Step (1) is sufficiently positive (we quantify this positivity using the concept of ``harmonic average''; see (\ref{def_Hv})). Building on this, we apply the FKG inequality to show that the intersection of these two increasing events still has a significant probability, and then we apply a powerful formula in \cite{lupu2018random} (see (\ref{2.22})) to achieve the lower bound for the conditional connecting probability mentioned in Step (2). However, this method does not work well for proving the other direction of quasi-multiplicativity, mainly because the two parts of boundary conditions in Step (1) exhibit a strong long-range correlation.

For the upper bound on quasi-multiplicativity, in light of the aforementioned challenge, we instead study the question from the perspective of loop clusters. The inequality in the opposite direction of (\ref{newadd_1.29}) indicates that when $A_1$ and $A_2$ are connected by loop clusters, the following two events typically occur. 
\begin{enumerate}
	\item[(i)]   There are two loop clusters: one connecting $A_1$ and $\partial B(cN)$, and the other connecting $A_2$ and $\partial B(CN)$. Moreover, these two clusters   are composed of disjoint collections of loops, allowing the use of the van den Berg-Kesten-Reimer (BKR) inequality (see Lemma \ref{lemma_BKR}).

	\item[(ii)]   The probability of connecting these two (partial) loop clusters using the remaining loops (i.e., the loops not used in the construction of these two loop clusters) is at most of order $N^{0\boxdot (6-d)}$. 
	
\end{enumerate}
Due to significant differences, we divide the subsequent discussion into two cases: the low-dimensional case ($3\le d<6$) and the high-dimensional case ($d>6$).

\textbf{When $3\le d<6$.} We construct the loop clusters in Item (i) based on the consideration of crossing loops as follows. Precisely, if no loop crosses the annulus $B(N)\setminus B(cN)$, then $A_1$ is connected to $\partial B(cN)$ by loops contained in $\widetilde{B}(N)$ (otherwise, it would be impossible for $A_1\xleftrightarrow{} A_2$ to happen without such a crossing loop). Similarly, if no loop crosses $B(CN)\setminus B(N)$, then $A_2$ must be connected to $\partial B(CN)$ by loops within $[\widetilde{B}(N)]^c$. Combining these two observations with the thinning property of the loop soup (as a Poisson point process), we obtain an upper bound (which fits the desired bound, according to Corollary \ref{corollary_compare_boundtoset}) for the probability of $A_1\xleftrightarrow{} A_2$ restricted to the absence of these crossing loops. In the remaining case (i.e., when a large crossing loop exists), we employ a decomposition method for loops (see Section \ref{subsection_loopsoup}) to separate the influence of these crossing loops on the connecting events inside and outside $\partial B(N)$, and control the influence through stochastic domination (see Lemma \ref{lemma_sto_dom}). Meanwhile, noting that the event in Item (ii) almost surely occurs (by the trivial upper bound for a probability), we derive the desired upper bound.

    \textbf{P.S.} The reason why this coarse bound does not lead to any essential loss is that in low dimensions, intuitively, the two clusters in Item (i) typically have large capacities, resulting in a uniformly positive probability of being hit by the same macroscopic loop. Furthermore, assuming this intuition holds, we will immediately obtain the lower bound for quasi-multiplicativity (see (\ref{newadd_1.29})). However, this intuition is not yet technically clear due to the lack of quantitative estimates on boundary-to-set connecting probabilities.

\textbf{When $d>6$.} Similar to the low-dimensional case, we construct the two clusters in Item (i) by considering crossing loops in a larger scale (specifically, in the annuli $B(N^{\frac{d-4}{2}})\setminus B(N)$ and $B(N)\setminus B(N^{\frac{2}{d-4}})$; see Section \ref{section5.2}). To achieve the estimate in Item (ii) (which is no longer trivial, unlike in low dimensions), we employ the tree expansion argument (see e.g. \cite[Section 3.4]{cai2024high}) to control the probability for a single loop to intersect both these two clusters, which highly relies on the comparison between point-to-set connecting probabilities (see Propositions \ref{prop_pointtoset_speed} and \ref{prop_pointtoset_harnack}). Instead of detailing the tree expansion argument, we offer the following heuristic to illustrate how the correction factor $N^{6-d}$ emerges. Consider all macroscopic clusters (i.e., clusters of diameters at least of order $N$) within the annulus $\widetilde{B}(CN)\setminus \widetilde{B}(cN)$. In fact, the number of these clusters (denoted by $\mathfrak{X}$) is typically of order $N^{d-6}$. To see this, note that the probability for a lattice point to be contained in a macroscopic cluster is of order $N^{-2}$ (recalling that $\theta_d(N)\asymp N^{-2}$ for all $d>6$), and that the volume of each macroscopic cluster is typically of order $N^4$ (see \cite{cai2024high, werner2021clusters}). These two observations lead to 
\begin{equation}
	\mathfrak{X}\asymp \mathrm{vol}\big (\widetilde{B}(CN)\setminus \widetilde{B}(cN)\big) \cdot N^{-2}\cdot N^{-4} \asymp N^{d-6}. 
\end{equation}
Meanwhile, for each $i\in \{1,2\}$, we explore the loop cluster containing $A_i$ (i.e., the collection of points connected to $A_i$ by loops in $\widetilde{\mathcal{L}}_{1/2}$). When $\mathcal{C}_i$ reaches the annulus $\widetilde{B}(CN)\setminus \widetilde{B}(cN)$ (the probability for this event to occur for both $i\in \{1,2\}$ is of the same order as $\mathbb{P}^{D_1}\big( A_1 \xleftrightarrow{\ge 0} \partial B(N) \big) \mathbb{P}^{D_2}\big( A_2 \xleftrightarrow{\ge 0} \partial B(N) \big)$), we naturally assume that $\mathcal{C}_i$ uniformly selects one of these macroscopic clusters and we further assume (plausibly) that these two selections are approximately independent. Moreover, $A_1$ and $A_2$ are connected by loop clusters if and only if $\mathcal{C}_1$ and $\mathcal{C}_2$ select the same macroscopic cluster, which (by our assumptions on uniformity and independence) happens with probability $\mathfrak{X}^{-1}\asymp N^{6-d}$. Combining these observations, we heuristically derive quasi-multiplicativity with correction factor $N^{6-d}$.

All in all, throughout the analysis above, we heavily rely on comparing connecting probabilities both between point-to-set and boundary-to-set types, as well as within each type (even though we lack quantitative estimates for them). This is why the majority of this paper is dedicated to establishing Propositions \ref{prop_pointtoset_speed}--\ref{prop_relation_pointandboundary}.

\subsection{Organization of the paper}

In Section \ref{section_notation}, we collect some necessary notations
and review some useful results. Propositions \ref{prop_pointtoset_speed} and  \ref{prop_pointtoset_harnack} are established in Section \ref{section_property_pointtoset}. The proof of Proposition \ref{prop_boundtoset_zero} is presented in Section \ref{section_boundary_to_set_stable}. In Section \ref{section_relation}, we prove Proposition \ref{prop_relation_pointandboundary}. Subsequently, we establish Theorem \ref{prop_new_QM} in Section \ref{section_proof_theorem1}. Finally, in Section \ref{section_volume} we confirm Theorem \ref{thm_volume}.

\section{Preliminaries}\label{section_notation}

In this section, we review some preliminary notations and results.

\subsection{Basic notations}

\begin{itemize}


		\item  For any $N\ge 0$, recall that the box of the lattice $B(N)$ and the box of the metric graph $\widetilde{B}(N)$ have been defined in Section \ref{section_intro}. We also define the Euclidean ball of the lattice $\mathcal{B}(N):=\{y\in \mathbb{Z}^d:|y|\le N\}$. For any $x\in \mathbb{Z}^d$, let $B_x(N)$ (resp. $\widetilde{B}_x(N)$, $\mathcal{B}_x(N)$) be the set obtained by translating $B(N)$ (resp. $\widetilde{B}(N)$, $\mathcal{B}(N)$) by $x$. Note that $B_x(d^{-\frac{1}{2}}N) \subset  \mathcal{B}_x(N)\subset  B_x(N)$.

	\item  For any $v_1,v_2\in \widetilde{\mathbb{Z}}^d$, let $\|v_1-v_2\|$ be the graph distance between $v_1$ and $v_2$ on the metric graph $\widetilde{\mathbb{Z}}^d$. Especially, $\|x-y\|=d$ for all $\{x,y\}\in \mathbb{L}^d$.

	\item For any sets $V_1,V_2\subset \widetilde{\mathbb{Z}}^d$, the graph distance between them is 
	$$\mathrm{dist}(V_1,V_2):=\inf\nolimits_{v_1\in V_1,v_2\in V_2}\|v_1-v_2\|.$$

	\item For any $V\subset  \widetilde{\mathbb{Z}}^d$, the interior of $V$ is defined as  
	$$V^{\circ}:=\big\{v\in \widetilde{\mathbb{Z}}^d:\mathrm{dist}(\{v\}, \widetilde{\mathbb{Z}}^d \setminus V)>0 \big \}.$$
 The boundary of $V$ in $\widetilde{\mathbb{Z}}^d$ is defined as
 $$
 \widetilde{\partial} V:=\big\{v\in \widetilde{\mathbb{Z}}^d:  \mathrm{dist}(\{v\},  V)=\mathrm{dist}(\{v\}, \widetilde{\mathbb{Z}}^d \setminus V)=0 \big\}. 
 $$


	


	\item For any path $\widetilde{\eta}$ on $\widetilde{\mathbb{Z}}^d$ (defined as a continuous function $\widetilde{\eta}: [0,T]\to \widetilde{\mathbb{Z}}^d$ for some $T>0$), we denote its range (i.e.,  the collection of points in $\widetilde{\mathbb{Z}}^d$ visited by $\widetilde{\eta}$) by $\mathrm{ran}(\widetilde{\eta})$.

	\item A rooted loop $\widetilde{\varrho}$ is a path on $\widetilde{\mathbb{Z}}^d$ that starts and ends at the same point. This point is called the root of $\widetilde{\varrho}$.

	\item For any $x\in \mathbb{Z}^d$, we denote by $\widehat{\mathbb{P}}_x$ the law of the simple random walk $\{S_n\}_{n\ge 0}$ on $\mathbb{Z}^d$ starting from $x$. I.e., $S_0=x$ and for any $n\ge 1$, 
	 $$\widehat{\mathbb{P}}_x(S_n=z \mid S_{n-1}=y )= \tfrac{1}{2d}, \ \ \forall \{y,z\}\in \mathbb{L}^d.$$

\end{itemize}

\subsection{Brownian motion on $\widetilde{\mathbb{Z}}^d$}

The Markov process $\big\{\widetilde{S}_t\big\}_{t \ge 0}$, representing a Brownian motion on the metric graph $\widetilde{\mathbb{Z}}^d$, is defined as follows. Inside each interval $I_e$,  $\widetilde{S}_{\cdot}$ evolves as a standard one-dimensional Brownian motion. When the process reaches a lattice point $x \in \widetilde{\mathbb{Z}}^d$, it selects one of the neighboring intervals $\{I_{\{x,y\}}\}_{\{x,y\}\in \mathbb{L}^d}$ uniformly at random and performs as a Brownian excursion from $x$ along this interval. If the excursion arrives at a neighboring point $y$, the process resumes from $y$, continuing in the same manner. Referring to \cite[Section 2]{lupu2016loop}, the total local time accumulated at $x$ during all excursions until reaching the aforementioned neighboring point $y$ in this step is distributed as an independent exponential random variable with parameter $1$. Moreover, the projection of $\big\{\widetilde{S}_t\big\}_{t \ge 0}$ onto the lattice $\mathbb{Z}^d$ behaves exactly as a simple random walk on $\mathbb{Z}^d$. For any $v \in \widetilde{\mathbb{Z}}^d$, let $\widetilde{\mathbb{P}}_v$ denote the law of  $\big\{\widetilde{S}_t\big\}_{t \ge 0}$ starting from $v$, and let $\widetilde{\mathbb{E}}_v$ be the expectation under $\widetilde{\mathbb{P}}_v$.

    For any $t\ge 0$ and $v,w\in \widetilde{\mathbb{Z}}^d$, let $\widetilde{q}_t(v,w)$ denote the transition density of the Brownian motion $\big\{\widetilde{S}_t\big\}_{t \ge 0}$ from $x$ to $y$ at time $t$ with respect to the Lebesgue measure on $\widetilde{\mathbb{Z}}^d$. The transition density of the Brownian bridge with duration $t$, starting from $v$ and ending at $w$, is given by $\frac{\widetilde{q}_s(v,\cdot )\widetilde{q}_{t-s}(\cdot,w)}{\widetilde{q}_t(v,w)}$ for all $0\le s\le t$.

\textbf{Hitting times.} For any $D \subset \widetilde{\mathbb{Z}}^d$, we define $\tau_D := \inf\{t \ge 0 : \widetilde{S}_t \in D \}$ as the first time $\big\{\widetilde{S}_t\big\}_{t \ge 0}$ enters $D$ (with the convention $\inf\emptyset = +\infty$ for completeness). For simplicity, we denote $\tau_{\{v\}}$ by $\tau_v$ for all $v \in \widetilde{\mathbb{Z}}^d$. The following lemma shows that the hitting distribution remains stable under the addition of certain obstacles.

\begin{lemma}\label{lemma_new2.1}
For any $N\ge 1$, $D_1 \subset \widetilde{B}(N)$, $x\in \partial B(2N)$ and $D_2\subset [\widetilde{B}(3N)]^c$ and for each $j\in \{1,2\}$, 
   		\begin{equation} 
				\widetilde{\mathbb{P}}_{x}\big(\tau_{D_j}=\tau_{v}<\tau_{D_{3-j}} \big) \asymp \widetilde{\mathbb{P}}_{x}\big(\tau_{D_j}=\tau_{v}<\infty \big) ,\  \ \ \forall v  \in \widetilde{\partial }D_j.	
				\end{equation}
\end{lemma}
\begin{proof}
It is straightforward to see that
 \begin{equation}\label{ineq_Px_tau_D}
 	\widetilde{\mathbb{P}}_{x}\big(\tau_{D_j}=\tau_{v}<\tau_{D_{3-j}} \big) \le  \widetilde{\mathbb{P}}_{x}\big(\tau_{D_j}=\tau_{v}<\infty \big).
 \end{equation}
 Next, we aim to establish the inequality in the opposite direction with a constant factor. For any fixed $v\in \widetilde{\partial }D_j$, let $x_*$ be the point in $\partial B(2N)$ maximizing the probability $\widetilde{\mathbb{P}}_{\cdot }\big(\tau_{D_j}=\tau_{v}<\infty \big)$. Thus, by the strong Markov property, one has 
	\begin{equation*}
	\begin{split}
		&\widetilde{\mathbb{P}}_{x_*}\big(\tau_{D_j}=\tau_{v}<\infty \big)- \widetilde{\mathbb{P}}_{x_*}\big(\tau_{D_j}=\tau_{v}<\tau_{D_{3-j}} \big)\\
		=&  \sum_{y\in \widetilde{\partial}D_{3-j}} \widetilde{\mathbb{P}}_{x_*}\big(\tau_{D_{3-j}}=\tau_{y}<\tau_{D_j} \big) \sum_{z\in \partial B(2N)}\widetilde{\mathbb{P}}_{y}\big(\tau_{\partial B(2N)}=\tau_{z}<\infty \big)   \widetilde{\mathbb{P}}_{z}\big(\tau_{D_j}=\tau_{v}<\infty \big)  \\
		\le & \Big[ \widetilde{\mathbb{P}}_{x_*}\big(\tau_{D_{3-j}}<\infty \big) \land \max_{y\in \widetilde{\partial}D_{3-j}} \widetilde{\mathbb{P}}_{y}\big(\tau_{\partial B(2N)}<\infty \big) \Big]\cdot  \widetilde{\mathbb{P}}_{x_*}\big(\tau_{D_j}=\tau_{v}<\infty \big),
			\end{split}
	\end{equation*}
	where in the last line we used the maximality of $x_*$. Combined with that fact that $\widetilde{\mathbb{P}}_{x_*}\big(\tau_{D_{1}}<\infty \big)$ and $\max\nolimits_{y\in \widetilde{\partial}D_{2}} \widetilde{\mathbb{P}}_{y}\big(\tau_{\partial B(2N)}<\infty \big)$ are both uniformly bounded away from $1$ (by the invariance principle), it yields that 
	\begin{equation}\label{new2.7}
		\widetilde{\mathbb{P}}_{x_*}\big(\tau_{D_j}=\tau_{v}<\infty \big) \lesssim \widetilde{\mathbb{P}}_{x_*}\big(\tau_{D_j}=\tau_{v}<\tau_{D_{3-j}} \big).\end{equation}
For any $x\in \partial B(2N)$, by the strong Markov property, one has 
	\begin{equation*}
	\begin{split}
				\widetilde{\mathbb{P}}_{x}\big(\tau_{D_j}=\tau_{v}<\tau_{D_{3-j}} \big)\ge &\widetilde{\mathbb{P}}_{x}\big(\tau_{B_{x_*}(\frac{1}{10}|x_*|)}<\tau_{D_{3-j}} \big)\min_{w\in B_{x_*}(\frac{1}{10}|x_*|)}\widetilde{\mathbb{P}}_{w}\big(\tau_{D_j}=\tau_{v}<\tau_{D_{3-j}} \big)\\
		\gtrsim &  \widetilde{\mathbb{P}}_{x_*}\big(\tau_{D_j}=\tau_{v}<\tau_{D_{3-j}} \big),
	\end{split}
	\end{equation*}
	where in the last inequality we used the invariance principle and Harnack's inequality. Combined with (\ref{new2.7}) and the maximality of $x_*$, it implies that 
\begin{equation}\label{ineq_Px_tau_D2}
			\widetilde{\mathbb{P}}_{x}\big(\tau_{D_j}=\tau_{v}<\tau_{D_{3-j}} \big)\gtrsim     \widetilde{\mathbb{P}}_{x_*}\big(\tau_{D_j}=\tau_{v}<\infty \big)\ge \widetilde{\mathbb{P}}_{x}\big(\tau_{D_j}=\tau_{v}<\infty \big).
	\end{equation}
By (\ref{ineq_Px_tau_D}) and (\ref{ineq_Px_tau_D2}), we conclude the proof of this lemma.
	\end{proof}

\textbf{Green's function on $\widetilde{\mathbb{Z}}^d$.} For any $D\subset \widetilde{\mathbb{Z}}^d$, the Green's function for the set $D$ is defined as
\begin{equation}
	\widetilde{G}_D(v,w):=\int_{0}^{\infty} \Big\{\widetilde{q}_t(v,w) - \widetilde{\mathbb{E}}_v \big[ \widetilde{q}_{t-\tau_{D}}(\widetilde{S}_{\tau_D},w)\cdot \mathbbm{1}_{\tau_D<t}  \big]\Big\}dt,\ \ \forall v, w\in \widetilde{\mathbb{Z}}^d\setminus D.
	\end{equation}
Let $\widetilde{G}_D(v,w):=0$ when $v$ or $w$ is contained in $D$. As noted in \cite[Section 3]{lupu2016loop}, $\widetilde{G}_D(\cdot, \cdot)$ is finite, symmetric and continuous. Moreover, for any $D_1\subset D_2\subset \widetilde{\mathbb{Z}}^d$, 
\begin{equation}\label{green_D1_D2}
	\widetilde{G}_{D_1}(v,w)\ge \widetilde{G}_{D_2}(v,w), \ \ \forall v,w\in \widetilde{\mathbb{Z}}^d. 
\end{equation} 
In particular, when $D=\emptyset$, we write $\widetilde{G}_{\emptyset}(\cdot ,\cdot )$ as $\widetilde{G}(\cdot ,\cdot )$. When $v,w\in \mathbb{Z}^d$, $\widetilde{G}(v,w)$ equals $G(v ,w)$, the Green's function on $\mathbb{Z}^d$. This implies that 
\begin{equation}\label{bound_green}
	\widetilde{G}(v ,w) \asymp (\|v-w\|+1)^{2-d}, \ \ \forall v,w\in \widetilde{\mathbb{Z}}^d. 
\end{equation}
As shown in the following lemma, the Green's function also exhibits the stability presented in Lemma \ref{lemma_new2.1}.

\begin{lemma}\label{lemma_new2.2}
	For any $N\ge 1$, $D\subset \widetilde{B}(N)\cup [\widetilde{B}(4N)]^c$ and $v,w \in \widetilde{B}(3N)\setminus \widetilde{B}(2N)$,
	\begin{equation}
		\widetilde{G}_{D}(v,w) \asymp  (\|v-w\|+1)^{2-d}.
	\end{equation}
\end{lemma}
\begin{proof}
On the one hand, by the monotonicity of $\widetilde{G}_{\cdot }$ and (\ref{bound_green}), we have 
\begin{equation}
	\widetilde{G}_{D}(v,w) \le \widetilde{G}(v,w) \asymp    (\|v-w\|+1)^{2-d}.
\end{equation}

On the other hand, it follows from the definition of the Green's function that 
\begin{equation}\label{new2.13}
	\widetilde{G}_{D}(v,w) = \widetilde{\mathbb{P}}_v(\tau_{w}<\tau_{D})\widetilde{G}_{D}(w,w)\asymp \widetilde{\mathbb{P}}_v(\tau_{w}<\tau_{D}). 
\end{equation}
By the strong Markov property and the invariance principle, one has 
\begin{equation}\label{new2.14}
	\begin{split}
		\widetilde{\mathbb{P}}_v(\tau_{w}<\tau_{D}) \gtrsim   \min_{z\in \partial B_w(\frac{1}{10}|w-v|)} \widetilde{\mathbb{P}}_z(\tau_{w}<\tau_{D}).
			\end{split}
\end{equation}
	Moreover, for any $z\in \partial B_w(\frac{1}{10}|w-v|)$, by Lemma \ref{lemma_new2.1} we have 
	\begin{equation}\label{new2.15}
		\widetilde{\mathbb{P}}_z(\tau_{w}<\tau_{D}) \gtrsim  \widetilde{\mathbb{P}}_z(\tau_{w}<\infty) = \frac{\widetilde{G}(z,w)}{\widetilde{G}(w,w)} \overset{(\ref*{bound_green})}{\asymp } (\|v-w\|+1)^{2-d}.
	\end{equation}
	Combining (\ref{new2.13}), (\ref{new2.14}) and (\ref{new2.15}), we obtain 
	\begin{equation*}
		\widetilde{G}_{D}(v,w)  \gtrsim  (\|v-w\|+1)^{2-d} 
	\end{equation*}
	and thus complete the proof. 
	\end{proof}

\textbf{Boundary excursion kernel.} For any $D\subset \widetilde{\mathbb{Z}}^d$, the boundary excursion kernel for $D$ is defined as 
\begin{equation}\label{def_KD}
	\mathbb{K}_D(v,w) := \lim\limits_{\epsilon \downarrow 0} (2\epsilon)^{-1}\sum\nolimits_{v'\in \widetilde{\mathbb{Z}}^d: \|v'-v\|=\epsilon} \widetilde{\mathbb{P}}_{v'} (\tau_{D}=\tau_{w}<\infty), \ \ \forall v,w\in \widetilde{\partial} D.
\end{equation}
The time-reversal invariance of the Brownian motion implies that $\mathbb{K}_D(v,w)$ is symmetric, i.e., $\mathbb{K}_D(v,w)=\mathbb{K}_D(w,v)$ for all $v,w\in \widetilde{\partial} D$. Referring to \cite{lupu2018random}, $\mathbb{K}_D(v,w)$ can also be considered as the effective equivalent conductance between $v$ and $w$ in the metric graph $\widetilde{\mathbb{Z}}^d\setminus D^{\circ}$.

\textbf{Equilibrium measure and capacity.} For any $D\subset \widetilde{\mathbb{Z}}^d$ and $v\in \widetilde{\partial} D$, the equilibrium measure for $D$ at $v$, denoted by $\mathbb{Q}_D(v)$, can be regarded as the effective equivalent conductance between $v$ and infinity in $\widetilde{\mathbb{Z}}^d\setminus D^{\circ}$. Precisely, define
\begin{equation}
	\mathbb{Q}_D(v):= \lim\limits_{N \to \infty} \sum\nolimits_{w\in \partial B(N)}\mathbb{K}_{D\cup \partial B(N)}(v,w).
\end{equation}
The capacity of $D$ is the total mass of $\mathbb{Q}_D$, i.e., $\mathrm{cap}(D):= \sum_{v\in \widetilde{\partial} D}\mathbb{Q}_D(v)$, where the sum is well-defined because $\widetilde{\partial} D$ is countable. This countability follows from the facts that $D$ consists of finitely many connected components and that $\widetilde{\mathbb{Z}}^d$ is locally one-dimensional. Moreover, the order of the capacity of $\widetilde{B}(N)$ is given by 
\begin{equation}
	\mathrm{cap}\big(\widetilde{B}(N)\big)\asymp N^{d-2}, \ \ \forall N\ge 1. 
\end{equation}

The last-exit decomposition for the Brownian motion (see e.g. \cite[Section 8.2]{morters2010brownian}) implies that for any disjoint $D,D'\subset \widetilde{\mathbb{Z}}^d$ and $w\in \widetilde{\mathbb{Z}}^d\setminus (D\cup D')$, 
\begin{equation}\label{ineq_2.15}
	\widetilde{\mathbb{P}}_w\big(\tau_{D}<\tau_{D'}\big) = \sum\nolimits_{v\in  \widetilde{\partial} D} \widetilde{G}_{D'}(w,v) \mathbb{Q}_D(v). 
\end{equation}
This formula relates the hitting probability and the capacity as follows.

\begin{lemma}\label{lemma_prop_hitting}
For any $N\ge 1$, $D\subset \widetilde{B}(N)$, $x\in \partial B(2N)$ and $D'\subset [\widetilde{B}(3N)]^c$,
	\begin{equation}\label{new2.20}
		\widetilde{\mathbb{P}}_x(\tau_D<\tau_{D'})\asymp N^{2-d}\mathrm{cap}(D).
	\end{equation}
	Especially, when $D=\widetilde{B}(n)$ with $n\le N$, one has 
	\begin{equation}\label{new2.20_1}
		\widetilde{\mathbb{P}}_x(\tau_{\widetilde{B}(n)}<\tau_{D'})\asymp \big(\frac{n}{N}\big)^{d-2}.
	\end{equation}
\end{lemma}
\begin{proof}
	By the last-exit decomposition (\ref{ineq_2.15}), we have 
	\begin{equation}
	 \mathrm{cap}(D)\cdot \min\nolimits_{y\in D}\widetilde{G}_{D'}(x,y) \le 	\widetilde{\mathbb{P}}_x(\tau_D<\tau_{D'})\le  \mathrm{cap}(D)\cdot \max\nolimits_{y\in D} \widetilde{G}_{D'}(x,y).
	\end{equation}
	Combined with the fact that $\widetilde{G}_{D'}(x,y)\asymp N^{2-d}$ for all $y\in D$ (by Lemma \ref{lemma_new2.2}), this implies (\ref{new2.20}). Plugging $\mathrm{cap}\big(\widetilde{B}(n)\big)\asymp n^{d-2}$ into (\ref{new2.20}), we get (\ref{new2.20_1}).	
\end{proof}

The next lemma gives the decay rate of the hitting probability with respect to the distance between the starting point and the target set.

\begin{lemma}\label{lemma_compare_hm}
	For any $M\ge N\ge 1$, $v\in D\subset \widetilde{B}(\frac{N}{2})$, $x_1\in \partial B(N)$, $x_2\in \partial B(M)$ and $D'\subset [\widetilde{B}(2M)]^c$, 
	 \begin{equation}\label{ineq2.19}
	 N^{d-2}	\widetilde{\mathbb{P}}_{x_1}(\tau_{D}=\tau_{v}<\tau_{D'}) \asymp M^{d-2}	\widetilde{\mathbb{P}}_{x_2}(\tau_{D}=\tau_{v}<\tau_{D'}).
	 \end{equation}
\end{lemma}
\begin{proof}
	By the strong Markov property, one has
	\begin{equation}\label{ineq2.21}
	\begin{split}
				\widetilde{\mathbb{P}}_{x_2}(\tau_{D}=\tau_{v}<\tau_{D'}) =& \sum_{w\in \partial B(N)}\widetilde{\mathbb{P}}_{x_2}(\tau_{\partial B(N)}=\tau_{w}<\tau_{D'})\widetilde{\mathbb{P}}_{w}(\tau_{D}=\tau_{v}<\tau_{D'}).
	\end{split}
	\end{equation}
	Meanwhile, for any $w_1,w_2\in \partial B(N)$, by applying the strong Markov property, the invariance principle and Harnack's inequality in sequence, we have  
\begin{equation*}
\begin{split}
		\widetilde{\mathbb{P}}_{w_1}(\tau_{D}=\tau_{v}<\tau_{D'})	\ge & \widetilde{\mathbb{P}}_{w_1}(\tau_{B_{w_2}(\frac{1}{10}N)}<\tau_{D'})\min_{z\in B_{w_2}(\frac{1}{10}N)} \widetilde{\mathbb{P}}_{z}(\tau_{D}=\tau_{v}<\tau_{D'})\\
		\gtrsim &\min_{z\in B_{w_2}(\frac{1}{10}N)} \widetilde{\mathbb{P}}_{z}(\tau_{D}=\tau_{v}<\tau_{D'})
		\gtrsim    \widetilde{\mathbb{P}}_{w_2}(\tau_{D}=\tau_{v}<\tau_{D'}).
\end{split}
\end{equation*}
This implies that for any $w\in \partial B(N)$, 
\begin{equation}\label{ineq_new2.22}
	\widetilde{\mathbb{P}}_{w}(\tau_{D}=\tau_{v}<\tau_{D'}) \asymp \widetilde{\mathbb{P}}_{x_1}(\tau_{D}=\tau_{v}<\tau_{D'}). 
\end{equation}
Combined with (\ref{ineq2.21}), it shows that 
\begin{equation}\label{ineqnew2.22}
	\widetilde{\mathbb{P}}_{x_2}(\tau_{D}=\tau_{v}<\tau_{D'}) \asymp \widetilde{\mathbb{P}}_{x_2}(\tau_{\partial B(N)}<\tau_{D'})\widetilde{\mathbb{P}}_{x_1}(\tau_{D}=\tau_{v}<\tau_{D'}). 
\end{equation}
By (\ref{ineqnew2.22}) and $\widetilde{\mathbb{P}}_{x_2}(\tau_{\partial \mathcal{B}(N)}<\tau_{D'})  \asymp \big(\frac{N}{M}\big)^{d-2}$ (which follows from (\ref{new2.20_1})), we obtain the desired bound (\ref{ineq2.19}). 
\end{proof}


\subsection{Properties of the GFF}\label{subsection_property_GFF}

For any $D\subset \widetilde{\mathbb{Z}}^d$, recall that $\mathbb{P}^D$ denotes the law of $\{\widetilde{\phi}_v\}_{v\in \widetilde{\mathbb{Z}}^d}$ conditioned on the GFF values on $D$ being zero. The covariance of $ \widetilde{\phi}_\cdot $ under this conditioning satisfies (let $\mathbb{E}^D$ be the expectation under $\mathbb{P}^D$)
\begin{equation}
	\mathbb{E}^D\big[\widetilde{\phi}_{v}\widetilde{\phi}_{w}\big]= \widetilde{G}_{D}(v,w),\ \ \forall v,w\in \widetilde{\mathbb{Z}}^d.
\end{equation}

	\begin{lemma}\label{lemma_revise_B1}
		For any $x\in \mathbb{Z}^d$ and disjoint $A,D\subset \widetilde{\mathbb{Z}}^d\setminus \widetilde{B}_x(2)$,
		\begin{equation}
			\mathbb{P}^{D}\big(\widetilde{B}_x(1)\xleftrightarrow{\ge 0} A \big)\asymp \mathbb{P}^{D}\big(x\xleftrightarrow{\ge 0} A \big). 
		\end{equation}
	\end{lemma}
	\begin{proof}
For any $x'\in \mathbb{Z}^d$ adjacent to $x$, by the FKG inequality we have 
\begin{equation}\label{revise2.25}
\begin{split}
		\mathbb{P}^{D}\big(x \xleftrightarrow{\ge 0}A  \big) \ge &  \mathbb{P}^{D}\big(x' \xleftrightarrow{\ge 0}A, x' \xleftrightarrow{\ge 0}x \big) \\
		\overset{(\text{FKG})}{\ge } & \mathbb{P}^{D}\big(x' \xleftrightarrow{\ge 0}A \big) \cdot \mathbb{P}^{D}\big( x' \xleftrightarrow{\ge 0}x \big)
		  \gtrsim  \mathbb{P}^{D}\big(x' \xleftrightarrow{\ge 0}A  \big),
\end{split}
\end{equation}
where in the last inequality we need the condition $D\subset  \widetilde{\mathbb{Z}}^d\setminus \widetilde{B}_x(2)$ to ensure that $\mathbb{P}^{D}\big( x' \xleftrightarrow{\ge 0}x \big)\asymp 1$. By the union bound and (\ref{revise2.25}), we conclude this lemma:   
\begin{equation*}
	 \mathbb{P}^{D}\big(\widetilde{B}_x(1) \xleftrightarrow{\ge 0}A  \big)\le \sum\nolimits_{x'\in \mathbb{Z}^d:\{x,x'\}\in \mathbb{L}^d}\mathbb{P}^{D}\big(x' \xleftrightarrow{\ge 0}A  \big) \lesssim \mathbb{P}^{D}\big(x \xleftrightarrow{\ge 0}A  \big). \qedhere
\end{equation*}
	\end{proof}

\textbf{Harmonic average.} For any $D_1\subset D_2\subset \widetilde{\mathbb{Z}}^d$ and $v\in \widetilde{\mathbb{Z}}^d$, given all the GFF values on $D_2$, the harmonic average in $D_1$ for this boundary condition is 
	\begin{equation}\label{def_Hv}
		\mathcal{H}_v(D_1,D_2):=\left\{\begin{array}{ll}
			0   &\   \text{if}\ v\in D_2^{\circ}; \\
			\sum\nolimits_{w\in \widetilde{\partial} D_1} \widetilde{\mathbb{P}}_v\big(\tau_{D_2}=\tau_w<\infty \big) \widetilde{\phi}_{w}  &\  \text{otherwise}.
		\end{array}
		\right.
	\end{equation} 
	For brevity, when $D_1=D_2=D$, we denote $\mathcal{H}_v(D,D)$ by $\mathcal{H}_v(D)$.

\textbf{Strong Markov property of the GFF.} For any $D\subset \widetilde{\mathbb{Z}}^d$, assume that $\mathcal{A}$ is a random compact subset of $\widetilde{\mathbb{Z}}^d\setminus D^{\circ}$, which is measurable with respect to $\widetilde{\phi}_{\cdot } \sim \mathbb{P}^D$ and satisfies that for any compact set $V\subset \widetilde{\mathbb{Z}}^d\setminus D^{\circ}$, $\{\mathcal{A}\subset V\}$ is measurable with respect to $\mathcal{F}_V$, the $\sigma$-field generated by $\{\widetilde{\phi}_v\}_{v\in V}$. The strong Markov property of the GFF (see e.g. \cite[Theorem 8]{ding2020percolation}) states that conditioning on $\mathcal{F}_{\mathcal{A}}$ (i.e., given $\mathcal{A}$ and all the GFF values on $\mathcal{A}$), when $\mathcal{A}=D'$ happens for a certain $D'\subset \widetilde{\mathbb{Z}}^d$, one has the following equivalence between two distributions:
\begin{equation}
	\big\{\widetilde{\phi}_v\big\}_{v\in \widetilde{\mathbb{Z}}^d\setminus (D\cup D')} \overset{\mathrm{d}}{=}  \big\{\widetilde{\phi}_v'+\mathcal{H}_v(D',D\cup D') \big\}_{v\in \widetilde{\mathbb{Z}}^d\setminus (D\cup D')}.
\end{equation}
Here $\widetilde{\phi}'_{\cdot}$ is an independent GFF with law $\mathbb{P}^{D\cup D'}$.

\textbf{Negative cluster.} Parallel to $\widetilde{E}^{\ge 0}$, we also denote the negative level-set by $\widetilde{E}^{\le 0}:= \big\{v\in \widetilde{\mathbb{Z}}^d: \widetilde{\phi}_v\le 0\big\}$. For any $A_1,A_2\subset \widetilde{\mathbb{Z}}^d$, let $A_1 \xleftrightarrow{\le 0} A_2$ be the event that $A_1$ and $A_2$ are connected by $\widetilde{E}^{\le 0}$. For any non-empty $A\subset \widetilde{\mathbb{Z}}^d$, we define the negative cluster containing $A$ by $\mathcal{C}_{A}^{-}:= \big\{v\in \widetilde{\mathbb{Z}}^d: v\xleftrightarrow{\le 0} A  \big\}$. To have better control over the boundary values, we consider positive clusters restricted to single intervals and add those intersecting $A$ to $\mathcal C_A^-$, leading to the following definition:
\begin{equation}
	\widehat{\mathcal{C}}_{A}^{-}:= \mathcal{C}_{A}^{-}\cup \bigcup\nolimits_{v\in A\setminus \mathbb{Z}^d} \big\{w\in \widetilde{\mathbb{Z}}^d : \exists e\in \mathbb{L}^d\ \text{such that}\ v\xleftrightarrow{\widetilde{E}^{\ge 0}\cap I_e}  w \big\}.
\end{equation}
Clearly, one has $A\subset \widehat{\mathcal{C}}_{A}^{-}$ and $\widetilde{\phi}_v=0$ for all $v\in (\widetilde{\partial} \widehat{\mathcal{C}}_{A}^{-})\setminus \mathbb{Z}^d$. This further implies that $\mathcal{H}_v( \widehat{\mathcal{C}}_{A}^{-},\widehat{\mathcal{C}}_{A}^{-}\cup D)=\mathcal{H}_v(\mathbb{Z}^d\cap \widetilde{\partial} \widehat{\mathcal{C}}_{A}^{-},\widehat{\mathcal{C}}_{A}^{-}\cup D)$ for all $D\subset \widetilde{\mathbb{Z}}^d$ and $v\in \widetilde{\mathbb{Z}}^d\setminus D^{\circ}$.

\textbf{Connecting probability.} The following powerful formula was introduced in \cite[Equation (18)]{lupu2018random}. For any disjoint $D_1,D_2\subset \widetilde{\mathbb{Z}}^d$, given all the GFF values on $D_1\cup D_2$, if these values are all non-negative, then the conditional probability of $\{D_1\xleftrightarrow{\ge 0}D_2\}$ exactly equals
	\begin{equation}\label{2.22}
		1-e^{-2\sum_{z_1\in \widetilde{\partial }D_1,z_2\in \widetilde{\partial} D_2}\mathbb{K}_{D_1\cup D_2}(z_1,z_2)\widetilde{\phi}_{z_1}\widetilde{\phi}_{z_2}}. 
	\end{equation}
	In addition, if we assume that the GFF values on $[\widetilde{\partial}(D_1\cup D_2)]\setminus \mathbb{Z}^d$ all equal $0$, then it was shown in the proof of  \cite[Lemma 2.6]{cai2024one} that 
	\begin{equation}\label{new2.44}
	\sum\nolimits_{z_1\in \widetilde{\partial }D_1,z_2\in \widetilde{\partial} D_2}\mathbb{K}_{D_1\cup D_2}(z_1,z_2)\widetilde{\phi}_{z_1}\widetilde{\phi}_{z_2}=  \sum\nolimits_{y \in \widetilde{\partial } D_1\cap \mathbb{Z}^d} \widehat{\mathcal{H}}_{y}(D_2,D_1\cup D_2)\widetilde{\phi}_{y},
\end{equation}
	where $\widehat{\mathcal{H}}_y(D_2, D_1\cup D_2) := (2d)^{-1}\sum\nolimits_{z:\{y,z\}\in \mathbb{L}^d,I_{\{y,z\}}^{\circ}\cap (D_1\cup D_2)=\emptyset } \mathcal{H}_z(D_2, D_1\cup D_2)$. As a complement of \cite[Lemma 3.5]{cai2024one}, the following lemma provides the order of (\ref{new2.44}).

	\begin{lemma}\label{lemma_hat_H}
		We keep the notations above. If we further assume that $D_1\subset \mathbf{B}_1:=\widetilde{B}(N_1)$ and $D_2\subset \mathbf{B}_2:=[\widetilde{B}(N_2)]^c$ for some $N_1\ge 1$ and $N_2\ge 10d^2N_1$, then 
		\begin{equation}\label{newineq_2.32}
			\sum\nolimits_{y \in \widetilde{\partial } D_1\cap \mathbb{Z}^d} \widehat{\mathcal{H}}_{y}(D_2,D_1\cup D_2)\widetilde{\phi}_{y}  \asymp N_1^{d-2}\mathcal{H}_1\mathcal{H}_2,		\end{equation}
			where $\mathcal{H}_1:=\frac{\sum_{z\in \partial \mathcal{B}(dN_1)} \mathcal{H}_z(D_1,D_1\cup D_2)}{|\partial \mathcal{B}(dN_1)|} $ and $\mathcal{H}_2:=\frac{\sum_{z\in \partial \mathcal{B}(N_2/2)} \mathcal{H}_z(D_2,D_1\cup D_2)}{|\partial \mathcal{B}(N_2/2)|}$.
	\end{lemma}
	\begin{proof}
		For any $y\in \widetilde{\partial } D_1\cap \mathbb{Z}^d$, $y'\in \mathbb{Z}^d$ with $\{y,y'\}\in \mathbb{L}^d$ and $I^{\circ}_{\{y,y'\}}\cap (D_1\cup D_2)=\emptyset$, and for any $z\in \widetilde{\partial } D_2$, by the last-exit decomposition, $\widetilde{\mathbb{P}}_{y'}\big(\tau_{D_1\cup D_2}= \tau_{z} <\infty \big)$ equals 
		\begin{equation}\label{newineq_2.33}
			\begin{split}
				\sum_{z_1\in \partial \mathcal{B}(dN_1)} \widetilde{G}_{D_1\cup D_2}(y',z_1) \tfrac{1}{2d}\sum_{z_1'\in [\mathcal{B}(dN_1)]^c:\{z_1,z_1'\}\in \mathbb{L}^d} \widetilde{\mathbb{P}}_{z_1'}\big(\tau_{ D_2\cup \mathcal{B}(dN_1)}= \tau_{z} <\infty \big).
			\end{split}
		\end{equation}
		Moreover, by the strong Markov property one has 
		\begin{equation*}
			\begin{split}
				&\widetilde{\mathbb{P}}_{z_1'}\big(\tau_{ D_2\cup \mathcal{B}(dN_1)}= \tau_{z} <\infty \big) \\
					=& \sum\nolimits_{z_2\in \partial \mathcal{B}(2d^2N_1),z_3\in \partial \mathcal{B}(N_2/2) }\widetilde{\mathbb{P}}_{z_1'}\big(\tau_{   \partial \mathcal{B}(2d^2N_1)} =  \tau_{z_2} < \tau_{  \mathcal{B}(dN_1)} \big)   \\
					&\ \ \ \ \cdot \widetilde{\mathbb{P}}_{z_2}\big( \tau_{ \partial  \mathcal{B}(N_2/2)}= \tau_{z_3}  < \tau_{ \mathcal{B}(dN_1)} \big)  \widetilde{\mathbb{P}}_{z_3}\big(\tau_{ D_2}= \tau_{z} <\tau_{ \mathcal{B}(dN_1)} \big)\\
					\overset{\text{Lemma\ \ref*{lemma_new2.1}}}{\asymp} &\sum\nolimits_{z_2\in \partial \mathcal{B}(2d^2 N_1),z_3\in \partial \mathcal{B}(N_2/2) }\widetilde{\mathbb{P}}_{z_1'}\big(\tau_{   \partial \mathcal{B}(2d^2N_1)} =  \tau_{z_2} < \tau_{  \mathcal{B}(dN_1)} \big)    \\
					&\ \ \ \ \cdot \widetilde{\mathbb{P}}_{z_2}\big( \tau_{ \partial  \mathcal{B}(N_2/2)}= \tau_{z_3}  < \infty  \big)  \widetilde{\mathbb{P}}_{z_3}\big(\tau_{ D_2}= \tau_{z} <\tau_{D_1}  \big). 
								\end{split}
		\end{equation*}
		Combined with the facts that $\widetilde{\mathbb{P}}_{z_2}\big( \tau_{ \partial  \mathcal{B}(N_2/2)}= \tau_{z_3}  < \infty  \big) \asymp | \partial \mathcal{B}(\tfrac{N_2}{2}) |^{-1}$ (see e.g. \cite[Lemma 6.3.7]{lawler2010random}) and that $\widetilde{\mathbb{P}}_{z_1'}\big(\tau_{   \partial \mathcal{B}(2d^2N_1)} < \tau_{  \mathcal{B}(dN_1)} \big) \asymp N_1^{-1}$ (see e.g. \cite[Lemma 6.3.4]{lawler2010random}), it yields that 
		\begin{equation*}
			\begin{split}
				\widetilde{\mathbb{P}}_{z_1'}\big(\tau_{ D_2\cup \mathcal{B}(dN_1)}= \tau_{z} <\infty \big) \asymp N_1^{-1} | \partial \mathcal{B}(\tfrac{N_2}{2}) |^{-1} \sum\nolimits_{z_3\in \partial \mathcal{B}(N_2/2) } \widetilde{\mathbb{P}}_{z_3}\big(\tau_{ D_2}= \tau_{z} <\tau_{D_1}  \big). 
			\end{split}
		\end{equation*}
		This together with (\ref{newineq_2.33}) shows that 
		\begin{equation}\label{revise2.31}
			\begin{split}
				& \mathcal{H}_{y'}(D_2,D_1\cup D_2) \\
				=&\sum\nolimits_{z\in \widetilde{\partial } D_2 } \widetilde{\mathbb{P}}_{y'}\big(\tau_{D_1\cup D_2}= \tau_{z} <\infty \big) \widetilde{\phi}_z\\
				\asymp &  \sum_{z_1\in \partial \mathcal{B}(dN_1)} \widetilde{G}_{D_1\cup D_2}(y',z_1)  N_1^{-1} | \partial \mathcal{B}(\tfrac{N_2}{2}) |^{-1}  \sum_{z_3\in \partial \mathcal{B}(N_2/2) }  \mathcal{H}_{z_3}(D_2,D_1\cup D_2)   \\
				=  & N_1^{-1} \mathcal{H}_2  \sum\nolimits_{z_1\in \partial \mathcal{B}(dN_1)  } \widetilde{G}_{D_1\cup D_2}(y',z_1). 
			\end{split}
		\end{equation}

	Next, we estimate $\widetilde{G}_{D_1\cup D_2}(y',z_1)$. On the one hand, it has been proved in \cite[inequality below (3.25)]{cai2024one} that 
		\begin{equation}\label{newineq_2.36}
	\sum_{y':\{y,y'\}\in \mathbb{L}^d,I^{\circ}_{\{y,y'\}}\cap (D_1\cup D_2)=\emptyset} 		\widetilde{G}_{D_1\cup D_2}(y',z_1)\lesssim \widetilde{\mathbb{P}}_{z_1}\big(\tau_{D_1\cup D_2}=\tau_y<\infty \big). 
		\end{equation}
	On the other hand, by the symmetry of the Green's function, one has 
		\begin{equation}\label{newineq_2.37}
			\begin{split}
				\widetilde{G}_{D_1\cup D_2}(y',z_1) = \widetilde{G}_{D_1\cup D_2}(z_1,y')= \widetilde{\mathbb{P}}_{z_1}\big( \tau_{y'}<\tau_{D_1\cup D_2} \big) \widetilde{G}_{D_1\cup D_2}(y',y'). 
			\end{split}
		\end{equation}
		Let $\frac{1}{2}(y+y')$ denote the middle point of $I_{\{y,y'\}}$. By \cite[Equality (2.6)]{cai2024one} we have 
		\begin{equation}\label{newineq_2.38}
			\widetilde{G}_{D_1\cup D_2}(y',y') = \tfrac{\widetilde{\mathbb{P}}_{y'}(\tau_{\frac{1}{2}(y+y')}<\tau_{D_1\cup D_2}  ) \widetilde{G}_{D_1\cup D_2}(\frac{1}{2}(y+y'),\frac{1}{2}(y+y'))}{\widetilde{\mathbb{P}}_{\frac{1}{2}(y+y')}(\tau_{y'} <\tau_{D_1\cup D_2} )}.
		\end{equation}
		Combining (\ref{newineq_2.37}), (\ref{newineq_2.38}) and $\widetilde{G}_{D_1\cup D_2}(\frac{1}{2}(y+y'),\frac{1}{2}(y+y'))\asymp 1$ (which is ensured by $I^{\circ}_{\{y,y'\}}\cap (D_1\cup D_2)=\emptyset$), we get
		\begin{equation}\label{newineq_2.39}
			\widetilde{G}_{D_1\cup D_2}(y',z_1)  \gtrsim \widetilde{\mathbb{P}}_{z_1}\big( \tau_{y'}<\tau_{D_1\cup D_2} \big) \widetilde{\mathbb{P}}_{y'}(\tau_{\frac{1}{2}(y+y')}<\tau_{D_1\cup D_2}  ). 
		\end{equation}
		Moreover, for the Brownian motion started from $z_1$ to first hit $D_1\cup D_2$ at $y$, it must reach some $y'$ satisfying $\{y,y'\}\in \mathbb{L}^d$ and $I^{\circ}_{\{y,y'\}}\cap (D_1\cup D_2)=\emptyset$, and then hit $\frac{1}{2}(y+y')$. This together with the strong Markov property implies that 
	\begin{equation}\label{newineq_2.40} 
		\begin{split}
			&\widetilde{\mathbb{P}}_{z_1}\big(\tau_{D_1\cup D_2}=\tau_y<\infty \big) \\
			 \le & \sum_{y':\{y,y'\}\in \mathbb{L}^d,I^{\circ}_{\{y,y'\}}\cap (D_1\cup D_2)=\emptyset}\widetilde{\mathbb{P}}_{z_1}\big( \tau_{y'}<\tau_{D_1\cup D_2} \big) \widetilde{\mathbb{P}}_{y'}(\tau_{\frac{1}{2}(y+y')}<\tau_{D_1\cup D_2}  )\\
			 \overset{(\ref*{newineq_2.39} )}{\lesssim } & \sum_{y':\{y,y'\}\in \mathbb{L}^d,I^{\circ}_{\{y,y'\}}\cap (D_1\cup D_2)=\emptyset}\widetilde{G}_{D_1\cup D_2}(y',z_1). 
		\end{split}
	\end{equation}
	Putting (\ref{newineq_2.36}) and (\ref{newineq_2.40}) together, we obtain  	
	\begin{equation}\label{revise2.32}
		\sum_{y':\{y,y'\}\in \mathbb{L}^d,I^{\circ}_{\{y,y'\}}\cap (D_1\cup D_2)=\emptyset}\widetilde{G}_{D_1\cup D_2}(y',z_1) \asymp \widetilde{\mathbb{P}}_{z_1}\big(\tau_{D_1\cup D_2}=\tau_y<\infty \big). 
	\end{equation}

By (\ref{revise2.31}) and (\ref{revise2.32}), we conclude this lemma as follows:
\begin{align*}
	&\sum\nolimits_{y \in \widetilde{\partial } D_1\cap \mathbb{Z}^d} \widehat{\mathcal{H}}_{y}(D_2,D_1\cup D_2)\widetilde{\phi}_{y}\\
			\overset{(\ref*{revise2.31})}{\asymp} & \sum\nolimits_{y \in \widetilde{\partial } D_1\cap \mathbb{Z}^d} \widetilde{\phi}_{y} \sum\nolimits_{y':\{y,y'\}\in \mathbb{L}^d,I^{\circ}_{\{y,y'\}}\cap (D_1\cup D_2)=\emptyset}N_1^{-1} \mathcal{H}_2  \sum\nolimits_{z_1\in \partial \mathcal{B}(dN_1)  } \widetilde{G}_{D_1\cup D_2}(y',z_1)\\
			\overset{(\ref*{revise2.32})}{\asymp } &N_1^{-1} \mathcal{H}_2 \sum\nolimits_{z_1\in \partial \mathcal{B}(dN_1)  }  \sum\nolimits_{y \in \widetilde{\partial } D_1\cap \mathbb{Z}^d}   \widetilde{\mathbb{P}}_{z_1}\big(\tau_{D_1\cup D_2}=\tau_y<\infty \big)\widetilde{\phi}_{y}   \\
			\overset{(\ref*{def_Hv})}{= } & N_1^{-1} \mathcal{H}_2    \sum\nolimits_{z_1\in \partial \mathcal{B}(dN_1)  }  \mathcal{H}_y(D_1,D_1\cup D_2)  
			 \overset{|\partial \mathcal{B}(dN_1) |\asymp N_1^{d-2}}{\asymp   }     N_1^{d-2}\mathcal{H}_1\mathcal{H}_2.   \qedhere
\end{align*}
\end{proof}

%

Combining (\ref{2.22}), (\ref{new2.44}) and Lemma \ref{lemma_hat_H}, we obtain the following result. Note that in Corollary \ref{lemma_boxtobox_hm} and Lemma \ref{lemma_upper_boundarytoset}, $ \mathbf{B}_j$ and $\mathcal{H}_j$ are defined as in Lemma \ref{lemma_hat_H}.

	\begin{corollary}\label{lemma_boxtobox_hm}
		For any $d\ge 3$, there exist $C>c>0$ such that for any $N_1\ge 1$ and $N_2\ge 10dN_1$, $D_1\subset  \mathbf{B}_1$ and $D_2\subset  \mathbf{B}_2$, the following holds. Given all the GFF values on $\widetilde{\partial}(D_1\cup D_2)$, if they are all non-negative and equal to zero on $[\widetilde{\partial}(D_1\cup D_2)]\setminus \mathbb{Z}^d$, then the conditional probability of $\{D_1\xleftrightarrow{\ge 0}D_2\}$ is bounded from above and below by $1-e^{-CN_1^{d-2} \mathcal{H}_1 \mathcal{H}_2}$ and $1-e^{-cN_1^{d-2} \mathcal{H}_1\mathcal{H}_2}$ respectively.
	\end{corollary}

The subsequent lemma establishes a useful bound on the boundary-to-set connecting probability under certain positive boundary conditions. Compared to Corollary \ref{lemma_boxtobox_hm}, it imposes weaker requirements on the boundary conditions but only bounds the connecting probability from above.

	    \begin{lemma}\label{lemma_upper_boundarytoset}
	    	For any $j\in \{1,2\}$, we impose all conditions on $d,N_1,N_2,D_1,D_2$ in Corollary \ref{lemma_boxtobox_hm} except that we replace the condition $D_j \subset \mathbf{B}_j$ with $D_j \cap \mathbf{B}_j  \neq \emptyset$. Given all the GFF values on $\widetilde{\partial}(D_1\cup D_2)$, if they are all non-negative and equal to zero on $[D_j\setminus \mathbf{B}_j]\cup D_{3-j}$, then the conditional probability of $\big\{ D_j \xleftrightarrow{\ge 0} \partial B(N_{3-j})\big\} $ is at most of order $N_j^{d-2}\theta_d(N_{3-j}/4)\mathcal{H}_j$. 
	    \end{lemma}
	\begin{proof}
	The proofs for $j\in \{1,2\}$ are similar, so we only provide the proof details for the case $j=1$. By  (\ref{2.22}), (\ref{new2.44}) and \cite[Lemma 3.5]{cai2024one}, the conditional probability of $\big\{ D_1 \xleftrightarrow{\ge 0} \partial B(N_{2})\big\}$ is bounded from above by 
	\begin{equation}\label{newfinish2.43}
	\begin{split}
			&C\mathbb{E}\big[ \big( N_1^{d-2} \mathcal{H}_1\widehat{\mathcal{H}}_2\big) \land 1 \mid \mathcal{F}_{D_1\cup D_2}\big]\\
			\overset{\text{(Jensen's Ineq)}}{\le}  & CN_1^{d-2} \mathcal{H}_1 \cdot \big( \mathbb{E} \big[ \widehat{\mathcal{H}}_2 \mid \mathcal{F}_{D_1\cup D_2} \big] \land (N_1^{2-d} \mathcal{H}_1^{-1})  \big).
	\end{split}
	\end{equation}
	Here $\widehat{\mathcal{H}}_2$ is defined as $\widehat{\mathcal{H}}_2:=|\partial \mathcal{B}(N_2/2)|^{-1}\sum_{z\in \partial \mathcal{B}(N_2/2)} \mathcal{H}_z(\mathcal{C}_{\partial B(N_2)}^{-},D_1\cup \mathcal{C}_{\partial B(N_2)}^{-})$. Referring to 	\cite[Lemma 3.7]{cai2024one} (here we only need to replace $\mathcal{C}_{B(n)}^{-}$ in \cite[Lemma 3.7]{cai2024one} by $D_1$, which does not cause any difference on the proof), we have 
	\begin{equation}\label{newfinish2.44}
		\mathbb{E} \big[ \widehat{\mathcal{H}}_2 \mid \mathcal{F}_{D_1\cup D_2} \big] \lesssim \theta_d(N_2/4)+ (N_1/N_2)^{d-2} \mathcal{H}_1. 
	\end{equation}
	Similar to \cite[Inequality (3.42)]{cai2024one}, by (\ref{newfinish2.43}) and the inequality that $[(a_1+a_2)b]\land 1 \le a_1b+ \sqrt{a_2b}$ for $a_1,a_2,b>0$, the right-hand side of (\ref{newfinish2.43}) is at most
	\begin{equation}\label{finish2.45}
	\begin{split}
			& CN_1^{d-2}\mathcal{H}_1 \cdot \big[   \big( \theta_d(N_2/4)+ (N_1/N_2)^{d-2} \mathcal{H}_1\big) \land \big( N_1^{2-d} \mathcal{H}_1^{-1} \big)   \big] \\
			\lesssim & N_1^{d-2} \mathcal{H}_1 \cdot  \big[ \theta_d(N_2/4)+ N_2^{-\frac{d}{2}+1} \big]  \overset{(\ref*{one_arm_low})-(\ref*{one_arm_high})}{\lesssim } N_1^{d-2} \theta_d(N_2/4) \mathcal{H}_1 .
	\end{split}
		\end{equation}
	Combined with (\ref{newfinish2.43}), it implies this lemma for $j=1$. 
		\end{proof}

\subsection{Loop soup and isomorphism theorem}\label{subsection_loopsoup}
The loop measure $\widetilde{\mu}$ is defined as 
\begin{equation}\label{def_mu}
	 \widetilde{\mu}(\cdot) := \int_{v\in \widetilde{\mathbb{Z}}^d} \mathrm{d}m(v) \int_{0< t< \infty} t^{-1} \widetilde{q}_t(v,v)\widetilde{\mathbb{P}}^t_{v,v}(\cdot) \mathrm{d}t, 
\end{equation}
where $m(\cdot)$ is the Lebesgue measure on $\widetilde{\mathbb{Z}}^d$ and $\widetilde{\mathbb{P}}^t_{v,v}(\cdot)$ is the law of the Brownian bridge on $\widetilde{\mathbb{Z}}^d$ with duration $t$, starting and ending at $v$. The loop soup of intensity $\alpha > 0$, denoted by $\widetilde{\mathcal{L}}_{\alpha}$, is the Poisson point process with intensity measure $\alpha \widetilde{\mu}$. In the remainder of this paper, by forgetting the roots of all rooted loops, we consider the loop measure $\widetilde{\mu}$ in (\ref{def_mu}) as a measure on the space of equivalence classes of rooted loops, where every pair of loops in the same class are equivalent under a certain time-shift. We refer to each equivalence class as a loop.

For any loop $\widetilde{\ell}$, let its range, denoted by $\mathrm{ran}(\widetilde{\ell})$, be the range of any rooted loop in $\widetilde{\ell}$. For any $D\subset \widetilde{\mathbb{Z}}^d$, we denote $\widetilde{\mu}^{ D}:=\widetilde{\mu} \circ \mathbbm{1}_{\mathrm{ran}(\widetilde{\ell}) \subset \widetilde{\mathbb{Z}}^d\setminus D}$. For any $\alpha>0$, let $\widetilde{\mathcal{L}}_\alpha^{D}$ be the Poisson point process with intensity measure $\widetilde{\mu}^{D}$, i.e., the point process consisting of loops in $\widetilde{\mathcal{L}}_{\alpha}$ that do not intersect $D$.

The following estimate for the measure of loops that intersect three given unit boxes will be used multiple times. For convenience, we set $0^{-a}:=1$ for $a>0$.
\begin{lemma}[{\cite[Lemma 2.3]{cai2024high}}]\label{lemma_revise_2.3}
For any $d\ge 3$, there exists $C>0$ such that for any $x_1,x_2,x_3\in \mathbb{Z}^d$ with $\max\{|x_1-x_2|,|x_2-x_3|,|x_3-x_1|\}\ge C$,
\begin{equation*}
	\widetilde{\mu}\big(\{\widetilde{\ell}:\forall 1\le i\le 3, \widetilde{B}_{x_i}(1)\cap  \mathrm{ran}(\widetilde{\ell})\neq \emptyset \} \big) \lesssim  |x_1-x_2|^{2-d}|x_2-x_3|^{2-d}|x_3-x_1|^{2-d}. 
\end{equation*}
\end{lemma}

\textbf{Isomorphism theorem.} We next review Lupu's powerful and elegant coupling between the loop soup and the GFF, which enriches both models and allows one to derive properties for one model from the other and vice versa. Referring to \cite[Proposition 2.1]{lupu2016loop}, for any $D\subset \widetilde{\mathbb{Z}}^{d}$ there is a coupling between $\widetilde{\mathcal{L}}_{1/2}^{D}$ and $\{\widetilde{\phi}_v\}_{v\in \widetilde{\mathbb{Z}}^d\setminus D}\sim  \mathbb{P}^{D}$ such that (let $\widehat{\mathcal{L}}^{D,v}_{1/2}$ be the total local time at $v$ of loops in $\widetilde{\mathcal{L}}_{1/2}^{D}$)
	\begin{itemize}
		\item  $\widehat{\mathcal{L}}^{D,v}_{1/2}=\frac{1}{2}\widetilde{\phi}_v^{2}$ for all $v\in \widetilde{\mathbb{Z}}^d\setminus D$;

		\item  the sign clusters of $\widetilde{\phi}_\cdot$ are exactly the loop clusters of $\widetilde{\mathcal{L}}_{1/2}^D$.

	\end{itemize}
As a corollary, all estimates on the connecting probabilities of $\widetilde{E}^{\ge 0}$ mentioned in Section \ref{section_intro} also apply to their analogues in loop clusters.

We denote the union of ranges of loops in a point measure $\widetilde{\mathcal{L}}$ by $\cup \widetilde{\mathcal{L}}$. For simplicity, we write ``$\xleftrightarrow{\cup \widetilde{\mathcal{L}}_{1/2}^{D}}$'' as ``$\xleftrightarrow{(D)}$''. When $D=\emptyset$, we abbreviate ``$\xleftrightarrow{\cup \widetilde{\mathcal{L}}_{1/2}}$'' as ``$\xleftrightarrow{}$''. For any $D,A_1,A_2\subset \widetilde{\mathbb{Z}}^d$, by the symmetry of $\widetilde{\phi}_{\cdot}$, we have 
\begin{equation}\label{coro2.1_1}
		\mathbb{P}^{D}\big(A_1 \xleftrightarrow{\ge 0} A_2 \big)=  \mathbb{P}^{D}\big(A_1 \xleftrightarrow{\le 0} A_2 \big).
	\end{equation}
Moreover, under the coupling given by the isomorphism theorem, one has
\begin{equation}
\big\{A_1 \xleftrightarrow{\ge 0} A_2\big\} \subset 	\big\{A_1 \xleftrightarrow{(D)} A_2 \big\}\subset \big\{A_1 \xleftrightarrow{\ge 0} A_2\big\}\cup \big\{A_1 \xleftrightarrow{\le 0} A_2 \big\}. 
\end{equation}
Combined with (\ref{coro2.1_1}), it yields that 
\begin{equation}\label{newfinal_compare}
	\mathbb{P}^{D}\big(A_1 \xleftrightarrow{\ge 0} A_2 \big) \le  \mathbb{P}^{D}\big(A_1 \xleftrightarrow{(D)} A_2 \big)\le 2\mathbb{P}^{D}\big(A_1 \xleftrightarrow{\ge 0} A_2 \big). 
\end{equation}
Meanwhile, since $\{A_1 \xleftrightarrow{(D)} A_2\}$ is decreasing in $D$, one has 
\begin{equation}\label{newfinal_monotone}
\mathbb{P}\big(A_1 \xleftrightarrow{(D_2)} A_2 \big)	\le \mathbb{P}\big(A_1 \xleftrightarrow{(D_1)} A_2 \big), \ \ \forall D_1\subset D_2\subset \widetilde{\mathbb{Z}}^d.
\end{equation}
Putting (\ref{newfinal_compare}) and (\ref{newfinal_monotone}) together, we have 
\begin{equation}\label{ineq_compare}
\mathbb{P}^{D_2}\big(A_1 \xleftrightarrow{\ge 0} A_2 \big)	\le 2\mathbb{P}^{D_1}\big(A_1 \xleftrightarrow{\ge 0} A_2 \big), \ \ \forall D_1\subset D_2\subset \widetilde{\mathbb{Z}}^d.
\end{equation}



\textbf{Loop decomposition.} Next, we review a decomposition for loops. For any disjoint $A_1,A_2\subset \mathbb{Z}^d$ and any loop $\widetilde{\ell}$ intersecting both $A_1$ and $A_2$, let $\widetilde{\varrho}:[0,L]\to \widetilde{\mathbb{Z}}^d$ be an arbitrarily rooted loop in $\widetilde{\ell}$ satisfying the following conditions: (i) $\widetilde{\varrho}(0)\in A_1$; (ii) there exists $t\in [0,L]$ such that $\widetilde{\varrho}(t)\in A_2$ and $\widetilde{\varrho}(t')\notin A_1\cup A_2$ for all $t'\in (t,L)$. For such $\widetilde{\varrho}$, we first set $\tau_0:=0 $ and then recursively define a sequence of stopping times as follows: for each $k\in \mathbb{N}$, we define that $\tau_{2k+1}:= \inf\{t>\tau_{2k}: \widetilde{\varrho}(t)\in A_2\}$ and  $\tau_{2k+2}:= \inf\{t>\tau_{2k+1}: \widetilde{\varrho}(t)\in A_1\}$. Let $\kappa(\widetilde{\varrho};A_1,A_2)$ be the unique integer such that $\tau_{2\kappa}=L$. A forward (resp. backward) crossing path of $\widetilde{\ell}$ from $A_1$ to $A_2$, denoted by $\widetilde{\eta}^{\mathrm{F}}$ (resp. $\widetilde{\eta}^{\mathrm{B}}$), is a sub-path of $\widetilde{\varrho}\in \widetilde{\ell}$ satisfying that $\widetilde{\eta}^{\mathrm{F}}(s)=\widetilde{\varrho}(\tau_{2i-2}+s)$, $\forall 0\le s\le \tau_{2i-1}-\tau_{2i-2}$ (resp. $\widetilde{\eta}^{\mathrm{B}}(s)=\widetilde{\varrho}(\tau_{2i-1}+s)$, $\forall 0\le s\le \tau_{2i}-\tau_{2i-1}$) for some $1\le i\le \kappa(\widetilde{\varrho};A_1,A_2)$. Note that a backward crossing path from $A_1$ to $A_2$ actually starts from $A_2$ and ends at $A_1$. In addition, $\kappa(\widetilde{\varrho};A_1,A_2)$ and the collection of forward and backward crossing paths do not change with the selection of $\widetilde{\varrho}$. In light of this, we denote $\kappa(\widetilde{\varrho};A_1,A_2)$ by $\kappa(\widetilde{\ell};A_1,A_2)$. Referring to \cite[Lemma 2.7]{chang2016phase} and \cite[Corollary 2.2]{cai2024high}, we have the following estimate for $\kappa$.

\begin{lemma}\label{lemma_loop_measure_crossing}
	For $d\ge 3$, there exist $C,c>0$ such that for $N\ge 1$ and $A\subset \widetilde{B}(cN)$, 
	\begin{equation}
			\widetilde{\mu}\big(\{\widetilde{\ell}: \kappa(\widetilde{\ell}; \partial B(N), A)\ge 1 \} \big) \gtrsim N^{2-d} \mathrm{cap}(A), 
	\end{equation}
	\begin{equation}\label{newineq_2.50}
		\widetilde{\mu}\big(\{\widetilde{\ell}: \kappa(\widetilde{\ell}; \partial B(N), A)\ge j \} \big)\le \big[ CN^{2-d} \mathrm{cap}(A)\big]^{j}, \ \ \forall j\in \mathbb{N}^+.
	\end{equation}
\end{lemma}

For any integer $l \ge 0 $, we denote by 
\begin{equation}\label{def_Gl_crossing}
	\mathsf{G}_l^{D}(A_1,A_2):= \Big\{ \sum\nolimits_{\widetilde{\ell}\in \widetilde{\mathcal{L}}^{D}_{1/2}} \kappa(\widetilde{\ell};A_1,A_2) = l \Big\}
\end{equation}
the event that there are exactly $l$ crossings for loops in $\widetilde{\mathcal{L}}_{1/2}^{D}$. We also denote 
\begin{equation}
	\mathsf{G}_{\ge l}^{D}(A_1,A_2):= \cup_{m\ge l } \mathsf{G}_m^{D}(A_1,A_2). 
\end{equation}

Based on Lemma \ref{lemma_loop_measure_crossing}, we derive the decay rate of $\mathsf{G}_{\ge l}^{D}$ as follows.

\begin{lemma}\label{lemma_Gl}
	For any $d\ge 3$, there exist $C,c>0$ such that for any $D\subset \widetilde{\mathbb{Z}}^d$, $N\ge 1$, $A\subset \widetilde{B}(cN)$ and $l\ge 1$, 	\begin{equation}\label{final2.55}
		\mathbb{P}\big(\mathsf{G}_{\ge l}^{D}(\partial B(N),A)   \big) \le \big[CN^{2-d}\mathrm{cap}(A)\big]^{l}. 
	\end{equation}
\end{lemma}
\begin{proof}
Let $\mathfrak{N}$ be the number of loops $\widetilde{\ell}\in \widetilde{\mathcal{L}}_{1/2}^{D}$ with $\kappa(\widetilde{\ell};\partial B(N),A)\ge 1$. By Lemma \ref{lemma_loop_measure_crossing}, given a loop $\widetilde{\ell}\in \widetilde{\mathcal{L}}_{1/2}^{D}$ satisfying $\kappa(\widetilde{\ell};\partial B(N),A)\ge 1$, we have that $\kappa(\widetilde{\ell};\partial B(N),A)$ is stochastically dominated by a geometric random variable with success probability $CN^{2-d}\mathrm{cap}(A)$ (which is much smaller than $1$ due to the condition that $A \subset \widetilde{B}(cN)$ for some sufficiently small $c>0$). Moreover, for any $k\ge 1$, conditioning on $\{\mathfrak{N} =k \}$, the values of $\kappa(\cdot  \ ;\partial B(N),A)$ for these $k$ loops are independent. Therefore, given $\{\mathfrak{N} =k \}$, $\sum\nolimits_{\widetilde{\ell}\in \widetilde{\mathcal{L}}^{D}_{1/2}} \kappa(\widetilde{\ell};A_1,A_2) $ is stochastically dominated by a negative binomial random variable for $k$ successes with success probability $CN^{2-d}\mathrm{cap}(A)$. Thus, for any $k<l$, we have 
\begin{equation}\label{newineq_2.53}
	\begin{split}
		&\mathbb{P}\big( \mathsf{G}_{\ge l}^{D}(\partial B(N),A)  \mid \mathfrak{N} =k \big)\\
		\le &  \sum\nolimits_{m\ge l} \binom{m-1}{k-1}\big[CN^{2-d}\mathrm{cap}(A)\big]^{m-k}[1-CN^{2-d}\mathrm{cap}(A)]^{k}\\
		\le & \sum\nolimits_{m\ge l} m^k\big[CN^{2-d}\mathrm{cap}(A)\big]^{m-k}. 
	\end{split}
\end{equation}
Meanwhile, by (\ref{newineq_2.50}) (with $j=1$) one has 
\begin{equation}\label{newineq_2.54}
	\begin{split}
		\mathbb{P}(\mathfrak{N} =k)\le (k!)^{-1}\big[CN^{2-d}\mathrm{cap}(A)\big]^k. 
	\end{split}
\end{equation}
By (\ref{newineq_2.53}) and (\ref{newineq_2.54}), we conclude this lemma as follows:
\begin{equation}
	\begin{split}
		&\mathbb{P}\big( \mathsf{G}_{\ge l}^{D}(\partial B(N),A)    \big)\\
		\le & \mathbb{P}(\mathfrak{N} \ge l)+ \sum\nolimits_{1\le k<l} \mathbb{P}(\mathfrak{N}=k)\cdot \mathbb{P}( \mathsf{G}_{\ge l}^{D}(\partial B(N),A)   \mid \mathfrak{N}=k) \\
		\lesssim  &\mathbb{P}(\mathfrak{N} \ge l)+ \sum\nolimits_{1\le k<l} (k!)^{-1}\big[CN^{2-d}\mathrm{cap}(A)\big]^k\sum\nolimits_{m\ge l} m^k\big[CN^{2-d}\mathrm{cap}(A)\big]^{m-k}\\
		\lesssim  &\big[CN^{2-d}\mathrm{cap}(A)\big]^{l}+ \sum\nolimits_{m\ge l} \big[CN^{2-d}\mathrm{cap}(A)\big]^{m} e^{m}
		\lesssim \big[CeN^{2-d}\mathrm{cap}(A)\big]^{l}. \qedhere
					\end{split}
\end{equation}

\end{proof}

\textbf{Properties of the crossing paths.} On the event $\mathsf{G}_l^{D}(A_1,A_2)$, we enumerate all forward and backward crossing paths of loops in $\widetilde{\mathcal{L}}_{1/2}^{D}$ from $A_1$ to $A_2$ by $\{\widetilde{\eta}_i^{\mathrm{F}}\}_{i=1}^{l}$ and $\{\widetilde{\eta}_i^{\mathrm{B}}\}_{i=1}^{l}$. For each $1\le i\le l$, we denote by $\mathbf{x}_i$ (resp. $\mathbf{y}_i$) the starting (resp. ending) point of $\widetilde{\eta}_i^{\mathrm{F}}$. Note that $\{ \mathbf{x}_i \}_{i=1}^{l}  \subset  A_1$ and $\{ \mathbf{y}_i \}_{i=1}^{l}  \subset A_2$. Let $\Omega_l(A_1,A_2)$ denote the collection of all possible configurations of $\{(\mathbf{x}_i,\mathbf{y}_i)\}_{i=1}^{l}$. For any $\omega_l=\{(x_i,y_i)\}_{i=1}^{l} \in 	\Omega_l(A_1,A_2)$, we denote the event 
\begin{equation}\label{3.36}
	   	\widehat{\mathsf{G}}^{D}(\omega_l;A_1,A_2):= 	 \mathsf{G}^{D}_l( A_1,A_2)\cap \{\mathbf{x}_i=x_i,\mathbf{y}_i=y_i,\forall 1\le i\le l \}.
	   \end{equation}

Thanks to the spatial Markov property of the loop soup (which is thoroughly discussed in \cite{werner2016spatial}; see also \cite{cai2024high,chang2016phase}), the distribution of forward and backward crossing paths can be described as follows.

\begin{lemma}\label{lemma_decomposition}
	For any $d\ge 3$, $D\subset \widetilde{\mathbb{Z}}^d$, disjoint $A_1,A_2\subset \mathbb{Z}^d\setminus D$, $l\ge 1$ and $\omega_l\in 	\Omega_l(A_1,A_2)$, conditioning on the event $\widehat{\mathsf{G}}^{D}(\omega_l;A_1,A_2)$, the following holds: $\{\widetilde{\eta}_i^{\mathrm{F}}\}_{i=1}^{l}$ and $\{\widetilde{\eta}_i^{\mathrm{B}}\}_{i=1}^{l}$ are independent and have distributions $\widetilde{\mathbb{P}}_{x_i}\big(\{\widetilde{S}_t\}_{0\le t\le \tau_{A_2}}\in \cdot \mid \tau_{A_2}=\tau_{y_i}<\tau_{D}\big)$ and $\widetilde{\mathbb{P}}_{y_i}\big(\{\widetilde{S}_t\}_{0\le t\le \tau_{A_1}}\in \cdot \mid \tau_{A_1}=\tau_{x_{i+1}}<\tau_{D}\big)$ for $1\le i\le l$ (where we denote $x_{l+1}:=x_1$).
\end{lemma}

Next, we present some estimates on hitting probabilities for crossing paths.

\begin{lemma}\label{lemma_hitting_crossing_path}
	For any $d\ge 3$ and $N_2\ge N_1\ge 1$, if $\widetilde{\eta}^{\mathrm{F}}$ (resp. $\widetilde{\eta}^{\mathrm{B}}$) is a forward (resp. backward) crossing path of a loop in $\widetilde{\mathcal{L}}_{1/2}^{D}$ from $\partial B(N_2)$ to $\partial B(N_1)$, then the following two items hold. 
	\begin{enumerate}

		\item  For any $w_1,w_2\in [B(10N_2)]^c$ and $D\subset \widetilde{B}(N_1)\cup [\widetilde{B}(2(|w_1|\vee |w_2|))]^c$,
		\begin{equation}\label{newineq_2.69}
			\mathbb{P}\big(\mathrm{ran}(\widetilde{\eta}^{\mathrm{F}})\cap \widetilde{B}_{w_i}(1)\neq \emptyset,\forall i=1,2 \big)\lesssim |w_1|^{2-d}|w_1-w_2|^{2-d}|w_2|^{2-d}N_2^{d-2}. 
		\end{equation}

		\item For any $w_1,w_2\in B(\frac{N_1}{10})$ and $D\subset \widetilde{B}(\frac{1}{2d}(|w_1|\land |w_2|))\cup [\widetilde{B}(N_2)]^c$,  
		\begin{equation}\label{newnewineq_2.73}
			\mathbb{P}\big(\mathrm{ran}(\widetilde{\eta}^{\mathrm{B}})\cap \widetilde{B}_{w_i}(1)\neq \emptyset,\forall i=1,2 \big)\lesssim N_1^{2-d}|w_1-w_2|^{2-d}. 		\end{equation}

	\end{enumerate}	
\end{lemma}
\begin{proof}
 For (\ref{newineq_2.69}), assume that $\widetilde{\eta}^{\mathrm{F}}$ starts from $x\in \partial B(N_2)$ and ends at $y\in \partial B(N_1)$. Let $\mathsf{A}_1$ (resp. $\mathsf{A}_2$) be the event that $\widetilde{\eta}^{\mathrm{F}}$ visits $\widetilde{B}_{w_1}(1)$ and $\widetilde{B}_{w_2}(1)$ (resp. $\widetilde{B}_{w_2}(1)$ and $\widetilde{B}_{w_1}(1)$) in order. Note that $\{\mathrm{ran}(\widetilde{\eta}^{\mathrm{F}})\cap \widetilde{B}_{w_i}(1)\neq \emptyset,\forall i=1,2\}\subset \mathsf{A}_1\cup \mathsf{A}_2$. By the strong Markov property, we have 
	\begin{equation}\label{newineq_2.73}
		\begin{split}
			\mathbb{P}\big(\mathsf{A}_1 \big)\lesssim  &  \widetilde{\mathbb{P}}_{x}\big(\tau_{\widetilde{B}_{w_1}(1)}<\infty \big) \sum_{w_1'\in \mathbb{Z}^d: \{w_1,w_1'\}
			\in \mathbb{L}^d} \widetilde{\mathbb{P}}_{w_1'}\big(\tau_{\widetilde{B}_{w_2}(1)}<\infty \big)\\
			&\cdot \sum_{w_2'\in \mathbb{Z}^d: \{w_2,w_2'\}
			\in \mathbb{L}^d}  \frac{\widetilde{\mathbb{P}}_{w_2'}\big(\tau_{\partial B(N_1)}=\tau_{y}<\tau_{D}  \big)}{\widetilde{\mathbb{P}}_{x}\big(\tau_{\partial B(N_1)}=\tau_{y}<\tau_{D}  \big) }.
		\end{split}
	\end{equation}
	Note that $\widetilde{\mathbb{P}}_{x}\big(\tau_{\widetilde{B}_{w_1}(1)}<\infty \big)\asymp |x-w_1|^{2-d}\asymp |w_1|^{2-d}$ and that $\widetilde{\mathbb{P}}_{w_1'}\big(\tau_{\widetilde{B}_{w_2}(1)}<\infty \big)\asymp |w_1-w_2|^{2-d}$ (by Lemma \ref{lemma_prop_hitting}). Moreover, by Lemma \ref{lemma_compare_hm} one has 
\begin{equation}
	|w_2|^{d-2}\widetilde{\mathbb{P}}_{w_2'}\big(\tau_{\partial B(N_1)}=\tau_{y}<\tau_{D}  \big)\asymp N_2^{d-2}\widetilde{\mathbb{P}}_{x}\big(\tau_{\partial B(N_1)}=\tau_{y}<\tau_{D}  \big).
\end{equation}
Plugging these estimates into (\ref{newineq_2.73}), we get 
\begin{equation}\label{new_ineq2.75}
	\mathbb{P}\big(\mathsf{A}_1 \big)\lesssim |w_1|^{2-d}|w_1-w_2|^{2-d}|w_2|^{2-d}N_2^{d-2}. 
\end{equation}
By swapping $w_1$ and $w_2$, we know that the bound in (\ref{new_ineq2.75}) also holds for $\mathbb{P}\big(\mathsf{A}_2 \big)$, and thus obtain (\ref{newineq_2.69}). We omit the proof of (\ref{newnewineq_2.73}) since it follows similarly.
	\end{proof}

By taking $w_1=w_2=w\in [B(10N_2)]^c$ in (\ref{newineq_2.69}), we have 
\begin{equation}\label{special_newineq_2.69}
		\mathbb{P}\big(\mathrm{ran}(\widetilde{\eta}^{\mathrm{F}})\cap \widetilde{B}_{w}(1)\neq \emptyset \big)\lesssim |w|^{4-2d}N_2^{d-2}.
\end{equation}

\textbf{van den Berg-Kesten-Reimer (BKR) inequality.} The BKR inequality is crucial in the analysis of loop clusters. To state it formally, we first need to divide the loops into following types: 
\begin{itemize}
	\item  fundamental loop: a loop that intersects at least two points in $\mathbb{Z}^d$;

	\item  point loop: a loop that intersects exactly one point in $\mathbb{Z}^d$;

	\item  edge loop: a loop that is contained in a single interval $I_e$, $e\in \mathbb{L}^{d}$.	
\end{itemize} 
Based on this classification, the glued loops of $\widetilde{\mathcal{L}}_{1/2}^{D}$ are defined as follows (see \cite{cai2024high}): 
\begin{itemize}
	\item  For any connected $A\subset \mathbb{Z}^d\setminus D$ intersecting at least two points in $\mathbb{Z}^d$, the glued fundamental loop supported on $A$ is defined as the union of ranges of fundamental loops $\widetilde{\ell}\in \widetilde{\mathcal{L}}_{1/2}^{D}$ such that $\mathrm{ran}(\widetilde{\ell})\cap \mathbb{Z}^d=A$.

	\item  For any $x\in \mathbb{Z}^d\setminus D$, the glued point loop supported on $x$ is the union of ranges of point loops $\widetilde{\ell}\in \widetilde{\mathcal{L}}_{1/2}^{D}$ such that $x\in \mathrm{ran}(\widetilde{\ell})$.

	\item  For any $e\in \mathbb{L}^d$, the glued edge loop supported on $I_e$ is the union of ranges of edge loops $\widetilde{\ell} \in \widetilde{\mathcal{L}}_{1/2}^{D}$ contained in $I_e$.


\end{itemize}

For any events $\mathsf{A}_1,...,\mathsf{A}_j$ ($j\ge 2$) that are measurable with respect to $\widetilde{\mathcal{L}}_{1/2}^{D}$, we say they happen disjointly, denoted by $\mathsf{A}_1\circ \mathsf{A}_2\circ ... \circ \mathsf{A}_j$, if there exist $j$ disjoint collections of glued loops such that for each $1\le i\le j$, the $i$-th collection certifies the event $\mathsf{A}_i$. It is important to clarify that ``disjoint collections'' here means that the collections do not share any glued loop with the same type and support. However, this does not require that every glued loop in one collection is disjoint from those in other collections. Conversely, if $\mathsf{A}_1,...,\mathsf{A}_j$ are certified by different loop clusters, then they indeed happen disjointly.

 For the loop soup $\widetilde{\mathcal{L}}_{1/2}^{D}$, we say an event $\mathsf{A}$ is a connecting event if there exists $A_1,A_2\subset \widetilde{\mathbb{Z}}^{d}\setminus D$ such that $\mathsf{A}=\{A_1\xleftrightarrow{(D)}A_2\}$. Based on \cite[Corollary 8]{arratia2018van}, one can obtain the following version of the BKR inequality (for a detailed proof, readers may refer to \cite[Corollary 3.4]{cai2024high}).

\begin{lemma}\label{lemma_BKR}
	If events $\mathsf{A}_1,\mathsf{A}_2,...,\mathsf{A}_j$ ($j\ge 2$) are connecting events, then we have
	\begin{equation}
		\mathbb{P}\big(\mathsf{A}_1\circ \mathsf{A}_2\circ ... \circ \mathsf{A}_j\big)\le \prod\nolimits_{1\le i\le j} \mathbb{P}\big(\mathsf{A}_i\big).
	\end{equation}
\end{lemma}

\begin{remark}\label{remark_decompose}
We introduce an extension of Lemma \ref{lemma_BKR} as follows. Arbitrarily take $D\subset \widetilde{\mathbb{Z}}^d$ and two disjoint sets $A_1,A_2\in \mathbb{Z}^d\setminus D$. Assume that all starting and ending points of forward and backward crossing paths of all loops in $\widetilde{\mathcal{L}}_{1/2}^{D}$ from $A_1$ to $A_2$ are given. By Lemma \ref{lemma_decomposition}, these crossing paths (enumerated by $\widetilde{\eta}_i^{\mathrm{F}},\widetilde{\eta}_i^{\mathrm{B}}$ for $1\le i\le k$) are conditionally independent. For the remaining loops (i.e., the loops in $\widetilde{\mathcal{L}}_{1/2}^{D}$ that do not intersect both $A_1$ and $A_2$), we define the glued loops as above. Similarly, for any $j$ events ($j\ge 2$), we say they happen disjointly if they are certified by $j$ disjoint collections of glued loops and crossing paths. Here ``disjoint collections'' means that the collections do not share any glued loop with the same type and support, nor any crossing path with the same superscript and subscript. In this way, for the conditional probability measure given the starting and ending points of all crossing paths, the framework in \cite{arratia2018van} and the argument in the proof of \cite[Corollary 3.4]{cai2024high} still work, and thus the BKR inequality holds as in Lemma \ref{lemma_BKR}.
\end{remark}

Using the isomorphism theorem, the loop decomposition and the BKR inequality,  it was proved in \cite[Lemma 3.3]{cai2024one} that for any $A\subset \mathbb{Z}^d$ and $x\in \mathbb{Z}^d\setminus A$,
\begin{equation}\label{ineq_2.22}
	\mathbb{E}\big[\widetilde{\phi}_x\cdot \mathbbm{1}_{x\xleftrightarrow{\ge 0} A}  \big] \lesssim  \mathbb{P}\big(x\xleftrightarrow{\ge 0} A\big). 
\end{equation}
The proof of \cite[Lemma 3.3]{cai2024one} is mainly based on the following observations:
\begin{enumerate}[(i)]

	\item The total local time of loops at $x$ only depends on the point loops at $x$ and the backward crossing paths from $\widetilde{\partial} \widetilde{B}_x(1)$ to $x$.

	\item Given the number of crossings from $\widetilde{\partial} \widetilde{B}_x(1)$ to $x$ (which decays exponentially), the event $\widetilde{\partial} \widetilde{B}_x(1) \xleftrightarrow{} A$ only depends on the loops that are not contained in $\widetilde{B}_x(1)$ and the forward crossing paths from $\widetilde{\partial} \widetilde{B}_x(1)$ to $x$, and thus is conditional independent of the total local time at $x$.

\end{enumerate}
Clearly, these observations are also valid for the sub-graph $\widetilde{\mathbb{Z}}^d\setminus D$. Hence, (\ref{ineq_2.22}) can be extended to the following version: for any $x\in \mathbb{Z}^d$ and $A,D\subset \widetilde{\mathbb{Z}}^d\setminus \widetilde{B}_x(1)$, 
          \begin{equation}\label{ineq_2.23}
	\mathbb{E}^{D}\big[\widetilde{\phi}_x\cdot \mathbbm{1}_{x\xleftrightarrow{\ge 0} A}  \big] \lesssim \mathbb{P}^{D}\big(x\xleftrightarrow{\ge 0} A\big). 
\end{equation}
Readers may refer to \cite[Section 3]{cai2024one} for the proof details. In addition, by the FKG inequality, one has $\mathbb{E}^{D}\big[\widetilde{\phi}_x\cdot \mathbbm{1}_{x\xleftrightarrow{\ge 0} A}  \big] = \mathbb{E}^{D}\big[|\widetilde{\phi}_x|\cdot \mathbbm{1}_{x\xleftrightarrow{\ge 0} A}  \big]\gtrsim  \mathbb{P}^{D}\big(x\xleftrightarrow{\ge 0} A\big)$. Thus, 
\begin{equation}\label{ineq_2.25}
		 \mathbb{P}^{D}\big(x\xleftrightarrow{\ge 0} A\big) \asymp \mathbb{E}^{D}\big[\widetilde{\phi}_x\cdot \mathbbm{1}_{x\xleftrightarrow{\ge 0} A}  \big].	\end{equation}

Next, we provide a useful extension of (\ref{ineq_2.25}) as follows. 
	
\begin{lemma}\label{lemma3.3}
	For any $d\ge 3$, $x\in \mathbb{Z}^d$ and $A,D\subset  \widetilde{\mathbb{Z}}^d\setminus \widetilde{B}_x(1)$, and any increasing event $\mathsf{F}\subset \{x\xleftrightarrow{\le 0} A\}^{c}$ that is measurable with respect to $\mathcal{F}_{\widehat{\mathcal{C}}_A^{-}}$, we have 
	\begin{equation}\label{ineq3.4}
\mathbb{P}^{D}(\mathsf{F})	\mathbb{P}^{D}(x\xleftrightarrow{\ge 0} A)\lesssim 	\mathbb{E}^{D}\big[\widetilde{\phi}_x\cdot \mathbbm{1}_{ \mathsf{F}}  \big] =  \mathbb{E}^{D}\big[\mathcal{H}_x(\widehat{\mathcal{C}}_A^{-}\cup D) \cdot  \mathbbm{1}_{\mathsf{F}}\big]\lesssim \mathbb{P}^{D}(x\xleftrightarrow{\ge 0} A).
	\end{equation}
\end{lemma}	
\begin{proof}
The equality in (\ref{ineq3.4}) directly follows from the strong Markov property of the GFF. For the inequalities, we decompose $\mathbb{E}^{D}\big[\widetilde{\phi}_x\cdot \mathbbm{1}_{ \mathsf{F}}  \big]$ into two parts as follows: 
\begin{equation}\label{ineq3.5}
	\mathbb{E}^{D}\big[\widetilde{\phi}_x\cdot \mathbbm{1}_{ \mathsf{F}}  \big]= \mathbb{E}^{D}\big[\widetilde{\phi}_x\cdot \mathbbm{1}_{ \mathsf{F}\cap \{x\xleftrightarrow{\ge 0} A\}^{c}}  \big] + \mathbb{E}^{D}\big[\widetilde{\phi}_x\cdot \mathbbm{1}_{ \mathsf{F}\cap \{x\xleftrightarrow{\ge 0} A\}}  \big].
\end{equation}

We now show that the first term on the right-hand side of (\ref{ineq3.5}) equals $0$, as detailed below. Arbitrarily given the sign cluster containing $A$ (i.e., $\mathcal{C}_{A}^{\pm}:= \{v\in \widetilde{\mathbb{Z}}^d:v\xleftrightarrow{\ge 0 }A\ \text{or}\  v\xleftrightarrow{\le 0 }A\}$) and the GFF values on this sign cluster which satisfies the event $\mathsf{F}\cap \{x\xleftrightarrow{\ge 0} A\}^{c}$ (note that $\mathcal{C}_{A}^{\pm}$ has zero boundary values, and does not intersect $x$ since $\mathsf{F}\subset \{x\xleftrightarrow{\le 0} A\}^{c}$), the conditional expectation of $\widetilde{\phi}_x$ equals $0$ by the strong Markov property of the GFF. This implies $\mathbb{E}^{D}\big[\widetilde{\phi}_x\cdot \mathbbm{1}_{ \mathsf{F}\cap \{x\xleftrightarrow{\ge 0} A\}^{c}}  \big]=0$.

For the second term $\mathbb{E}^{D}\big[\widetilde{\phi}_x\cdot \mathbbm{1}_{ \mathsf{F}\cap \{x\xleftrightarrow{\ge 0} A\}}  \big]$, on the one hand, by (\ref{ineq_2.25}) we have 
\begin{equation}
	\mathbb{E}^{D}\big[\widetilde{\phi}_x\cdot \mathbbm{1}_{ \mathsf{F}\cap \{x\xleftrightarrow{\ge 0} A\}}  \big]\le \mathbb{E}^{D}\big[\widetilde{\phi}_x\cdot \mathbbm{1}_{ x\xleftrightarrow{\ge 0} A}  \big]\asymp  \mathbb{P}^{D}(x\xleftrightarrow{\ge 0} A).
\end{equation}
On the other hand, by the FKG inequality and (\ref{ineq_2.25}), one has 
\begin{equation}
	\mathbb{E}^{D}\big[\widetilde{\phi}_x\cdot \mathbbm{1}_{ \mathsf{F}\cap \{x\xleftrightarrow{\ge 0} A\}}  \big]\ge \mathbb{P}^{D}(\mathsf{F}) \cdot  \mathbb{E}^{D}\big[\widetilde{\phi}_x\cdot \mathbbm{1}_{x\xleftrightarrow{\ge 0} A}  \big]\asymp  \mathbb{P}^{D}(\mathsf{F}) \cdot 	\mathbb{P}^{D}(x\xleftrightarrow{\ge 0} A).
\end{equation}
To sum up, we conclude the proof of this lemma.
\end{proof}

Using Lemma \ref{lemma3.3}, in the next result we show that the point-to-set connecting probability is subharmonic as a function of the point.

\begin{lemma}\label{lemma_subharminic}
	For any disjoint $A,D_1\subset \widetilde{\mathbb{Z}}^d$ and $D_2\subset \mathbb{Z}^d$ such that $\widetilde{B}_w(1)\cap (A\cup D_1)=\emptyset$ for all $w\in \partial D_2$, and for any $x\in D_2$, 
	\begin{equation}
		\mathbb{P}^{D_1}(x\xleftrightarrow{\ge 0} A)\lesssim  \sum\nolimits_{w\in \partial  D_2} \widetilde{\mathbb{P}}_x \big(\tau_{\partial  D_2}= \tau_w < \tau_{D_1} \big)\cdot \mathbb{P}^{D_1}(w\xleftrightarrow{\ge 0} A). 
	\end{equation} 
\end{lemma}
\begin{proof}
	By the strong Markov property, one has 
	\begin{equation}\label{final_revision2.70}
		\mathcal{H}_x\big( \widehat{\mathcal{C}}_A^{-}\cup D_1 \big)\le  \sum\nolimits_{w\in  \partial D_2} \widetilde{\mathbb{P}}_x \big(\tau_{ \partial  D_2}= \tau_w < \tau_{D_1} \big)\cdot \mathcal{H}_w\big( \widehat{\mathcal{C}}_A^{-}\cup D_1 \big). 
	\end{equation}
	Combining (\ref{final_revision2.70}) with Lemma \ref{lemma3.3}, we obtain this lemma:
	\begin{equation*}
		\begin{split}
			\mathbb{P}^{D_1}(x\xleftrightarrow{\ge 0} A) \overset{\text{Lemma}\ \ref*{lemma3.3}}{\lesssim } & \mathbb{E}\big[ \mathcal{H}_x\big( \widehat{\mathcal{C}}_A^{-}\cup D_1 \big) \big] \\
			\overset{(\ref*{final_revision2.70})}{\le} & \sum\nolimits_{w\in  \partial D_2} \widetilde{\mathbb{P}}_x \big(\tau_{ \partial  D_2}= \tau_w < \tau_{D_1} \big)\cdot  \mathbb{E}\big[ \mathcal{H}_w\big( \widehat{\mathcal{C}}_A^{-}\cup D_1 \big) \big]\\
		\overset{\text{Lemma}\ \ref*{lemma3.3}}{\lesssim } & \sum\nolimits_{w\in \partial  D_2} \widetilde{\mathbb{P}}_x \big(\tau_{\partial  D_2}= \tau_w < \tau_{D_1} \big)\cdot \mathbb{P}^{D_1}(w\xleftrightarrow{\ge 0} A). \qedhere
		\end{split}
	\end{equation*}

\end{proof}

We record a useful application of Lemma \ref{lemma_subharminic} as follows.

\begin{corollary}\label{coro_subharmonic}
	For any $d\ge 3$, $N\ge 1$, $x\in \mathbb{Z}^d\setminus B(3N)$ and $D\subset \widetilde{\mathbb{Z}}^d$, 
		\begin{equation}
		\widetilde{\mathbb{P}}^{D}\big(x\xleftrightarrow{\ge 0} B(N) \big)
		\lesssim  |x|^{2-d} N^{[(\frac{d}{2}-1)\boxdot  (d-4)]+\varsigma_d(N)}. 
	\end{equation}
\end{corollary}
\begin{proof}
	By (\ref{ineq_compare}) and Lemma \ref{lemma_subharminic}, we have 
	\begin{equation}\label{revise2.68}
		\begin{split}
			\widetilde{\mathbb{P}}^{D}\big(x\xleftrightarrow{\ge 0} B(N) \big) \le &2\widetilde{\mathbb{P}}\big(x\xleftrightarrow{\ge 0} B(N) \big)\\
			\lesssim  & \sum\nolimits_{y\in \partial B(2N)}  \widetilde{\mathbb{P}}_x\big(\tau_{\partial B(2N)}= \tau_y < \infty \big)\cdot \mathbb{P}\big(y \xleftrightarrow{\ge 0} B(N)\big). 
		\end{split}
	\end{equation}
	For any $y\in \partial B(2N)$, since $\{y \xleftrightarrow{\ge 0} B(N)\}\subset \{y \xleftrightarrow{\ge 0} B_y(N)\}$, one has 
	\begin{equation}\label{revise2.69}
		\mathbb{P}\big(y \xleftrightarrow{\ge 0} B(N)\big) \le \theta_d(N) \overset{(\ref*{one_arm_low})-(\ref*{one_arm_high})}{\lesssim} N^{[(-\frac{d}{2}+1)\boxdot (-2)]+\varsigma_d(N)}. 
	\end{equation}
	Plugging (\ref{revise2.69}) into (\ref{revise2.68}) and using (\ref{new2.20_1}), we obtain the desired bound:
	\begin{equation}
		\begin{split}
			\widetilde{\mathbb{P}}^{D}\big(x\xleftrightarrow{\ge 0} B(N) \big) \lesssim  &   \widetilde{\mathbb{P}}_x\big(\tau_{\partial B(2N)} < \infty \big)\cdot  N^{[(-\frac{d}{2}+1)\boxdot (-2)]+\varsigma_d(N)}\\
			\lesssim  & |x|^{2-d} N^{[(\frac{d}{2}-1)\boxdot  (d-4)]+\varsigma_d(N)}. \qedhere
		\end{split}
	\end{equation}

\end{proof}

\subsection{Formulas for the sum over lattice}

We cite the following formulas in \cite{cai2024high}, which will be frequently-used
 in our calculations for high-dimensional GFFs.

\begin{lemma}[{\cite[Lemma 4.2]{cai2024high}}]
	For any $d\ge 3$, $a\in \mathbb{R}$ and $M\ge 1$, 
\begin{align}
	&\text{when}\ a\neq d-1,\ \max\nolimits_{y\in \mathbb{Z}^d}\sum\nolimits_{x\in \partial B(M)}  |x-y|^{-a} \lesssim M^{(d-1-a)\vee 0}; \label{|x|-a_aneqd-1} \\
	&\text{when}\ a\neq d, \ \ \ \ \ \  \max\nolimits_{y\in \mathbb{Z}^d}\sum\nolimits_{x\in B(M)}  |x-y|^{-a} \lesssim M^{(d-a)\vee 0}.\label{|x|-a_aneqd}
\end{align}
\end{lemma}

\begin{lemma}[{\cite[Lemmas 4.3, 4.4, Corollary 4.5]{cai2024high}}]
	For any $d\ge 7$ and $x,y\in \mathbb{Z}^d$, 
		\begin{equation}\label{2-d_2-d}
		\sum\nolimits_{z\in \mathbb{Z}^d}|x-z|^{2-d}|z-y|^{2-d}\lesssim |x-y|^{4-d},
	\end{equation}
	\begin{equation}\label{2-d_4-d}
			\sum\nolimits_{z\in \mathbb{Z}^d}|x-z|^{2-d}|z-y|^{4-d}\lesssim |x-y|^{6-d}, 
	\end{equation}
	\begin{equation}\label{2-d_6-2d}
		\sum\nolimits_{z\in \mathbb{Z}^d}|x-z|^{2-d}|z-y|^{6-2d}\lesssim |x-y|^{2-d},
	\end{equation}
	\begin{equation}\label{z1z2z3}
		\sum_{z_1,z_2,z_3\in \mathbb{Z}^d}|z_1-z_2|^{2-d}|z_2-z_3|^{2-d}|z_3-z_1|^{2-d}|x-z_1|^{2-d}|y-z_2|^{2-d}\lesssim |x-y|^{4-d},
	\end{equation}
		\begin{equation}\label{z1z2}
		\sum\nolimits_{z_1,z_2\in \mathbb{Z}^d}|x-z_1|^{2-d}|z_1-z_2|^{2-d}|z_2-x|^{2-d}|z_1-y|^{2-d}\lesssim |x-y|^{2-d}.
	\end{equation}
\end{lemma}

We conclude this section by presenting the following lemma. 
\begin{lemma}\label{lemma_1-d_6-2d}
	For any $d\ge 7$, there exists $C>0$ such that for all $n\ge 1$, 
	\begin{equation}\label{finish2.90}
		\sum\nolimits_{z_1\in \mathcal{B}(n),z_2\in \partial \mathcal{B}(n)}  |z_1|^{1-d}|z_1-z_2|^{6-2d}\le C.
	\end{equation} 
\end{lemma}
\begin{proof}
	We claim that for any $z_1\in \mathcal{B}(n) $,
	\begin{equation}\label{latelate5.45}
	\begin{split}
		\sum\nolimits_{z_2\in \partial \mathcal{B}(n)} |z_1-z_2|^{6-2d}\lesssim  \big( n-|z_1|\big)^{5-d}. 
	\end{split}
\end{equation} 
To see this, for $w\in \partial \mathcal{B}_{z_1}(n-|z_1|)$, we denote by $\mathfrak{T}_w$ the set of points $z_2\in \partial \mathcal{B}(n)$ such that the line in $\mathbb{R}^d$ connecting $z_1$ and $z_2$ intersects the Euclidean ball in $\mathbb{R}^d$ centered at $w$ with radius $d$. Since each interval $I_e\subset \widetilde{\mathbb{Z}}^d$ has length $d$, one has 
\begin{equation}\label{revise2.90}
	\partial \mathcal{B}(n)\subset \cup_{w\in \partial \mathcal{B}_{z_1}(n-|z_1|)}\mathfrak{T}_w. 
\end{equation}
In addition, since $\partial \mathcal{B}(n)$ is $(d-1)$-dimensional, one has 
\begin{equation}\label{revise2.91}
	|\mathfrak{T}_w|\asymp \big(\frac{r_w}{n-|z_1|}\big)^{d-1}, 
\end{equation}
where $r_w$ is the Euclidean distance between $w$ and $\mathfrak{T}_w$. By these two observations, we can derive the claim (\ref{latelate5.45}) as follows:
\begin{equation*}
	\begin{split}
	\sum\nolimits_{z_2\in \partial \mathcal{B}(n)} |z_1-z_2|^{6-2d} 
		 \overset{(\ref*{revise2.90})}{\le}&  \sum\nolimits_{w\in  \partial \mathcal{B}_{z_1}(n-|z_1|)} \sum\nolimits_{z_2\in \mathfrak{T}_w}|z_1-z_2|^{6-2d}\\
		 \overset{(\ref*{revise2.91}), \mathrm{vol}( \partial \mathcal{B}_{z_1}(n-|z_1|) )\asymp (n-|z_1|)^{d-1}}{\lesssim  } &  (n-|z_1|)^{d-1}\cdot  \big(\frac{r_w}{n-|z_1|}\big)^{d-1} \cdot r_w^{6-2d} \\
		 \overset{r_w\gtrsim  n-|z_1|,d\ge 7}{\lesssim } & (n-|z_1|)^{5-d}. 
	\end{split}
\end{equation*}
By (\ref{latelate5.45}), $\mathcal{B}(n)=\cup_{0\le k\le n} \partial \mathcal{B}(k)$ and $\mathrm{vol}(\partial \mathcal{B}(k))\asymp k^{d-1}$ for $k\ge 1$, we obtain (\ref{finish2.90}): 
\begin{equation}
	\begin{split}
		&\sum\nolimits_{z_1\in \mathcal{B}(n),z_2\in \partial \mathcal{B}(n)}  |z_1|^{1-d}|z_1-z_2|^{6-2d} \\
		\overset{}{\lesssim} & \sum\nolimits_{z_1\in \mathcal{B}(n) }  |z_1|^{1-d} \big( n-|z_1|\big)^{5-d}\\
		\lesssim & \sum\nolimits_{0\le k\le n} k^{d-1} \cdot k^{1-d}(n-k)^{5-d} \overset{d\ge 7}{\le} C. \qedhere
	\end{split}
\end{equation}

\end{proof}

\section{Properties of point-to-set connecting probabilities}\label{section_property_pointtoset}

This section is devoted to proving Propositions \ref{prop_pointtoset_speed} and \ref{prop_pointtoset_harnack}. Before showing the proof details, we first offer some intuitions behind these results. To prove that the point-to-set connecting probability exhibits certain ``harmonic'' properties, a natural approach is to relate it to a quantity that possesses these properties, which is the harmonic average in our case. More precisely, as shown in Lemma \ref{lemma3.3}, the connecting probability between $x$ and $A$ (where $x\in \mathbb{Z}^d$ and $A\subset \widetilde{\mathbb{Z}}^d\setminus \widetilde{B}_x(1)$) is comparable to the expectation of the harmonic average at $x$ after exploring the negative cluster containing $A$. Moreover, when $A\subset \widetilde{B}(N)$, if the negative cluster containing $A$ remains in the box $\widetilde{B}(CN)$, then the harmonic average for $x\in [\widetilde{B}(2CN)]^c$ automatically satisfies the ``harmonic'' properties such as Harnack's inequality and the decay rate of $|x|^{2-d}$. However, by the estimates on crossing probabilities in (\ref{crossing_low}) and (\ref{crossing_high}), such a favorable event (i.e., $\widehat{\mathcal{C}}_A^{-}\subset \widetilde{B}(CN)$) occurs with a uniformly positive probability when $3\le d<6$, but occurs only with a vanishing probability when $d> 6$. This phenomenon necessitates a more sophisticated analysis on the probability for an independent Brownian motion to hit the negative cluster containing $A$ in high dimensions. As shown in the next subsection, we achieve this by showing that adding the negative cluster containing $A$ as an obstacle (i.e., stopping the Brownian motion upon hitting this negative cluster) will decrease the Green's function outside $\widetilde{B}(2CN)$ by only a constant factor, in an averaged sense. Furthermore, we expect that our arguments for high dimensions can be extended to the critical dimension $6$, assuming that the order of $\theta_6(N)$ is derived.

\subsection{Proof of Proposition \ref{prop_pointtoset_speed}}\label{subsection_proof_pointtoset_speed}
 We first consider the case when $3\le d\le 6$. Recall that $\rho_d(n,N)= \mathbb{P}\big(B(n)\xleftrightarrow{\ge 0} \partial B(N)\big)$. By (\ref{crossing_low}) and  (\ref{crossing_6}), one has
\begin{equation}\label{3.8}
\max\big\{	\rho_d(N, \lambda N^{1+\varsigma_d(N)}),\rho_d(\lambda^{-1}N^{1-\varsigma_d(N)},N ) \big\} \lesssim  \lambda^{-\frac{d}{2}+1} \overset{\lambda \to \infty }{\to } 0
\end{equation}
for $3\le d\le 6$ and $N\ge 1$. Thus, by $A\subset \widetilde{B}(N)$, (\ref{coro2.1_1}), (\ref{ineq_compare}) and (\ref{3.8}), we have 
\begin{equation}\label{ineq3.9}
	\begin{split}
	\mathbb{P}^D\big( A \xleftrightarrow{\le 0} \partial B(C_*N^{1+\varsigma_d(N)}) \big) 
		\overset{(\ref*{coro2.1_1})}{=} &\mathbb{P}^D\big( A \xleftrightarrow{\ge 0} \partial B(C_*N^{1+\varsigma_d(N)}) \big)\\
		\overset{(\ref*{ineq_compare})}{\le } & 2\mathbb{P}\big( B(N)\xleftrightarrow{\ge 0} \partial B(C_*N^{1+\varsigma_d(N)}) \big) \overset{(\ref*{3.8})}{\le} \tfrac{1}{4}, 
	\end{split}
\end{equation}
where we took a sufficiently large $C_*$ to ensure the last inequality. Arbitrarily
take $x_1,x_2\in [B(2C_*N^{1+\varsigma_d(N)})]^c$. Note that $\mathsf{F}:=\{A \xleftrightarrow{\le 0} \partial B(C_*N^{1+\varsigma_d(N)})\}^c$ is increasing in the field $\widetilde{\phi}_\cdot$ and measurable with respect to $\widehat{\mathcal{C}}_A^{-}$. In addition, the event $\mathsf F$ implies $\cap_{i=1,2}\{A \xleftrightarrow{\le 0} x_i\}^c$, and satisfies $\mathbb{P}^{D}(\mathsf{F})\ge \frac{3}{4}$ (by (\ref{ineq3.9})). Thus, by Lemma \ref{lemma3.3},\begin{equation}\label{ineq3.12}
	\mathbb{E}^{D}\big[\mathcal{H}_{x_i}(\widehat{\mathcal{C}}_A^{-}\cup D) \cdot \mathbbm{1}_{\mathsf{F}} \big] \asymp \mathbb{P}^D\big( x_i \xleftrightarrow{\ge 0}  A \big), \ \ \forall i\in \{1,2\}. 
\end{equation}
On $\mathsf{F}$ (which is equivalent to $\{\widehat{\mathcal{C}}^-_{A}\subset \widetilde{B}(C_*N^{1+\varsigma_d(N)})\}$), by Lemma \ref{lemma_compare_hm} one has 
\begin{equation}
|x_1|^{d-2}	\mathcal{H}_{x_1}(\widehat{\mathcal{C}}_A^{-}\cup D) \asymp |x_2|^{d-2}	\mathcal{H}_{x_2}(\widehat{\mathcal{C}}_A^{-}\cup D). 
\end{equation}
Multiplying by $\mathbbm{1}_{\mathsf{F}}$ and taking the expectations on both sides, we get
\begin{equation}\label{ineq3.13}
	|x_1|^{d-2}\mathbb{E}\big[\mathcal{H}_{x_1}(\widehat{\mathcal{C}}_A^{-}\cup D) \cdot \mathbbm{1}_{\mathsf{F}} \big]\asymp 	|x_2|^{d-2}\mathbb{E}\big[\mathcal{H}_{x_2}(\widehat{\mathcal{C}}_A^{-}\cup D) \cdot \mathbbm{1}_{\mathsf{F}} \big]. 
\end{equation}
Combining (\ref{ineq3.12}) and (\ref{ineq3.13}), we obtain this proposition for $3\le d\le 6$.

Now we consider the high-dimensional cases $d\ge 7$. Take a large constant $C_\diamond$ and let $x_1,x_2\in [B(C_\diamond N)]^c$. We divide our proof into two cases.

\noindent\textbf{Case 1:} $|x_1-x_2|\le \frac{1}{20}(|x_1|\vee |x_2|)$. We first show that when $|x_1-x_2|\le \frac{1}{10}|x_1|$,
\begin{equation}\label{PDx_1x_2_case1}
	\mathbb{P}^D\big( x_1 \xleftrightarrow{\ge 0}  A \big) \gtrsim  \mathbb{P}^{D}\big( x_2 \xleftrightarrow{\ge 0}  A \big).
\end{equation}
By Lemma \ref{lemma3.3}, we know that $\mathbb{P}^D\big( x_1 \xleftrightarrow{\ge 0}  A \big)$ is of the same order as 
\begin{equation*}
	\begin{split}
		\mathbb{E}^{D}\big[ \mathcal{H}_{x_1}(\widehat{\mathcal{C}}_A^{-}\cup D)\big]=\mathbb{E}^{D}\Big[ \sum_{z_1 \in \partial B_{x_1}(\frac{1}{2}|x_1|)} \widetilde{\mathbb{P}}_{x_1}\big(\tau_{\partial B_{x_1}(\frac{1}{2}|x_1|)}=\tau_{z_1}<\tau_{\widehat{\mathcal{C}}_A^{-}\cup D} \big) \mathcal{H}_{z_1}(\widehat{\mathcal{C}}_A^{-}\cup D)  \Big].  
	\end{split}
\end{equation*}
Moreover, by the last-exit decomposition, for any $z_1 \in \partial B_{x_1}(\frac{1}{2}|x_1|)$, one has 
\begin{equation*}
	\begin{split}
		&\widetilde{\mathbb{P}}_{x_1}\big(\tau_{\partial B_{x_1}(\frac{1}{2}|x_1|)}=\tau_{z_1}<\tau_{\widehat{\mathcal{C}}_A^{-}\cup D} \big)\\
		=&\sum_{z_2\in \partial B_{x_1}(\frac{1}{5}|x_1|)}\widetilde{G}_{\widehat{\mathcal{C}}_A^{-}\cup D\cup \partial B_{x_1}(\frac{1}{2}|x_1|)}(x_1,z_2)\\
		 &\cdot \tfrac{1}{2d} \sum_{z_2'\in [B_{x_1}(\frac{1}{5}|x_1|)]^c:\{z_2,z_2'\}\in \mathbb{L}^d}\widetilde{\mathbb{P}}_{z_2'}\big(\tau_{\partial B_{x_1}(\frac{1}{2}|x_1|)}=\tau_{z_1}<\tau_{\widehat{\mathcal{C}}_A^{-}\cup D\cup \partial B_{x_1}(\frac{1}{5}|x_1|) }\big).	\end{split}
\end{equation*}
In addition, $\widetilde{G}_{\widehat{\mathcal{C}}_A^{-}\cup D\cup \partial B_{x_1}(\frac{1}{2}|x_1|)}(x_1,z_2)$, $\widetilde{\mathbb{P}}_{z_2'}\big(\tau_{\partial B_{x_1}(\frac{1}{2}|x_1|)}=\tau_{z_1}<\tau_{\widehat{\mathcal{C}}_A^{-}\cup D\cup \partial B_{x_1}(\frac{1}{5}|x_1|) }\big)$ and $\mathcal{H}_{z_1}(\widehat{\mathcal{C}}_A^{-}\cup D)$ are all increasing (because the negative cluster $\widehat{\mathcal{C}}_A^{-}$ shrinks as the GFF values increase). Altogether, we can derive from the FKG inequality that
\begin{equation}\label{newineq_3.6}
	\begin{split}
		\mathbb{E}^{D}\big[ \mathcal{H}_{x_1}(\widehat{\mathcal{C}}_A^{-}\cup D)\big]\ge   \sum_{z_1 \in \partial B_{x_1}(\frac{1}{2}|x_1|)}  \sum_{z_2\in \partial B_{x_1}(\frac{1}{5}|x_1|)} \mathbb{I}_{z_2}\cdot \mathbb{J}_{z_1,z_2},
	\end{split}
\end{equation}
where $\mathbb{I}_{z_2}:=\mathbb{E}^D\Big[\widetilde{G}_{\widehat{\mathcal{C}}_A^{-}\cup D\cup \partial B_{x_1}(\frac{1}{2}|x_1|)}(x_1,z_2)\Big]$, and $\mathbb{J}_{z_1,z_2}$ is defined as 
\begin{equation*}
	\mathbb{E}^D\Big[  \tfrac{1}{2d}\sum_{z_2'\in [B_{x_1}(\frac{1}{5}|x_1|)]^c:\{z_2,z_2'\}\in \mathbb{L}^d}\widetilde{\mathbb{P}}_{z_2'}\big(\tau_{\partial B_{x_1}(\frac{1}{2}|x_1|)}=\tau_{z_1}<\tau_{\widehat{\mathcal{C}}_A^{-}\cup D\cup \partial B_{x_1}(\frac{1}{5}|x_1|) }\big)  \mathcal{H}_{z_1}(\widehat{\mathcal{C}}_A^{-}\cup D)  \Big].
\end{equation*}

We claim that for any $z_2\in \partial B_{x_1}(\frac{1}{5}|x_1|)$, 
\begin{equation}\label{ineq_Iz_2}
	\mathbb{I}_{z_2}  \gtrsim \widetilde{G}(x_1,z_2).\end{equation}
To see this, for any $x\in \mathbb{Z}^d$, recall that $\widehat{\mathbb{P}}_{x}$ denotes the law of a simple random walk $\{S_n\}_{n\ge 0}$ starting from $x$, and that the projection of $\{\widetilde{S}_t\}_{t\ge 0}\sim \widetilde{\mathbb{P}}_x$ onto $\mathbb{Z}^d$ has the law $\widehat{\mathbb{P}}_{x}$. Based on this correspondence, $\mathbb{I}_{z_2}$ can be estimated as follows (here we also use the notation $\tau_\cdot$ for the hitting time of $\{S_n\}_{n\ge 0}$): 
\begin{equation}\label{newineq_3.8}
	\begin{split}
		\mathbb{I}_{z_2} \ge & \mathbb{E}^{D}\Big[ \sum_{|x_1|^2\le n\le 2|x_1|^2} \widehat{\mathbb{P}}_{x_1}\big(S_{n}=z_2,\tau_{\partial B_{x_1}(\frac{1}{2}|x_1|)}>n,  \bigcup_{0\le j\le n}\widetilde{B}_{S_j}(1)\cap \widehat{\mathcal{C}}_A^{-}=\emptyset\big) \Big]\\
		\ge &\sum_{|x_1|^2\le n\le 2|x_1|^2} \widehat{\mathbb{P}}_{x_1}\big(S_{n}=z_2, \tau_{\partial B_{x_1}(\frac{1}{2}|x_1|)}>n\big)\\
		&\cdot \min_{F\subset B_{x_1}(\frac{1}{2}|x_1|):\mathrm{vol}(F)\le 2|x_1|^2}\mathbb{P}^{D}\big(\widetilde{B}_y(1)\cap \widehat{\mathcal{C}}_A^{-}= \emptyset,\forall y\in F\big).
	\end{split}
\end{equation}
By the local central limit theorem (see e.g. \cite[Theorem 2.1.1]{lawler2010random}), one has 
\begin{equation}\label{newineq_3.9}
	\sum_{|x_1|^2\le n\le 2|x_1|^2} \widehat{\mathbb{P}}_{x_1}\big(S_{n}=z_2, \tau_{\partial B_{x_1}(\frac{1}{2}|x_1|)} >n \big) \asymp\sum_{|x_1|^2\le n\le 2|x_1|^2} |x_1|^{-d} \asymp |x_1|^{2-d}.
\end{equation}
Meanwhile, for any $F\subset B_{x_1}(\frac{1}{2}|x_1|)$ containing at most $2|x_1|^2$ points, 
\begin{equation}\label{new_add3.32}
\begin{split}
	&\mathbb{P}^{D}\big(\exists  y\in F\ \text{such that}\ \widetilde{B}_y(1)\cap \widehat{\mathcal{C}}_A^{-}\neq \emptyset \big) \\
	\le  &\sum\nolimits_{y\in F}  \mathbb{P}^{D}\big( \widetilde{B}_y(1)\cap \widehat{\mathcal{C}}_A^{-}\neq \emptyset\big) \\
	\overset{(\ref*{ineq_compare})}{\lesssim} &   |x_1|^2\max\nolimits_{y\in B_{x_1}(\frac{3}{5}|x_1|)}  \mathbb{P}\big(y \xleftrightarrow{\ge 0} \partial B(N)\big)\\
	\overset{\text{Corollary}\ \ref*{coro_subharmonic}}{\lesssim} &  |x_1|^{4-d} N^{d-4}\lesssim C_\diamond^{4-d}. 
	\end{split}	
\end{equation}
By plugging (\ref{newineq_3.9}) and (\ref{new_add3.32}) into (\ref{newineq_3.8}) and taking a large $C_\diamond$, we obtain (\ref{ineq_Iz_2}): 
\begin{equation*}
	\begin{split}
		\mathbb{I}_{z_2} \gtrsim (1-C_\diamond^{4-d})|x_1|^{2-d} \asymp \widetilde{G}(x_1,z_2), \ \ \forall z_2\in \partial B_{x_1}(\tfrac{1}{5}|x_1|). 
	\end{split}
\end{equation*}

For any $z_1 \in \partial B_{x_1}(\frac{1}{2}|x_1|)$, by $(\ref{ineq_Iz_2})$ and (\ref{green_D1_D2}), we have 
\begin{equation}\label{revise_3.12}
	\begin{split}
		&\sum\nolimits_{z_2\in \partial B_{x_1}(\frac{1}{5}|x_1|)} \mathbb{I}_{z_2}\cdot \mathbb{J}_{z_1,z_2}\\ \gtrsim &	\mathbb{E}^D\Big[ \widetilde{G}_{D\cup \partial B_{x_1}(\frac{1}{2}|x_1|)\cup [\widehat{\mathcal{C}}_A^{-}\setminus \widetilde{B}_{x_1}(\frac{1}{5}|x_1|) ]}(x_1,z_2) \cdot  \tfrac{1}{2d}\sum_{z_2'\in [B_{x_1}(\frac{1}{5}|x_1|)]^c:\{z_2,z_2'\}\in \mathbb{L}^d}\\
		&\ \ \ \ \ \ \ \widetilde{\mathbb{P}}_{z_2'}\big(\tau_{\partial B_{x_1}(\frac{1}{2}|x_1|)}=\tau_{z_1}<\tau_{\widehat{\mathcal{C}}_A^{-}\cup D\cup \partial B_{x_1}(\frac{1}{5}|x_1|) }\big)  \mathcal{H}_{z_1}(\widehat{\mathcal{C}}_A^{-}\cup D)  \Big]\\
		\overset{}{=} & \mathbb{E}^D\Big[ \widetilde{\mathbb{P}}_{x_1}\big(\tau_{\partial B_{x_1}(\frac{1}{2}|x_1|)}=\tau_{z_1}<\tau_{D\cup [\widehat{\mathcal{C}}_A^{-}\setminus \widetilde{B}_{x_1}(\frac{1}{5}|x_1|) ]} \big) \mathcal{H}_{z_1}(\widehat{\mathcal{C}}_A^{-}\cup D)  \Big],
	\end{split}
\end{equation}
where the last step follows from the last-exit decomposition. Furthermore, for any $x_2\in \mathbb{Z}^d$ with $|x_1-x_2|\le \frac{1}{10}|x_1|$, by Harnack's inequality, one has 
\begin{equation}\label{revise_3.13}
	\begin{split}
		& \widetilde{\mathbb{P}}_{x_1}\big(\tau_{\partial B_{x_1}(\frac{1}{2}|x_1|)}=\tau_{z_1}<\tau_{D\cup [\widehat{\mathcal{C}}_A^{-}\setminus \widetilde{B}_{x_1}(\frac{1}{5}|x_1|) ]} \big)\\
		\asymp &\widetilde{\mathbb{P}}_{x_2}\big(\tau_{\partial B_{x_1}(\frac{1}{2}|x_1|)}=\tau_{z_1}<\tau_{D\cup [\widehat{\mathcal{C}}_A^{-}\setminus \widetilde{B}_{x_1}(\frac{1}{5}|x_1|) ]} \big)\\
		\ge & \widetilde{\mathbb{P}}_{x_2}\big(\tau_{\partial B_{x_1}(\frac{1}{2}|x_1|)}=\tau_{z_1}<\tau_{D\cup \widehat{\mathcal{C}}_A^{-}} \big). 
	\end{split}
\end{equation}
 Putting (\ref{newineq_3.6}), (\ref{revise_3.12}) and (\ref{revise_3.13}) together, we get\begin{equation*}
	\begin{split}
\mathbb{E}^{D}\big[ \mathcal{H}_{x_1}(\widehat{\mathcal{C}}_A^{-}\cup D)\big]	\gtrsim &  \mathbb{E}^{D}\Big[ \sum_{z_1 \in \partial B_{x_1}(\frac{1}{2}|x_1|)} 
		\widetilde{\mathbb{P}}_{x_2}\big(\tau_{\partial B_{x_1}(\frac{1}{2}|x_1|)}=\tau_{z_1}<\tau_{D\cup \widehat{\mathcal{C}}_A^{-}} \big) \mathcal{H}_{z_1}(\widehat{\mathcal{C}}_A^{-}\cup D)   \Big]\\
		=& \mathbb{E}^{D}\big[ \mathcal{H}_{x_2}(\widehat{\mathcal{C}}_A^{-}\cup D)\big],
	\end{split}
\end{equation*}
which together with Lemma \ref{lemma3.3} implies (\ref{PDx_1x_2_case1}).

Note that $|x_1-x_2|\le \frac{1}{20}|x_1|$ implies $|x_1-x_2|\le \frac{1}{10}|x_2|$. Thus, by switching the roles of $x_1$ and $x_2$ in the proof of (\ref{PDx_1x_2_case1}), we also have 
\begin{equation}\label{newineq_3.16}
	\mathbb{P}^D\big( x_2 \xleftrightarrow{\ge 0}  A \big) \gtrsim  \mathbb{P}^{D}\big( x_1 \xleftrightarrow{\ge 0}  A \big), \ \ \text{when}\ |x_1-x_2|\le \tfrac{1}{20}|x_1|. 
\end{equation}
Also note that $|x_1-x_2|\le \frac{1}{20}(|x_1|\vee |x_2|)$ implies $|x_1|\asymp |x_2|$. Therefore, by (\ref{PDx_1x_2_case1}) and (\ref{newineq_3.16}), we conclude this proposition for Case 1.

\noindent \textbf{Case 2:} $|x_1-x_2|\ge \frac{1}{20}(|x_1|\vee |x_2|)$. In this case, we aim to show that
\begin{equation}\label{PDx_1x_2_case2}
	|x_1|^{d-2}\mathbb{P}^D\big( x_1 \xleftrightarrow{\ge 0}  A \big) \asymp |x_2|^{d-2}\mathbb{P}^{D}\big( x_2 \xleftrightarrow{\ge 0}  A \big).
\end{equation}
By Lemma \ref{lemma3.3}, we know that $\mathbb{P}^{D}\big( x_2 \xleftrightarrow{\ge 0}  A \big)$ is of the same order as $\mathbb{E}^{D}\big[\mathcal{H}_{x_2}(\widehat{\mathcal{C}}_A^{-}\cup D) \big] $, which, by the strong Markov property, is bounded from below by
\begin{equation}\label{new_ineq_3.18}
\begin{split}
	&\sum\nolimits_{w\in \partial \mathcal{B}_{x_1}(\frac{1}{100}|x_1|)} \mathbb{E}^{D}\big[\widetilde{\mathbb{P}}_{x_2}\big(\tau_{ \partial \mathcal{B}_{x_1}(\frac{1}{100}|x_1|)}=\tau_{w}<\tau_{\widehat{\mathcal{C}}_A^{-} \cup D\cup B(3N)} \big)  \mathcal{H}_{w}(\widehat{\mathcal{C}}_A^{-}\cup D)\big]    \\
			\ge &   \mathbb{K}-  \sum\nolimits_{w\in \partial \mathcal{B}_{x_1}(\frac{1}{100}|x_1|)}\mathbb{E}^{D}[\hat{\mathbf{K}}_w],
\end{split}
	\end{equation}
where $ \mathbb{K}:=\sum_{w\in \partial \mathcal{B}_{x_1}(\frac{1}{100}|x_1|)} \widetilde{\mathbb{P}}_{x_2}\big(\tau_{ \partial \mathcal{B}_{x_1}(\frac{1}{100}|x_1|)}=\tau_{w}<\tau_{D\cup B(3N)} \big)\mathbb{E}^{D}\big[ \mathcal{H}_{w}(\widehat{\mathcal{C}}_A^{-}\cup D)\big]$ and $\hat{\mathbf{K}}_w:=  \widetilde{\mathbb{P}}_{x_2}\big(\tau_{\widehat{\mathcal{C}}_{A}^{-}}<\tau_{\partial \mathcal{B}_{x_1}(\frac{1}{100}|x_1|)}=\tau_{w}< \tau_{ D\cup B(3N)} \big)\mathcal{H}_{w}(\widehat{\mathcal{C}}_A^{-}\cup D)$. For $\mathbb{K}$, by Lemma \ref{lemma3.3}, (\ref{newineq_3.16}) and the assumption $|x_1-x_2|\ge \frac{1}{20}(|x_1|\vee |x_2|)$, we have 
\begin{equation}\label{newineq_3.19}
\begin{split}
		\mathbb{K} \overset{\text{Lemma}\ \ref*{lemma3.3}}{\gtrsim} & \widetilde{\mathbb{P}}_{x_2}\big(\tau_{ \partial \mathcal{B}_{x_1}(\frac{1}{100}|x_1|)}<\tau_{D\cup B(3N)} \big)  \min_{w\in \partial \mathcal{B}_{x_1}(\frac{1}{100}|x_1|)}\mathbb{P}^{D}\big( w \xleftrightarrow{\ge 0}  A \big)  \\
	\overset{(\ref*{new2.20_1}),(\ref*{newineq_3.16})}{\gtrsim} & |x_2-x_1|^{2-d}|x_1|^{d-2} \mathbb{P}^{D}\big( x_1 \xleftrightarrow{\ge 0}  A \big) 	\\
\overset{|x_1-x_2|\ge \frac{1}{20}(|x_1|\vee |x_2|)}{\gtrsim} & |x_2|^{2-d}|x_1|^{d-2} \mathbb{P}^{D}\big( x_1 \xleftrightarrow{\ge 0}  A \big).
	\end{split}
\end{equation}
Next, we estimate $\mathbb{E}^{D}[\hat{\mathbf{K}}_w]$ for $w\in \partial \mathcal{B}_{x_1}(\frac{1}{100}|x_1|)$. By the union bound, we have 
\begin{equation}\label{new_add_3.31}
		\begin{split}
			& \widetilde{\mathbb{P}}_{x_2}\times \mathbb{P}^{D}\big(\tau_{\widehat{\mathcal{C}}_{A}^{-}}< \tau_{\mathcal{B}_{x_1}(\frac{1}{100}|x_1|)} < \tau_{ D\cup B(3N)}  \big)\\
			\le & \sum\nolimits_{z\in [B(3N)]^c}  \widetilde{\mathbb{P}}_{x_2}\times \mathbb{P}^{D}\big(\tau_{z}<\tau_{\mathcal{B}_{x_1}(\frac{1}{100}|x_1|)}< \tau_{D\cup B(3N)} ,\widetilde{B}_z(1) \xleftrightarrow{\le 0}A  \big)  \\
			\lesssim  &\sum\nolimits_{z \in [B(3N)]^c}  \mathbb{P}^{D}\big(\widetilde{B}_z(1) \xleftrightarrow{\ge 0}A  \big) \cdot \widetilde{\mathbb{P}}_{x_2}(\tau_{z}<\infty) \cdot \widetilde{\mathbb{P}}_{z}(\tau_{\mathcal{B}_{x_1}(\frac{1}{100}|x_1|)}<\infty).
		\end{split}
	\end{equation} 
 Moreover, for any $z \in [B(3N)]^c$, by Lemma \ref{lemma_revise_B1} and (\ref{ineq_compare}), we have 
  \begin{equation*}
 	\begin{split}
 		 \mathbb{P}^{D}\big(\widetilde{B}_z(1) \xleftrightarrow{\ge 0}A  \big) \overset{}{\lesssim }    \mathbb{P}\big(z \xleftrightarrow{\ge 0}A  \big) \overset{A\subset [\widetilde{B}_z(N)]^c}{\lesssim }\theta_d(N)   \overset{(\ref*{one_arm_high})}{\lesssim } N^{-2}.  
 	\end{split}
 \end{equation*}
Meanwhile, by Lemma \ref{lemma_prop_hitting} one has 
\begin{equation*}
	\widetilde{\mathbb{P}}_{x_2}(\tau_{z}<\infty)\asymp |x_2-z|^{2-d},\ \ \widetilde{\mathbb{P}}_{z}(\tau_{\mathcal{B}_{x_1}(\frac{1}{100}|x_1|)}<\infty)\lesssim |z-x_1|^{2-d}|x_1|^{d-2}.
\end{equation*}
Plugging these three estimates into (\ref{new_add_3.31}), we get 
\begin{equation}\label{newineq_3.25}
\begin{split}
		 & \widetilde{\mathbb{P}}_{x_2}\times \mathbb{P}^{D}\big(\tau_{\widehat{\mathcal{C}}_{A}^{-}}< \tau_{\mathcal{B}_{x_1}(\frac{1}{100}|x_1|)} < \tau_{ D\cup B(3N)}  \big)\\
		 	\lesssim &|x_1|^{d-2}N^{-2}\sum\nolimits_{z\in \mathbb{Z}^d} |x_2-z|^{2-d}|z-x_1|^{2-d}\\
		 	\overset{(\ref*{2-d_2-d})}{\lesssim} & |x_1|^{d-2}N^{-2} |x_1-x_2|^{4-d} \overset{|x_1-x_2|\gtrsim |x_2| \gtrsim C_\diamond N}{\lesssim} C_\diamond^{-2} |x_1|^{d-2}|x_2|^{2-d}.
		 	\end{split}
\end{equation}
For any $w\in \partial \mathcal{B}_{x_1}(\frac{1}{100}|x_1|)$, by the FKG inequality (noting that $\mathcal{H}_{w}(\widehat{\mathcal{C}}_A^{-}\cup D)$ is increasing, and that for any $V\subset \widetilde{\mathbb{Z}}^d$, $\{\widehat{\mathcal{C}}_A^{-}\cap V\neq \emptyset\}$ is decreasing), we have 
\begin{equation*}
	\begin{split}
		\mathbb{E}^{D}[\hat{\mathbf{K}}_w]=& \widetilde{\mathbb{E}}_{x_2}\Big[\mathbbm{1}_{\tau_{\partial \mathcal{B}_{x_1}(\frac{1}{100}|x_1|)}=\tau_{w}< \tau_{ D\cup B(3N)}}\mathbb{E}^{D}\big[\mathbbm{1}_{\widehat{\mathcal{C}}_A^{-}\cap \big\{\widetilde{S}_t: 0\le t\le \tau_{\partial \mathcal{B}_{x_1}(\frac{1}{100}|x_1|)} \big\} \neq \emptyset }\mathcal{H}_{w}(\widehat{\mathcal{C}}_A^{-}\cup D)\big]   \Big]\\
			\overset{(\text{FKG})}{\le}&\widetilde{\mathbb{E}}_{x_2}\Big[\mathbbm{1}_{\tau_{\partial \mathcal{B}_{x_1}(\frac{1}{100}|x_1|)}=\tau_{w}< \tau_{ D\cup B(3N)}} \mathbb{P}^{D}\big(\widehat{\mathcal{C}}_A^{-}\cap \big\{\widetilde{S}_t: 0\le t\le \tau_{\partial \mathcal{B}_{x_1}(\frac{1}{100}|x_1|)} \big\} \neq \emptyset \big)\\
			&\ \ \ \ \ \ \ \cdot \mathbb{E}^{D}\big[\mathcal{H}_{w}(\widehat{\mathcal{C}}_A^{-}\cup D)\big] \Big]  \\
			=&\widetilde{\mathbb{P}}_{x_2}\times \mathbb{P}^{D}\big(\tau_{\widehat{\mathcal{C}}_A^{-}}< \tau_{\mathcal{B}_{x_1}(\frac{1}{100}|x_1|)} =\tau_w < \tau_{ D\cup B(3N)}  \big) \cdot  \mathbb{E}^{D}\big[\mathcal{H}_{w}(\widehat{\mathcal{C}}_A^{-}\cup D)\big]\\
		\lesssim & \widetilde{\mathbb{P}}_{x_2}\times \mathbb{P}^{D}\big(\tau_{\widehat{\mathcal{C}}_A^{-}}< \tau_{\mathcal{B}_{x_1}(\frac{1}{100}|x_1|)} =\tau_w < \tau_{ D\cup B(3N)}  \big) \cdot \mathbb{P}^{D}\big( x_1 \xleftrightarrow{\ge 0}  A \big),
	\end{split}
\end{equation*}
where in the last inequality we used Lemma \ref{lemma3.3} and (\ref{PDx_1x_2_case1}). Summing over $w_1\in  \partial \mathcal{B}_{x_1}(\frac{1}{100}|x_1|)$ and using (\ref{newineq_3.25}), we obtain  
\begin{equation}\label{newineq_3.26}
	\sum\nolimits_{w\in \partial \mathcal{B}_{x_1}(\frac{1}{100}|x_1|)}\mathbb{E}^{D}[\hat{\mathbf{K}}_w] \lesssim  C_\diamond^{-2} |x_1|^{d-2}|x_2|^{2-d}\mathbb{P}^{D}(x_1 \xleftrightarrow{\ge 0} A ). 
\end{equation}
By (\ref{new_ineq_3.18}), (\ref{newineq_3.19}) and (\ref{newineq_3.26}), we derive that for some constants $C_\dagger,c_\dagger>0$, 
\begin{equation}
\begin{split}
		\mathbb{P}^{D}\big( x_2 \xleftrightarrow{\ge 0}  A \big) \gtrsim  &\mathbb{P}(x_1 \xleftrightarrow{\ge 0} A )\cdot \big( c_\dagger -C_\dagger C_\diamond^{-2} \big) |x_1|^{d-2}|x_2|^{2-d} \\
		\gtrsim &\mathbb{P}(x_1 \xleftrightarrow{\ge 0} A )\cdot |x_1|^{d-2}|x_2|^{2-d},
\end{split}
\end{equation}
where we ensure the last inequality by taking a sufficiently large $C_\diamond$. By swapping $x_1$ and $x_2$, we obtain (\ref{PDx_1x_2_case2}) and thus complete the proof of this proposition \qed

\subsection{Proof of Proposition \ref{prop_pointtoset_harnack}}\label{subsection_proof_pointtoset_harnack}

The proof of this proposition shares the same spirit as that of Proposition \ref{prop_pointtoset_speed}. Thus, for simplicity we will only sketch the proof instead of providing similar details.

We start from the cases when $3\le d\le 6$. Let $c_*>0$ be a sufficiently small constant. Parallel to (\ref{ineq3.9}), we have 
\begin{equation}\label{newineq_3.27}
	\mathbb{P}^D\big( A \xleftrightarrow{\le 0} \partial B(c_*N^{1-\varsigma_d(N)}) \big)\le \tfrac{1}{4}.  
\end{equation}
Let $\mathsf{F}:=\{A \xleftrightarrow{\le 0} \partial B(c_*N^{1-\varsigma_d(N)})\}^c$. For any $x_1,x_2\in B(\tfrac{1}{2}c_*N^{1-\varsigma_d(N)})$, by using (\ref{newineq_3.27}) and Lemma \ref{lemma3.3} as in (\ref{ineq3.12}), we have 
\begin{equation}\label{newineq_3.29}
	\mathbb{E}^{D}\big[\mathcal{H}_{x_i}(\widehat{\mathcal{C}}_A^{-}\cup D) \cdot \mathbbm{1}_{\mathsf{F}} \big] \asymp \mathbb{P}^D\big( x_i \xleftrightarrow{\ge 0}  A \big), \ \ \forall i\in \{1,2\}. 
\end{equation}
Combined with the fact that $\mathsf{F}\subset \{\mathcal{H}_{x_1}(\widehat{\mathcal{C}}_A^{-}\cup D) \asymp \mathcal{H}_{x_2}(\widehat{\mathcal{C}}_A^{-}\cup D) \}$, it concludes this proposition for $3\le d\le 6$.

Now we focus on the case $d\ge 7$. We take a small constant $c_\diamond>0$ and let $x_1,x_2\in B(c_\diamond N)$. When $|x_1-x_2|\le \frac{1}{10}|x_1|$, the arguments in the proof of (\ref{PDx_1x_2_case1}) also work here. More precisely, we can show that removing $\widehat{\mathcal{C}}_A^{-}\cap \widetilde{B}_{x_1}(\frac{1}{5}|x_1|)$ from the obstacle $\widehat{\mathcal{C}}_A^{-}\cup D\cup \partial B_{x_1}(\frac{1}{2}|x_1|)$ will increase the Green's function evaluated at $(x_1, z_1)$ (where $z_1\in \partial B_{x_1}(\frac{1}{5}|x_1|)$) by only a constant factor, allowing us to apply the last-exit decomposition, Harnack's inequality and Lemma \ref{lemma3.3} to derive 
\begin{equation}\label{newineq_3.30}
	\mathbb{P}^D\big( x_1 \xleftrightarrow{\ge 0}  A \big) \gtrsim  \mathbb{P}^{D}\big( x_2 \xleftrightarrow{\ge 0}  A \big).
\end{equation}
The  only technical difference arises in the estimate on $\mathbb{P}^{D}\big(\exists  y\in F\ \text{such that}\ \widetilde{B}_y(1)\cap \widehat{\mathcal{C}}_A^{-}\neq \emptyset \big)$ for $F\subset B_{x_1}(\frac{1}{2}|x_1|)$ containing at most $2|x_1|^2$ points (see (\ref{new_add3.32})). Instead of controlling $\max\nolimits_{y\in B_{x_1}(\frac{3}{5}|x_1|)}  \mathbb{P}\big(y \xleftrightarrow{\ge 0} \partial B(N)\big)$ using Corollary \ref{coro_subharmonic}, here we only need to use $\mathbb{P}\big(y \xleftrightarrow{\ge 0} \partial B(N)\big)\le \theta_d(\frac{N}{2})\lesssim N^{-2}$ for $y\in B_{x_1}(\frac{3}{5}|x_1|)$ to obtain 
\begin{equation}
	\mathbb{P}^{D}\big(\exists  y\in F\ \text{such that}\ \widetilde{B}_y(1)\cap \widehat{\mathcal{C}}_A^{-}\neq \emptyset \big) \lesssim |x_1|^2N^{-2} \lesssim c_\diamond^2,
\end{equation}
as achieved in (\ref{new_add3.32}). By swapping $x_1$ and $x_2$, we know that it remains to show
\begin{equation}
	\mathbb{P}^D\big( x_1 \xleftrightarrow{\ge 0}  A \big) \asymp \mathbb{P}^{D}\big( x_2 \xleftrightarrow{\ge 0}  A \big),
\end{equation}
for the case when $|x_1-x_2|\ge \frac{1}{20}(|x_1|\vee |x_2|)$. Similar to (\ref{new_ineq_3.18}), we have 
\begin{equation}\label{newineq_3.33}
	\begin{split}
		\mathbb{P}^D\big( x_2 \xleftrightarrow{\ge 0}  A \big) \asymp \mathbb{E}^{D}\big[\mathcal{H}_{x_2}(\widehat{\mathcal{C}}_A^{-}\cup D) \big] \ge \mathbb{M}-  \sum\nolimits_{w\in \partial \mathcal{B}_{x_1}(\frac{1}{100}|x_1|)}\mathbb{E}^{D}[\hat{\mathbf{M}}_w],
	\end{split}
\end{equation}
where $\mathbb{M}:=\sum_{w\in \partial \mathcal{B}_{x_1}(\frac{1}{100}|x_1|)} \widetilde{\mathbb{P}}_{x_2}\big(\tau_{ \partial \mathcal{B}_{x_1}(\frac{1}{100}|x_1|)}=\tau_{w}<  \tau_{D\cup \partial B(2c_\diamond N)} \big)\mathbb{E}^{D}\big[ \mathcal{H}_{w}(\widehat{\mathcal{C}}_A^{-}\cup D)\big]$ and $\hat{\mathbf{M}}_w:=  \widetilde{\mathbb{P}}_{x_2}\big(\tau_{\widehat{\mathcal{C}}_{A}^{-}}<\tau_{\partial \mathcal{B}_{x_1}(\frac{1}{100}|x_1|)}=\tau_{w}<  \tau_{ D\cup \partial B(2c_\diamond N)} \big)\mathcal{H}_{w}(\widehat{\mathcal{C}}_A^{-}\cup D)$. For any $w\in \partial \mathcal{B}_{x_1}(\frac{1}{100}|x_1|)$, by Lemma \ref{lemma3.3} and (\ref{newineq_3.30}), we have 
\begin{equation}
	\begin{split}
		\mathbb{E}^{D}\big[ \mathcal{H}_{w}(\widehat{\mathcal{C}}_A^{-}\cup D)\big] \overset{\text{Lemma}\ \ref*{lemma3.3}}{\asymp}  \mathbb{P}^{D}(w\xleftrightarrow{\ge 0} A ) \overset{(\ref*{newineq_3.30})}{\gtrsim } \mathbb{P}^{D}(x_1 \xleftrightarrow{\ge 0} A ),
	\end{split}
\end{equation}
which further implies that 
\begin{equation}\label{newineq_3.34}
	\begin{split}
		\mathbb{M}\gtrsim  \widetilde{\mathbb{P}}_{x_2}\big(\tau_{ \partial \mathcal{B}_{x_1}(\frac{1}{100}|x_1|)}<\tau_{D\cup \partial B(2c_\diamond N)} \big) \mathbb{P}^{D}(x_1 \xleftrightarrow{\ge 0} A )\gtrsim \mathbb{P}^{D}\big( x_1 \xleftrightarrow{\ge 0}  A \big).
	\end{split}
\end{equation}
Here in the second inequality we used the invariance principle (and note that the implicit constant does not depend on $c_\diamond$). For $\hat{\mathbf{M}}_w$, similar to (\ref{new_add_3.31}), we have
\begin{equation}\label{revise3.29}
	\begin{split}
		&\widetilde{\mathbb{P}}_{x_2}\times \mathbb{P}^{D}\big(\tau_{\widehat{\mathcal{C}}_{A}^{-}}< \tau_{\mathcal{B}_{x_1}(\frac{1}{100}|x_1|)} < \tau_{ D\cup \partial B(2c_\diamond N)}  \big)\\
		\le &\sum\nolimits_{w\in B(2c_\diamond N)}  \mathbb{P}^{D}\big(\widetilde{B}_w(1) \xleftrightarrow{\ge 0}A  \big) \cdot \widetilde{\mathbb{P}}_{x_2}(\tau_{w}<\infty) \cdot \widetilde{\mathbb{P}}_{w}(\tau_{\mathcal{B}_{x_1}(\frac{1}{100}|x_1|)}<\infty)\\
		\lesssim & \sum\nolimits_{w\in B(2c_\diamond N)} N^{-2}\cdot |x_2-w|^{2-d}  \overset{(\ref*{|x|-a_aneqd})}{\lesssim } c_\diamond^2,
	\end{split}
\end{equation}
where in the second inequality we used Lemma \ref{lemma_revise_B1} and $A\subset [\widetilde{B}(N)]^c$ to obtain that $\mathbb{P}^{D}\big(\widetilde{B}_w(1) \xleftrightarrow{\ge 0}A  \big) \lesssim N^{-2}$. Similar to (\ref{newineq_3.26}), this together with (\ref{newineq_3.30}) implies
\begin{equation}\label{newineq_3.36}
	\sum\nolimits_{w\in \partial \mathcal{B}_{x_1}(\frac{1}{100}|x_1|)}\mathbb{E}^{D}[\hat{\mathbf{M}}_w] \lesssim c_\diamond^{2}\mathbb{P}^{D}(x_1 \xleftrightarrow{\ge 0} A ).
\end{equation}
Plugging (\ref{newineq_3.34}) and (\ref{newineq_3.36}) into (\ref{newineq_3.33}), we get
\begin{equation}
	\mathbb{P}^D\big( x_2 \xleftrightarrow{\ge 0}  A \big) \gtrsim \mathbb{P}^D\big( x_1 \xleftrightarrow{\ge 0}  A \big). 
\end{equation}
By exchanging the roles of $x_1$ and $x_2$, we conclude this proposition.  \qed

\section{Stability of boundary-to-set connecting probabilities}\label{section_boundary_to_set_stable}

In this section, we present the proof of Proposition \ref{prop_boundtoset_zero}. This proof mainly relies on the stochastic domination between crossing paths and loops, as stated in Lemma \ref{lemma_sto_dom}. To get this domination, we present the following lemma as preparation.
\begin{lemma}\label{lemma_paths}
For any $d\ge 3$ and $\delta\in (0,1)$, there exist $0<c<c'<c''<1$ depending on $d$ and $\delta$ such that the following holds. Arbitrarily fix $N_{-1}>N_{1}\ge 1$, $x_1\in \partial B(N_{1})$, $x_{-1}\in \partial B(N_{-1})$ and $D\subset \widetilde{B}(cN_{1})\cup [\widetilde{B}(c^{-1}N_{-1})]^c$. For each $j\in \{1,-1\}$, we arbitrarily take $y_{j}\in \partial B((c'')^{j}N_{j})$ and $y_{j}'\in \partial B((10c'')^{j}N_{j})$ and consider the following two random paths:
\begin{equation}\label{newfinaluse4.1}
	\widetilde{\eta}_{1,j}\sim \widetilde{\mathbb{P}}_{x_{j}}\big(\{\widetilde{S}_t\}_{0\le t\le \tau_{\partial B(N_{-j})}}\in \cdot \mid \tau_{\partial B(N_{-j})}=\tau_{x_{-j}}<\tau_{D}\big),
\end{equation} 
\begin{equation}\label{newfinaluse4.2}
	\widetilde{\eta}_{2,j} \sim \widetilde{\mathbb{P}}_{y_{j}}\big(\{\widetilde{S}_t\}_{0\le t\le \tau_{\partial B((10c'')^{j}N_{j})}}\in \cdot \mid \tau_{\partial B((10c'')^{j}N_{j})}=\tau_{y_{j}'}<\tau_{D}\big).
\end{equation}
For each $i\in \{1,2\}$ and $j\in \{1,-1\}$, let $\mathbf{z}_{i,j}^{\mathrm{F}}$ be the position where $\widetilde{\eta}_{i,j}$ first hits $B((c')^{j}N_{j})$ and let $\mathbf{z}_{i,j}^{\mathrm{B}}$ be the last position where $\widetilde{\eta}_{i,j}$ stays in $B((c')^{j}N_j)$ (when $\widetilde{\eta}_{i,j}$ does not reach $B((c')^{j}N_j)$, we set $\mathbf{z}_{i,j}^{\mathrm{F}}=\mathbf{z}_{i,j}^{\mathrm{B}}=\emptyset$). Then for any $v,w\in \partial B((c')^{j}N_j)$, 		
		\begin{equation}\label{2.29}
			\mathbb{P}_{\widetilde{\eta}_{1,j}}(\mathbf{z}_{1,j}^{\mathrm{F}}=v, \mathbf{z}_{1,j}^{\mathrm{B}}=w) \le \delta \mathbb{P}_{\widetilde{\eta}_{2,j}}(\mathbf{z}_{2,j}^{\mathrm{F}}=v, \mathbf{z}_{2,j}^{\mathrm{B}}=w).
					\end{equation}    
\end{lemma}
\begin{proof}
Since the proofs for $j=1$ and $j=-1$ are carried out using the same arguments, we only show the details for the case when $j=1$. By the strong Markov property and the last-exit decomposition, we have 
	\begin{equation}\label{2.33}
		\begin{split}
			&\mathbb{P}_{\widetilde{\eta}_{1,1}}(\mathbf{z}_{1,1}^{\mathrm{F}}=v, \mathbf{z}_{1,1}^{\mathrm{B}}=w)= \widetilde{\mathbb{P}}_{x_1}\big(\tau_{B(c'N_1)}=\tau_v<\tau_D \big) \widetilde{G}_D(v,w)\cdot \tfrac{1}{2d}\\
			&\ \ \ \ \ \ \ \ \ \ \ \ \ \ \ \ \ \ \ \ \ \ \ \ \cdot\sum_{w'\in [B(c'N_1)]^c:\{w,w'\}\in \mathbb{L}^d} \frac{\widetilde{\mathbb{P}}_{w'}\big(\tau_{\partial B(N_{-1})}=\tau_{x_{-1}}<\tau_{D\cup B(c'N_1)}\big)}{\widetilde{\mathbb{P}}_{x_1}\big(\tau_{\partial B(N_{-1})}=\tau_{x_{-1}}<\tau_{D}\big)},
		\end{split}
	\end{equation}
	\begin{equation}
		\begin{split}
			&\mathbb{P}_{\widetilde{\eta}_{2,1}}(\mathbf{z}_{2,1}^{\mathrm{F}}=v, \mathbf{z}_{2,1}^{\mathrm{B}}=w)= \widetilde{\mathbb{P}}_{y_1}\big(\tau_{B(c'N_1)}=\tau_v<\tau_D \big) \widetilde{G}_D(v,w)\cdot \tfrac{1}{2d}\\
			&\ \ \ \ \ \ \ \ \ \ \ \ \ \ \ \ \ \ \ \ \ \  \cdot  \sum_{w'\in [B(c'N_1)]^c:\{w,w'\}\in \mathbb{L}^d}  \frac{\widetilde{\mathbb{P}}_{w'}\big(\tau_{\partial B(10c''N_1)}=\tau_{y_1'}<\tau_{D\cup B(c'N_1)}\big)}{\widetilde{\mathbb{P}}_{y_1}\big(\tau_{\partial B(10c''N_1)}=\tau_{y_1'}<\tau_{D}\big)}.
		\end{split}
	\end{equation}
	Firstly, Lemma \ref{lemma_compare_hm} implies that 
	\begin{equation}
		\begin{split}
			|x_1|^{d-2}\widetilde{\mathbb{P}}_{x_1}\big(\tau_{B(c'N_1)}=\tau_v<\tau_D \big) \asymp |y_1|^{d-2}\widetilde{\mathbb{P}}_{y_1}\big(\tau_{B(c'N_1)}=\tau_v<\tau_D \big).
		\end{split}
	\end{equation}
		Meanwhile, for any $w'$ involved in the sum on the right-hand side of (\ref{2.33}), by the strong Markov property, Lemma \ref{lemma_new2.1} and Harnack's inequality, we have (here we require that $c''>10c'$)
	\begin{equation}
		\begin{split}
			&\widetilde{\mathbb{P}}_{w'}\big(\tau_{\partial B(N_{-1})}=\tau_{x_{-1}}<\tau_{D\cup B(c'N_1)}\big)\\
			= & \sum_{z\in \partial B(4c' N_1)}\widetilde{\mathbb{P}}_{w'}\big(\tau_{\partial B(4c' N_1)}=\tau_z<\tau_{B(c'N_1)}\big) \widetilde{\mathbb{P}}_{z}\big( \tau_{\partial B(N_{-1})}=\tau_{x_{-1}}<\tau_{D\cup B(c'N_1)} \big)\\
			\asymp   &\sum_{z\in \partial B(4c' N_1)}\widetilde{\mathbb{P}}_{w'}\big(\tau_{\partial B(4c' N_1)}=\tau_z<\tau_{B(c'N_1)}\big) \widetilde{\mathbb{P}}_{z}\big( \tau_{\partial B(N_{-1})}=\tau_{x_{-1}}<\tau_{D} \big)\\
			\asymp &  \widetilde{\mathbb{P}}_{w'}\big(\tau_{\partial B(4c' N_1)}<\tau_{B(c'N_1)}\big) \widetilde{\mathbb{P}}_{x_1}\big( \tau_{\partial B(N_{-1})}=\tau_{x_{-1}}<\tau_{D} \big).
		\end{split}
	\end{equation}
	For the same reasons, we also have 
	\begin{equation}\label{2.37}
	\begin{split}
		&\widetilde{\mathbb{P}}_{w'}\big(\tau_{\partial B(10c''N_1)}=\tau_{y_1'}<\tau_{D\cup B(c'N_1)}\big)\\
		\asymp & \widetilde{\mathbb{P}}_{w'}\big(\tau_{\partial B(4c' N_1)}<\tau_{B(c'N_1)}\big) \widetilde{\mathbb{P}}_{y_1}\big( \tau_{\partial B(10c''N_1)}=\tau_{y_1'}<\tau_{D} \big).
	\end{split}	
	\end{equation}
Putting (\ref{2.33})--(\ref{2.37}) together, we get 
\begin{equation}\label{newfinaluse_4.9}
|x_1|^{d-2}	\mathbb{P}_{\widetilde{\eta}_{1,1}}(\mathbf{z}_{1,1}^{\mathrm{F}}=v, \mathbf{z}_{1,1}^{\mathrm{B}}=w) \asymp |y_1|^{d-2}	\mathbb{P}_{\widetilde{\eta}_{2,1}}(\mathbf{z}_{2,1}^{\mathrm{F}}=v, \mathbf{z}_{2,1}^{\mathrm{B}}=w).
\end{equation}
	By choosing $c''=c''(d,\delta)$ to be sufficiently small and noting that $|y_1|\lesssim c''|x_1|$, we obtain (\ref{2.29}) for $j=1$ from (\ref{newfinaluse_4.9}).
	\end{proof}


Now we are ready to establish the aforementioned stochastic domination.

\begin{lemma}\label{lemma_sto_dom}
	Keep the notations in Lemma \ref{lemma_paths}, where we take a sufficiently small $\delta = \delta(d)>0$. For any $n\ge 1$ and $j\in \{1,-1\}$, let $\mathfrak{L}_j(n)$ be the union of ranges of loops in $\widetilde{\mathcal{L}}^{D}_{1/2}$ intersecting both $\partial B(n)$ and $\partial B(10^{j}n)$. Then for each $j\in \{1,-1\}$, $\mathrm{ran}(\widetilde{\eta}_{1,j})\cap \widetilde{B}((c')^jN_j)$ is stochastically dominated by $\mathfrak{L}_j((c'')^jN_j)\cap \widetilde{B}((c')^jN_j)$.	
\end{lemma}
\begin{proof}
Since the proofs for $j=1$ and $j=-1$ are similar, we only provide the details for $j=1$. Let $\hat{\eta}_{1,1}$ be the subpath of $\widetilde{\eta}_{1,1}$ (defined in (\ref{newfinaluse4.1})) from the first time hitting $B(c''N_1)$ to the last time staying in $B(c''N_1)$. Note that $\mathrm{ran}(\widetilde{\eta}_{1,1})\cap \widetilde{B}(c'N_1)=\mathrm{ran}(\hat{\eta}_{1,1})\cap \widetilde{B}(c'N_1)$. Let $(\mathbf{v}_1,\mathbf{w}_1)$ be a random tuple satisfying that for any $v,w\in \{\emptyset\}\cup \partial B(c'N_1)$, 
	 \begin{equation}
	 	\mathbb{P}\big((\mathbf{v}_1,\mathbf{w}_1)=(v,w)\big) =\mathbb{P}_{\widetilde{\eta}_1}(\mathbf{z}_{1,1}^{\mathrm{F}}=v, \mathbf{z}_{1,1}^{\mathrm{B}}=w).
	 		 \end{equation}
	 Let $\hat{\eta}_{v,w}$ be a Brownian bridge on $\widetilde{\mathbb{Z}}^d\setminus D$ from $v$ to $w$, independent of $(\mathbf{v}_1, \mathbf{w}_1)$ (we set $\hat{\eta}_{\emptyset,\emptyset}:=\emptyset$). By the strong Markov property, we know that $\mathrm{ran}(\hat{\eta}_{1,1})$ has the same distribution as 
	 	 \begin{equation}\label{2.55}
	 	\sum\nolimits_{v,w\in \{\emptyset\}\cup \partial B(c'N_1)} \mathbbm{1}_{(\mathbf{v}_1,\mathbf{w}_1)=(v,w)}\cdot  \hat{\eta}_{v,w}. 
	 \end{equation}
	 Meanwhile, if $\mathfrak{L}_1(c''N_1)$ includes at least one loop $\widetilde{\ell}$ (which occurs with probability at least $c_\dagger$ for some constant $c_\dagger$; see \cite[Lemma 2.1]{cai2024high}), we consider one of the backward crossing paths of $\widetilde{\ell}$ from $\partial B(10c''N_1)$ to $\partial B(c''N_1)$ as the path $\widetilde{\eta}_{2,1}$ defined in (\ref{newfinaluse4.2}). Note that $\mathrm{ran}(\widetilde{\eta}_{2,1})\subset \mathrm{ran}(\widetilde{\ell})$. Let $\hat{\eta}_{2,1}$ be the counterpart of $\hat{\eta}_{1,1}$, obtained by replacing $\widetilde{\eta}_{1,1}$ with $\widetilde{\eta}_{2,1}$. Similar to (\ref{2.55}), we know that $\hat{\eta}_{2,1}$ has the same distribution as 
	 \begin{equation}\label{2.56}
	 	\sum\nolimits_{v,w\in \{\emptyset\}\cup \partial B(c'N)} \mathbbm{1}_{(\mathbf{v}_2,\mathbf{w}_2)=(v,w)}\cdot  \hat{\eta}_{v,w}, 
	 \end{equation} 
	 where $(\mathbf{v}_2,\mathbf{w}_2)$ is a random tuple (independent of $\hat{\eta}_{v, w}$) such that  
	 \begin{equation}
	 	 	\mathbb{P}\big((\mathbf{v}_2,\mathbf{w}_2)=(v,w)\big) =\mathbb{P}_{\widetilde{\eta}_{2,1}}(\mathbf{z}_2^{\mathrm{F}}=v, \mathbf{z}_2^{\mathrm{B}}=w), \ \ \forall v,w\in \{\emptyset\}\cup \partial B(c'N_1).
	 \end{equation}
	By taking $\delta<c_\dagger$ in Lemma \ref{lemma_paths} (with $c,c'$ and $c''$ selected accordingly), we know that for every $v,w\in \partial B(c'N_1)$, the probability of having $\hat{\eta}_{v,w}$ in $\widetilde{\eta}_{1,1}$ (i.e., $\mathbb{P}_{\widetilde{\eta}_{1,1}}(\mathbf{z}_{1,1}^{\mathrm{F}}=v, \mathbf{z}_{1,1}^{\mathrm{B}}=w)$) is smaller than the probability of having $\hat{\eta}_{v,w}$ in $\mathfrak{L}_1(c''N_1)$ (which is at least $c_\dagger\mathbb{P}_{\widetilde{\eta}_{2,1}}(\mathbf{z}_{2,1}^{\mathrm{F}}=v, \mathbf{z}_{2,1}^{\mathrm{B}}=w)$). In conclusion, $\mathrm{ran}(\widetilde{\eta}_{1,1})\cap \widetilde{B}(c'N_1)$ is stochastically dominated by $\mathfrak{L}_1(c''N_1)\cap \widetilde{B}(c'N_1)$.
\end{proof}

Subsequently, we present an application of Lemma \ref{lemma_sto_dom}, which will be used repeatedly. To simplify the expression, we first introduce some notations as follows. For any $D\subset \widetilde{\mathbb{Z}}^d$ and $m\ge n\ge 1$, let 
\begin{equation}\label{lateradd_4.14}
	\widetilde{\mathcal{L}}_{1/2}^D[n,m]:= \widetilde{\mathcal{L}}_{1/2}^D\cdot \mathbbm{1}_{\mathrm{ran}(\widetilde{\ell})\cap \partial B(n)\neq \emptyset, \mathrm{ran}(\widetilde{\ell})\cap \partial B(m)\neq \emptyset}.
\end{equation}
We also denote $\mathfrak{L}^{D}[n,m]:= \cup \widetilde{\mathcal{L}}_{1/2}^D[n,m]$ and $\overline{\mathfrak{L}}^{D}[n,m]:=\cup (\widetilde{\mathcal{L}}_{1/2}^D-\widetilde{\mathcal{L}}_{1/2}^D[n,m]) $. Recall $\mathsf{G}_l^{D}$ in (\ref{def_Gl_crossing}) and $\Omega_l$ in the paragraph before (\ref{3.36}). For simplicity, let
 \begin{equation}
	\mathsf{G}_l^{D}[n,m]:=\mathsf{G}_l^{D}(\partial B(m),\partial B(n))\ \ \text{and}\ \ \Omega_l[n,m]:=\Omega_l(\partial B(m),\partial B(n)).
\end{equation}
Also recall $	\widehat{\mathsf{G}}^{D}$ in (\ref{3.36}). For any $\omega_l\in \Omega_l[n,m]$, we denote 
\begin{equation}\label{def_widehat_G}
		\widehat{\mathsf{G}}^{D}[\omega_l;n,m]:=	\widehat{\mathsf{G}}^{D}(\omega_l;\partial B(m),\partial B(n)).
\end{equation}

We use $A_0\xleftrightarrow{} A_1 \xleftrightarrow{} ... \xleftrightarrow{} A_k$ to represent the event $\cap_{0\le i<j\le k}\{A_i\xleftrightarrow{} A_{j}\}$. We also denote $\{A_0\xleftrightarrow{} A_1,...,A_k\}:=\cap_{1\le i\le k}\{A_0\xleftrightarrow{} A_i\}$. Note that the former event always implies the later one, while the converse may not hold if $A_0$ contains more than one point.

\begin{lemma}\label{lemma_prepare_loop_decompose}
For any $d\ge 3$, there exists $\Cl\label{const_prepare_loop_decompose}>0$ such that the following holds. Let $\mathbf{B}_1:=[\widetilde{B}(10\Cref{const_prepare_loop_decompose}M)]^c$, $\mathscr{B}_1:= \partial B(\Cref{const_prepare_loop_decompose}M)$, $\mathbf{B}_2:=\widetilde{B}(\frac{N}{10\Cref{const_prepare_loop_decompose}})$ and $\mathscr{B}_2:= \partial B(\frac{M}{\Cref{const_prepare_loop_decompose}})$. Then for any $M> N\ge 1$, $j\in \{1,2\}$, $k\in \mathbb{N}^+$, $A_1,...,A_k,D\subset \mathbf{B}_j$ and $\omega_l\in \Omega_l[N,M]$, 
	\begin{equation}\label{newineq_3.4}
	\begin{split}
			&\mathbb{P}\Big( A_1\xleftrightarrow{\overline{\mathfrak{L}}^{D}[N,M]} ... \xleftrightarrow{\overline{\mathfrak{L}}^{D}[N,M]}   A_k \xleftrightarrow{\overline{\mathfrak{L}}^{D}[N,M]} \mathfrak{L}^{D}[N,M] \cup \mathscr{B}_j \mid \widehat{\mathsf{G}}^{D}[\omega_l; N,M] \Big)\\
		\le &	 (l+1)\mathbb{P}\big( A_1\xleftrightarrow{(D)} ... \xleftrightarrow{(D)} A_k \xleftrightarrow{(D)} \mathscr{B}_j \big).
	\end{split}
	\end{equation} 
\end{lemma}	
\begin{proof}
	We only provide the proof details for $j=1$ since the one for $j=2$ follows similarly. Let $\mathcal{C}_*:= \{v\in \widetilde{\mathbb{Z}}^d: v\xleftrightarrow{\overline{\mathfrak{L}}^{D}[N,M]}  A_1 \}$. When the event on the left-hand side of (\ref{newineq_3.4}) happens, we know that $\mathcal{C}_*$ intersects each of $A_2,...,A_k$. In addition, if $\mathcal{C}_*\cap \partial B(\Cref{const_prepare_loop_decompose} M)=\emptyset$, then $\mathcal{C}_*$ only consists of loops in $\widetilde{\mathcal{L}}^{D\cup B(\Cref{const_prepare_loop_decompose} M)}_{1/2}$, and thus $\mathcal{C}_*$ must intersect $\mathrm{ran}(\widetilde{\eta}^{\mathrm{F}}_i)\cap [\widetilde{B}(\Cref{const_prepare_loop_decompose} M)]^c$ for some $1\le i\le l$ (where $\{\widetilde{\eta}^{\mathrm{F}}_i\}_{i=1}^{l}$ are forward crossing paths of loops in $\widetilde{\mathcal{L}}_{1/2}^{D}$ from $\partial B(M)$ to $\partial B(N)$). Thus, the left-hand side of (\ref{newineq_3.4}) is bounded from above by     
	\begin{equation*}
		\begin{split}
			&\mathbb{P}\big( A_1\xleftrightarrow{\overline{\mathfrak{L}}^{D}[N,M]} ... \xleftrightarrow{\overline{\mathfrak{L}}^{D}[N,M]} A_k \xleftrightarrow{\overline{\mathfrak{L}}^{D}[N,M]}   \partial B(\Cref{const_prepare_loop_decompose} M) \big)\\
			&+\sum_{1\le i\le l} \mathbb{P}\Big( A_1\xleftrightarrow{(D\cup  B(\Cref{const_prepare_loop_decompose} M))} ... \ A_k \xleftrightarrow{(D\cup  B(\Cref{const_prepare_loop_decompose}M)) } \mathrm{ran}(\widetilde{\eta}^{\mathrm{F}}_i)\cap [\widetilde{B}(\Cref{const_prepare_loop_decompose} M)]^c \mid \widehat{\mathsf{G}}^{D}[\omega_l; N,M] \Big). 
		\end{split}
	\end{equation*} 
	Clearly, the first term is upper-bounded by $\mathbb{P}\big( A_1\xleftrightarrow{(D)} ... \xleftrightarrow{(D)} A_k \xleftrightarrow{(D)}   \partial B(\Cref{const_prepare_loop_decompose} M) \big)$. Moreover, for each $1\le i\le l$, Lemma \ref{lemma_sto_dom} shows that $\mathrm{ran}(\widetilde{\eta}^{\mathrm{F}}_i)\cap [\widetilde{B}(\Cref{const_prepare_loop_decompose} M)]^c$ is stochastically dominated by $\mathfrak{L}^{D}[\frac{C'M}{10},C'M]\cap [\widetilde{B}(\Cref{const_prepare_loop_decompose} M)]^c$ for some constant $C'\in (0,\Cref{const_prepare_loop_decompose})$. Consequently, since $\widetilde{\mathcal{L}}_{1/2}^{D\cup B(\Cref{const_prepare_loop_decompose} M)}+\widetilde{\mathcal{L}}_{1/2}^{D}[\frac{C'M}{10},C'M] \le \widetilde{\mathcal{L}}_{1/2}^{D}$, we get
	\begin{equation*}
		\begin{split}
			&\mathbb{P}\Big( A_1\xleftrightarrow{(D\cup B(\Cref{const_prepare_loop_decompose} M))} ... \ A_k \xleftrightarrow{(D\cup  B(\Cref{const_prepare_loop_decompose} M)) } \mathrm{ran}(\widetilde{\eta}^{\mathrm{F}}_i)\cap [\widetilde{B}(\Cref{const_prepare_loop_decompose} M)]^c \mid \widehat{\mathsf{G}}^{D}[\omega_l; N,M] \Big)\\
			\le & \mathbb{P}\big( A_1\xleftrightarrow{(D)} ... \xleftrightarrow{(D)} A_k \xleftrightarrow{(D)} \partial B(\Cref{const_prepare_loop_decompose} M) \big).
		\end{split}
	\end{equation*}
	To sum up, we conclude the proof for $j=1$. 
\end{proof}	
	
	We are now ready to prove Proposition \ref{prop_boundtoset_zero}.

\begin{proof}[Proof of Proposition \ref{prop_boundtoset_zero}]
	Since the two items are very similar, we will only prove Item (a) and omit the proof of Item (b). For Item (a), by (\ref{ineq_compare}), it suffices to show 
\begin{equation}
	\mathbb{P}^{D\cup \partial B(\Cref{prop_boundtoset_zero1}M)}\big(A\xleftrightarrow{\ge 0} \partial B(M) \big) \gtrsim \mathbb{P}^{D}\big(A\xleftrightarrow{\ge 0} \partial B(M) \big). 
\end{equation}
To get this, by the isomorphism theorem, we only need to establish
\begin{equation}\label{lateadd_3.39}
	\begin{split}
\mathbb{P}(\mathsf{A}_M):= & \mathbb{P}\big(\{A\xleftrightarrow{(D)} \partial B(M)\}\cap \{A\xleftrightarrow{(D\cup \partial B(\Cref{prop_boundtoset_zero1}M))} \partial B(M)\}^c \big) \\
	\le & \tfrac{1}{2} \mathbb{P}\big(A\xleftrightarrow{(D)} \partial B(M) \big).
	\end{split}
\end{equation}
Let $\mathcal{C}^{M}(A):= \{v\in \widetilde{\mathbb{Z}}^d:v\xleftrightarrow{(D\cup \partial B(M))} A\}$. When $\mathsf{A}_M$ happens, there must be a loop $\widetilde{\ell}_*\in \widetilde{\mathcal{L}}_{1/2}^{D}$ intersecting both $\partial B(\Cref{prop_boundtoset_zero1}M)$ and $\mathcal{C}^{M}(A)$. Since $\mathcal{C}^{M}(A)\subset \widetilde{B}(M)$, we know that $\widetilde{\ell}_*$ must cross the annulus $B(\Cref{prop_boundtoset_zero1}^{\frac{2}{3}}M)\setminus B(\Cref{prop_boundtoset_zero1}^{\frac{1}{3}}M)$. Note that any loop included in $\mathcal{C}^{M}(A)$ does not cross this annulus. By these observations, we have 
\begin{equation*}
\begin{split}
	\mathbb{P}(\mathsf{A}_M) \le & \mathbb{P}\big( A\xleftrightarrow{\overline{\mathfrak{L}}^{D}[\Cref{prop_boundtoset_zero1}^{\frac{1}{3}}M,\Cref{prop_boundtoset_zero1}^{\frac{2}{3}}M]  } \mathfrak{L}^{D}[\Cref{prop_boundtoset_zero1}^{\frac{1}{3}}M,\Cref{prop_boundtoset_zero1}^{\frac{2}{3}}M] \cup\partial B(M) , \mathfrak{L}^{D}[\Cref{prop_boundtoset_zero1}^{\frac{1}{3}}M,\Cref{prop_boundtoset_zero1}^{\frac{2}{3}}M] \neq \emptyset \big)\\
	=&\sum\nolimits_{l\ge 1}\sum\nolimits_{w_l\in \Omega_l[\Cref{prop_boundtoset_zero1}^{\frac{1}{3}}M,\Cref{prop_boundtoset_zero1}^{\frac{2}{3}}M]} \mathbb{P}\big( \widehat{\mathsf{G}}^{D}[w_l; \Cref{prop_boundtoset_zero1}^{\frac{1}{3}}M,\Cref{prop_boundtoset_zero1}^{\frac{2}{3}}M ] \big) \\
	&\cdot  \mathbb{P}\big(A\xleftrightarrow{\overline{\mathfrak{L}}^{D}[\Cref{prop_boundtoset_zero1}^{\frac{1}{3}}M,\Cref{prop_boundtoset_zero1}^{\frac{2}{3}}M]  } \mathfrak{L}^{D}[\Cref{prop_boundtoset_zero1}^{\frac{1}{3}}M,\Cref{prop_boundtoset_zero1}^{\frac{2}{3}}M] \cup\partial B(M)   \mid \widehat{\mathsf{G}}^{D}[w_l; \Cref{prop_boundtoset_zero1}^{\frac{1}{3}}M,\Cref{prop_boundtoset_zero1}^{\frac{2}{3}}M ] \big)\\
	\overset{\text{Lemma}\ \ref*{lemma_prepare_loop_decompose}}{\lesssim }  &\mathbb{P}\big(A\xleftrightarrow{(D)} \partial B(M)\big) \sum\nolimits_{l\ge 1}(l+1)\cdot  \mathbb{P}\big( \mathsf{G}_l^{D}[\Cref{prop_boundtoset_zero1}^{\frac{1}{3}}M,\Cref{prop_boundtoset_zero1}^{\frac{2}{3}}M] \big). 
\end{split}
\end{equation*} 
Moreover, by Lemma \ref{lemma_Gl}, the sum on the right-hand side is at most 
\begin{equation*}
	\begin{split}
		\sum\nolimits_{l\ge 1} (l +1 )\cdot (C\Cref{prop_boundtoset_zero1}^{\frac{1}{3}})^{(2-d)l} \le \tfrac{1}{2},
		\end{split}
\end{equation*}
where we require $\Cref{prop_boundtoset_zero1}$ to be sufficiently large to ensure the last inequality. By these two estimates, we obtain (\ref{lateadd_3.39}), thereby proving this proposition for Item (a). 
\end{proof}

\section{Point-to-set versus boundary-to-set connecting probabilities}\label{section_relation}

The aim of this section is to prove Proposition \ref{prop_relation_pointandboundary}. Specifically, in Section \ref{subsection_lower_prop_relation} we establish the lower bound in Proposition \ref{prop_relation_pointandboundary} with slightly weaker conditions. The proof is mainly based on an adaptation of \cite[Proposition 3.1]{cai2024one}. Subsequently, in Section \ref{subsection_loop_cluster_decomposition} we present a decomposition for loop clusters, which leads to a useful upper bound on the connecting probability between general sets (see Lemma \ref{lemma_decompose_moresets}). Based on this decomposition, in Section \ref{subsection_upper_prop_relation} we prove the upper bound in Proposition \ref{prop_relation_pointandboundary}.

\subsection{Lower bounds in Proposition \ref{prop_relation_pointandboundary}}\label{subsection_lower_prop_relation}

\textbf{(a)} We first prove the lower bound in (\ref{ineq_relation_pointandboundary}) with the condition $M\ge \Cref{const_relation_pointandboundary1}N^{1+\varsigma_d(N)}$ (instead of $M\ge \Cref{const_relation_pointandboundary1}N^{(1\boxdot \frac{d-4}{2} )+\varsigma_d(N)}$). Similar to $\mathcal{H}_1$ in Lemma \ref{lemma_hat_H}, we define 
\begin{equation}\label{finish4.1}
	\mathcal{H}^{\mathrm{in}}:= |\partial \mathcal{B}(d^{-2}M)|^{-1}\sum\nolimits_{z\in \partial \mathcal{B}(d^{-2}M) } \mathcal{H}_z(\widehat{\mathcal{C}}_A^{-},\widehat{\mathcal{C}}_A^{-} \cup D). 
\end{equation}
By Lemma \ref{lemma_upper_boundarytoset} (with $j=1$) and Lemma \ref
{lemma3.3}, we have 
\begin{equation}\label{finish4.2}
\begin{split}
	\mathbb{P}^{D}\big( A\xleftrightarrow{\ge 0} \partial B(M) \big) =& \mathbb{E}^{D}\big[ \mathbb{P}\big(\widehat{\mathcal{C}}_A^{-} \xleftrightarrow{\ge 0 } \partial B(M) \mid \mathcal{F}_{\widehat{\mathcal{C}}_A^{-}}\big) \big] \\
\overset{\text{Lemma}\ \ref*{lemma_upper_boundarytoset}}{\lesssim}  & M^{d-2}\theta_d(M/4) \mathbb{E} \big[ \mathcal{H}^{\mathrm{in}} \big] \\
\overset{\text{Lemma}\ \ref*{lemma3.3}}{\lesssim }  & M^{d-2}\theta_d(M/4) \cdot \frac{\sum\nolimits_{z\in \partial \mathcal{B}(d^{-2}M) } \mathbb{P}^{D}(A\xleftrightarrow{\ge 0} z)}{ |\partial \mathcal{B}(d^{-2}M)|}. 
\end{split}
	\end{equation}
Combined with $\theta_d(M/4)\lesssim  M^{-[(\frac{d}{2}-1)\boxdot  2]+\varsigma_d(M)}$ (by (\ref{one_arm_low})--(\ref{one_arm_high})) and $\mathbb{P}^{D}(A\xleftrightarrow{\ge 0} z)\asymp \mathbb{P}^{D}(A\xleftrightarrow{\ge 0} x)$ for all $z\in  \partial \mathcal{B}(d^{-2}M)$ and $x\in \partial B(M)$ (by Proposition \ref{prop_pointtoset_speed}), it concludes the lower bound in (\ref{ineq_relation_pointandboundary}).

\textbf{(b)} Likewise, we prove the lower bound in (\ref{ineq_relation_pointandboundary}) with $M\le \cref{const_relation_pointandboundary2}N^{1-\varsigma_d(N)}$ (instead of $M\le\cref{const_relation_pointandboundary2}N^{(1 \boxdot \frac{2}{d-4} )-\varsigma_d(N)}$). The proof is similar to that of Item (a). As in (\ref{finish4.1}), let
\begin{equation}\label{newineq_4.2_b}
	\mathcal{H}^{\mathrm{out}}:= |\partial \mathcal{B}(d^3M)|^{-1}\sum\nolimits_{z\in \partial \mathcal{B}(d^3M) } \mathcal{H}_z(\widehat{\mathcal{C}}_A^{-},\widehat{\mathcal{C}}_A^{-}\cup D). 
\end{equation}
Parallel to (\ref{finish4.2}) (here we use Lemma \ref{lemma_upper_boundarytoset}  with $j=2$ instead), we have 
\begin{equation}\label{newineq_4.10}
\begin{split}
		\mathbb{P}^{D}\big(A\xleftrightarrow{\ge 0} \partial B(M) \big)\lesssim &M^{d-2}\theta_d(M/4) \mathbb{E} \big[ \mathcal{H}^{\mathrm{out}} \big]\\
		\lesssim &M^{d-2}\theta_d(M/4) \cdot \frac{\sum\nolimits_{z\in \partial \mathcal{B}(d^3M) }  \mathbb{P}^{D}(A\xleftrightarrow{\ge 0} z)}{\partial \mathcal{B}(d^3M)}.
		 \end{split}
\end{equation}
Combined with $\theta_d(M/4)\lesssim M^{-[(\frac{d}{2}-1)\boxdot  2]+\varsigma_d(M)}$ and $\mathbb{P}^{D}(A\xleftrightarrow{\ge 0} z)\asymp \mathbb{P}^{D}(A\xleftrightarrow{\ge 0} x)$ for all $z\in  \partial \mathcal{B}(d^3M)$ and $x\in \partial B(M)$ (by Proposition \ref{prop_pointtoset_harnack}), it confirms the lower bound in (\ref{ineq_relation_pointandboundary}):
\begin{equation}\label{newaddto_4.13}
		\mathbb{P}^{D}\big(A\xleftrightarrow{\ge 0} \partial B(M) \big) \lesssim  	\mathbb{P}^{D}\big(A\xleftrightarrow{\ge 0} x \big)\cdot M^{[(\frac{d}{2}-1)\boxdot  (d-4)]+\varsigma_d(M)}.  \pushQED{\qed} 
\qedhere
\popQED
\end{equation}

We record two useful applications of (\ref{newaddto_4.13}) as follows. 
\begin{itemize}
	\item  For $d\ge 3$, there exists $\Cl\label{const_5.6}>0$ such that for any $N\ge 1$, $M\ge \Cref{const_5.6}N^{1+\varsigma_d(N)}$, $x\in [B(\Cref{const_5.6}N^{1+\varsigma_d(N)})]^c$, $A\subset \widetilde{B}(N)$ and $D\subset \widetilde{B}(N)\cup [\widetilde{B}(2(M\vee |x|))]^c$. By (\ref{newaddto_4.13}) and Proposition \ref{prop_pointtoset_speed}, we have (let $y_M$ be an arbitrary point in $\partial B(M)$)
	\begin{equation}\label{app2}
	\begin{split}
		\mathbb{P}^{D}\big(A\xleftrightarrow{\ge 0 } \partial B(M)\big)\lesssim & \mathbb{P}^{D}\big(A\xleftrightarrow{\ge 0 }  y_M\big)\cdot M^{[(\frac{d}{2}-1)\boxdot  (d-4)]+\varsigma_d(M)}\\
		\lesssim  & \mathbb{P}^{D}\big(A\xleftrightarrow{\ge 0 }  x\big)\cdot |x|^{d-2}M^{-[(\frac{d}{2}-1)\boxdot  2]+\varsigma_d(M)}.
	\end{split}
\end{equation}

	\item For $d\ge 3$, there exists $\cl\label{const_5.7}>0$ such that for any $N\ge 1$, $M\le \cref{const_5.7}N^{1-\varsigma_d(N)}$, $x\in B(\cref{const_5.7}N^{1-\varsigma_d(N)})$, $A\subset [\widetilde{B}(N)]^c$ and $D\subset [\widetilde{B}(N)]^c\cup \widetilde{B}(\frac{1}{2}(M\land |x|))$. By (\ref{newaddto_4.13}) and Proposition \ref{prop_pointtoset_harnack}, we have (likewise, we arbitrarily take $y_M\in \partial B(M)$)
	\begin{equation}\label{app3}
	\begin{split}
		\mathbb{P}^{D}\big(A\xleftrightarrow{\ge 0 } \partial B(M)\big)\lesssim & \mathbb{P}^{D}\big(A\xleftrightarrow{\ge 0 }  y_M \big)\cdot M^{[(\frac{d}{2}-1)\boxdot  (d-4)]+\varsigma_d(M)}\\
		\lesssim  & \mathbb{P}^{D}\big(A\xleftrightarrow{\ge 0 }  x\big)\cdot M^{[(\frac{d}{2}-1) \boxdot  (d-4)]+\varsigma_d(M)}.
	\end{split}
\end{equation}
\end{itemize}

\subsection{Loop cluster decomposition}\label{subsection_loop_cluster_decomposition}

Let $\Upsilon_k$ be the collection of all partitions of $\{1,2,...,k\}$. Here a partition is a collection of disjoint subsets of $\{1,2,...,k\}$ whose union is exactly $\{1,2,...,k\}$. 

\begin{lemma}\label{lemma_new_decomposition}
	For any $d\ge 3$ and $k\ge 1$, there exist $\Cl\label{const_new_decompose1}(d),\Cl\label{const_new_decompose3}(d),\Cl\label{const_new_decompose2}(d,k)>0$ such that for any $M\ge \Cref{const_new_decompose1}N$, $A_0,D \subset \widetilde{B}(\frac{N}{10\Cref{const_new_decompose3} })$ and $A_1,...,A_k,D'\subset [\widetilde{B}(10\Cref{const_new_decompose3}M)]^c$, 
	\begin{equation}\label{add3.32}
	\begin{split}
		&\mathbb{P}\Big( \{A_0\xleftrightarrow{(D\cup D')}A_1,...,A_k\}\cap \{\widetilde{\mathcal{L}}_{1/2}^{D \cup D'}[N,M]\neq 0\}\Big)\\
		\le & \Cref{const_new_decompose2}\big(\frac{N}{M}\big)^{d-2}\mathbb{P}\big( A_0\xleftrightarrow{(D)}\partial B(\tfrac{N}{\Cref{const_new_decompose3} }) \big)\\
		&\cdot  \sum\nolimits_{\gamma_k \in \Upsilon_k} \prod\nolimits_{\{j_1,...,j_m\}\in \gamma_k} \mathbb{P}\big(A_{j_1}\xleftrightarrow{ (D') } ...\xleftrightarrow{ (D') } A_{j_m}\xleftrightarrow{ (D') } \partial B(\Cref{const_new_decompose3} M) \big).
			\end{split}
	\end{equation}
\end{lemma}
\begin{proof}
For brevity, in this proof we denote $\mathfrak{L}:=\mathfrak{L}^{D\cup D'}[N,M]$ and $\overline{\mathfrak{L}}:=\overline{\mathfrak{L}}^{D\cup D'}[N,M]$. We claim that when the event on the left-hand side of (\ref{add3.32}) occurs, there exists a partition $\gamma_k \in \Upsilon_k$ such that the following events happen disjointly:
\begin{equation}\label{lastrevision5.9}
	\mathsf{F}_0:= \big\{ A_0\xleftrightarrow{ \overline{\mathfrak{L}} }\mathfrak{L}\cup   \partial B(\tfrac{N}{\Cref{const_new_decompose3} }) \big\}, 
\end{equation}
\begin{equation}\label{lastrevision5.10}
	\mathsf{F}(\{j_1,...,j_m\}):= \big\{ A_{j_1}\xleftrightarrow{ \overline{\mathfrak{L}} }... \xleftrightarrow{ \overline{\mathfrak{L}} } A_{j_m} \xleftrightarrow{\overline{\mathfrak{L}} } \mathfrak{L}\cup \partial B(\Cref{const_new_decompose3}M) \big\}
\end{equation}
for all $\{j_1,...,j_m\}\in \gamma_k$. To see this, we explore the cluster of $\overline{\mathfrak{L}}$ containing $A_0$, i.e., take an increasing sequence of subsets obtained by adding glued loops sequentially, where each newly added loop intersects the current subset. We stop our exploration when the added loop intersects $\mathfrak{L}\cup   \partial B(\tfrac{N}{\Cref{const_new_decompose3} })$, which must happen since $A_0$ (contained in $\widetilde{B}(\frac{N}{10\Cref{const_new_decompose3} })$) and $A_1$ (contained in $[\widetilde{B}(10\Cref{const_new_decompose3}M)]^c$) are connected by $\cup \widetilde{\mathcal{L}}_{1/2}^{D\cup D'}$. We denote this explored cluster by $\hat{\mathcal{C}}_0$. Similarly, we explore the clusters of $\overline{\mathfrak{L}}$ containing $A_i$ for $1\le i\le k$ inductively. To this end, for each $i$, we define the exploration procedure of $A_i$ as follows: we explore the cluster of $\overline{\mathfrak{L}}$ containing $A_i$, and stop when the added loop intersects $\mathfrak{L}\cup \partial B(\Cref{const_new_decompose3}N)$. If the explored cluster of $A_i$ (denoted by $\hat{\mathcal{C}}_i$) exactly intersects $A_{j_1^{(i)}},...,A_{j^{(i)}_{m_i}}$ (we refer to them as ``involved sets''), we then add the set $\hat{\mathbf{j}}_i:=\big\{j_{l}^{(i)}:1\le l\le m_i,A_{j_{l}^{(i)}}\ \text{is not involved in}\ \hat{\mathcal{C}}_{i'}\ \text{for any}\ 1\le i'\le i-1  \big\}$ and the cluster $\hat{\mathcal{C}}_i$ to the collections $\hat{\gamma}$ and $\hat{\mathfrak{C}}$ respectively ($\hat{\gamma}$ and $\hat{\mathfrak{C}}$ are defined to be empty at the beginning). Having defined this, our whole exploration can be described as follows: at every step we let $i$ be the minimal integer such that $A_i$ is not yet an involved set (if such $i$ does not exist, we stop) and then we explore $A_i$ (we must stop eventually because every $A_i$ for $1\le i\le k$ is connected to $A_0\subset \widetilde{B}(\tfrac{N}{10\Cref{const_new_decompose3} })$). An illustration is provided in Figure \ref{pic1}. Here are some observations on this construction: 
	\begin{itemize}

	\item The cluster $\hat{\mathcal{C}}_0$ certifies the event $\mathsf{F}_0$, and each $\hat{\mathcal{C}}_i \in \hat{\mathfrak{C}}$ certifies $\mathsf{F}(\hat{\mathbf{j}}_i)$.

		\item The collection $\hat{\gamma}$ is a partition of $\{1,2,...,k\}$.

		\item  Every two clusters in $\hat{\mathfrak{C}}$ do not include a common glued loop.

	\end{itemize}
Thus, to obtain our claim, it suffices to show that every $\hat{\mathcal{C}}_i \in \hat{\mathfrak{C}}$ does not include a common glued loop with $\hat{\mathcal{C}}_0$. To see this, we employ a proof by contradiction. Assume that $\widetilde{\ell}_*$ is included in both $\hat{\mathcal{C}}_i$ and $\hat{\mathcal{C}}_0$. Note that by our stopping rule in the exploration, for $\hat{\mathcal{C}}_i$ (resp. $\hat{\mathcal{C}}_0$), only the last added loop can intersect $\partial B(\Cref{const_new_decompose3}M)$ (resp. $\partial B(\frac{N}{\Cref{const_new_decompose3}})$), while the previous ones are contained in $[\widetilde{B}(\Cref{const_new_decompose3}N)]^c$ (resp. $\widetilde{B}(\frac{N}{\Cref{const_new_decompose3}})$). Therefore, the common loop $\widetilde{\ell}_*$ must be the last added loop in both $\hat{\mathcal{C}}_i$ and $\hat{\mathcal{C}}_0$. However, this implies that $\widetilde{\ell}_*$ intersects both $\partial B(\Cref{const_new_decompose3}M)$ and $\partial B(\frac{N}{\Cref{const_new_decompose3}})$, and hence is contained in $\mathfrak{L}$, which causes a contradiction with the assumption that $\widetilde{\ell}_*$ is included in $\hat{\mathcal{C}}_0$ (noting that $\hat{\mathcal{C}}_0$ is composed of loops in $\overline{\mathfrak{L}}$). Thus, such a common loop $\widetilde{\ell}_*$ does not exist, and our claim holds.

We continue to the proof of (\ref{add3.32}). For simplicity, we denote $\mathsf{G}_l:=\mathsf{G}^{D\cup D'}_l[N,M]$, $\Omega_l:=\Omega_l[N,M]$ and $\widehat{\mathsf{G}}(\omega_l):=\widehat{\mathsf{G}}^{D\cup D'}[\omega_l;N,M]$ for $\omega_l\in \Omega_l$. By our claim in the last paragraph and the BKR inequality, the left-hand side of (\ref{add3.32}) is at most
	   \begin{equation}\label{3.37}
	   	\begin{split}
	   		&\sum_{l\ge 1} \sum_{\omega_l \in \Omega_l} \mathbb{P}\big(\widehat{\mathsf{G}}(\omega_l) \big)
	   		 \sum_{\gamma_k \in \Upsilon_k} \prod_{\{j_1,...,j_m\}\in \gamma_k} \mathbb{P}\big(\mathsf{F}_0 \mid \widehat{\mathsf{G}}(\omega_l) \big)\cdot \mathbb{P}\big(\mathsf{F}(j_1,...,j_m) \mid \widehat{\mathsf{G}}(\omega_l) \big). 	\end{split}
	   \end{equation}
In addition, it follows from Lemma \ref{lemma_prepare_loop_decompose} that 
\begin{equation}\label{newineq_3.10}
		\mathbb{P}\big(\mathsf{F}_0 \mid \widehat{\mathsf{G}}(\omega_l) \big)\le (l+1)\mathbb{P}\big(A_0\xleftrightarrow{(D) }  \partial B(\tfrac{N}{\Cref{const_new_decompose3}})\big),
\end{equation}
	  \begin{equation}\label{newineq_3.11}
	  	\begin{split}
	  		\mathbb{P}\big(\mathsf{F}(j_1,...,j_m) \mid \widehat{\mathsf{G}}(\omega_l) \big)\le (l+1)\mathbb{P}\big( A_{j_1}\xleftrightarrow{ (D') }... \xleftrightarrow{(D') } A_{j_m} \xleftrightarrow{ (D') }\partial B(\Cref{const_new_decompose3}N)\big).
	  	\end{split}
	  \end{equation}
	 By (\ref{3.37}), (\ref{newineq_3.10}) and (\ref{newineq_3.11}), we obtain that the left-hand side of (\ref{add3.32}) is at most  
 \begin{equation}\label{add_3.12}
	  	\begin{split}
	  			&\sum\nolimits_{l\ge 1} (l+1)^{k+1}\mathbb{P}(\mathsf{G}_l) \cdot \mathbb{P}\big( A_0\xleftrightarrow{(D)}  \partial B(\tfrac{N}{\Cref{const_new_decompose3}})\big)\\
		&\cdot  \sum\nolimits_{\gamma_k \in \Upsilon_k} \prod\nolimits_{\{j_1,...,j_m\}\in \gamma_k} \mathbb{P}\big(A_{j_1}\xleftrightarrow{ (D') } ...\xleftrightarrow{ (D') } A_{j_m}\xleftrightarrow{ (D') } \partial B(\Cref{const_new_decompose3} M) \big).
					  	\end{split}
	  \end{equation}
	  Moreover, Lemma \ref{lemma_Gl} shows that $\mathbb{P}(\mathsf{G}_l) \le  (\frac{CN}{M})^{l(d-2)}$ for all $l\ge 1$, which implies 
    \begin{equation}\label{newadd_3.15}
    	\sum\nolimits_{l\ge 1} (l+1)^{k+1}\mathbb{P}(\mathsf{G}_l) \le C'(d,k)\cdot \big(\frac{N}{M}\big)^{d-2}.
    \end{equation}
    Combining (\ref{add_3.12}) and (\ref{newadd_3.15}), we conclude the proof of this lemma.     
    \end{proof}

	  \begin{figure}[h]
	\centering
	\includegraphics[width=0.6\textwidth]{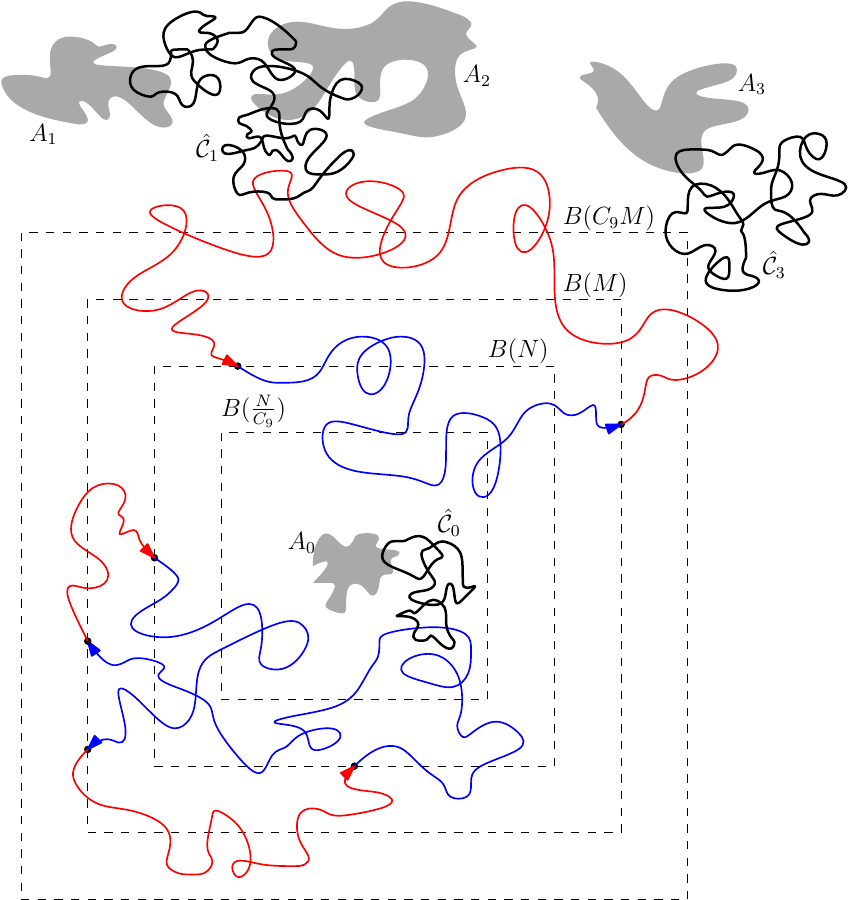}
	\caption{We consider the case $k=3$. The grey regions represent the target sets $\{A_j\}_{0\le j\le 3}$. The red (resp. blue) paths represent the forward (resp. backward) crossing paths from $\partial B(M)$ to $\partial B(N)$. The three black clusters are the explored clusters. As shown in this illustration, the event $\mathsf{F}_0$ in (\ref{lastrevision5.9}) is certified by the explored cluster $\hat{\mathcal{C}}_0$. In this example, the events $\mathsf{F}(\{1,2\})$ and $\mathsf{F}(\{3\})$ defined in (\ref{lastrevision5.10}) both happen and are certified by $\hat{\mathcal{C}}_1$ and $\hat{\mathcal{C}}_3$ respectively.  }\label{pic1}
\end{figure}

	  Now we are ready to show the main result of this subsection.

	\begin{lemma}\label{lemma_decompose_moresets}
	For any $d\ge 3$ and $k\ge 1$, there exist $\Cl\label{const_decompose_moresets1}(d),\Cl\label{const_decompose_moresets2}(d,k)>0$ such that for any $A_0,D\subset \widetilde{B}(\frac{N}{10\Cref{const_decompose_moresets1}})$ and $A_1,...,A_k,D'\subset [\widetilde{B}(10\Cref{const_decompose_moresets1}N)]^c$,  
	\begin{equation*}
	\begin{split}
			&\mathbb{P}\big(A_0\xleftrightarrow{(D\cup D')}A_1,...,A_k  \big)\\
\le &\Cref{const_decompose_moresets2} \mathbb{P}\big( A_0\xleftrightarrow{(D)}\partial B(\tfrac{N}{\Cref{const_decompose_moresets1}}) \big) 
\sum_{\gamma_k \in \Upsilon_k} \prod_{\{j_1,...,j_m\}\in \gamma_k} \mathbb{P}\big(A_{j_1} \xleftrightarrow{(D')}...\xleftrightarrow{(D')}A_{j_m}\xleftrightarrow{(D')} \partial B(\Cref{const_decompose_moresets1} N) \big).
	\end{split}
		\end{equation*}		
\end{lemma}
\begin{proof}
	Recall $\Cref{const_new_decompose1}$ in Lemma \ref{lemma_new_decomposition}. Suppose that the event $\mathsf{F}:=\big\{A_0\xleftrightarrow{(D\cup D')}A_1,...,A_k \big\}$ happens. If $\widetilde{\mathcal{L}}_{1/2}^{D \cup D'}[N,\Cref{const_new_decompose1}N]=0$, then the event $\big\{ A_0\xleftrightarrow{(D\cup [\widetilde{B}(\Cref{const_new_decompose1}N)]^c)} \partial B(N)\big\}$ must occur. Otherwise (i.e., on $\mathsf{F}\cap \big\{\widetilde{\mathcal{L}}_{1/2}^{D \cup D'}[N,\Cref{const_new_decompose1}N]=0 \big\}\cap \big\{ A_0\xleftrightarrow{(D\cup [\widetilde{B}(\Cref{const_new_decompose1}N)]^c)} \partial B(N)\big\}^c$), the loop cluster of $\widetilde{\mathcal{L}}_{1/2}^{D\cup [\widetilde{B}(\Cref{const_new_decompose1}N)]^c}$ containing $A_0$ is contained in $\widetilde{B}(N)$; however, this implies that any loop $\widetilde{\ell}\in  \widetilde{\mathcal{L}}_{1/2}^{D}-\widetilde{\mathcal{L}}_{1/2}^{D\cup [\widetilde{B}(\Cref{const_new_decompose1}N)]^c}$ intersecting this loop cluster (such $\widetilde{\ell}$ must exist on the event $\mathsf{F}$) has to cross the annulus $B(\Cref{const_new_decompose1}N)\setminus B(N)$, and thus the existence of such a loop violates the assumption that $\widetilde{\mathcal{L}}_{1/2}^{D \cup D'}[N,\Cref{const_new_decompose1}N]=0$. For the same reason, if $\widetilde{\mathcal{L}}_{1/2}^{D \cup D'}[\Cref{const_new_decompose1}N,\Cref{const_new_decompose1}^
	2N]=0$, then $\partial B(\Cref{const_new_decompose1}^
	2N)\xleftrightarrow{D'\cup \widetilde{B}(\Cref{const_new_decompose1}N)} A_1,...,A_k$ happens. This together with the exploration argument in the proof of Lemma \ref{lemma_new_decomposition} implies that there exists a partition $\gamma_k \in \Upsilon_k$ such that
	\begin{equation}
		A_{j_1} \xleftrightarrow{(D'\cup \widetilde{B}(\Cref{const_new_decompose1}N))}...\xleftrightarrow{(D'\cup \widetilde{B}(\Cref{const_new_decompose1}N))}A_{j_m}\xleftrightarrow{(D'\cup \widetilde{B}(\Cref{const_new_decompose1}N))} \partial B(\Cref{const_new_decompose1}^
	2N)
	\end{equation}
	for $\{j_1, \ldots, j_m\} \in \gamma_k$ happen disjointly. Thus, similar to (\ref{add_3.12}), we have 
		\begin{equation}\label{3.42}
		\begin{split}
				&\mathbb{P}\big( \mathsf{F}, \widetilde{\mathcal{L}}_{1/2}^{D\cup D'}[N,\Cref{const_new_decompose1}N]= \widetilde{\mathcal{L}}_{1/2}^{D\cup D'}[\Cref{const_new_decompose1}N,\Cref{const_new_decompose1}^2N]=0 \big)\\
\le & \mathbb{P}\big( A_0\xleftrightarrow{(D)}\partial B(N) \big)\sum_{\gamma_k \in \Upsilon_k} \prod_{\{j_1,...,j_m\}\in \gamma_k} \mathbb{P}\big(A_{j_1} \xleftrightarrow{(D')}...A_{j_m}\xleftrightarrow{(D')} \partial B(\Cref{const_new_decompose1}^2N) \big).
		\end{split}
	\end{equation}
		Meanwhile, Lemma \ref{lemma_new_decomposition} shows that 
	\begin{equation*}\label{newineq_3.18}
		\begin{split}
		&\max\big\{ \mathbb{P}(\mathsf{F}, \widetilde{\mathcal{L}}_{1/2}^{D\cup D'}[N,\Cref{const_new_decompose1}N]\neq 0 ),\mathbb{P}(\mathsf{F}, \widetilde{\mathcal{L}}_{1/2}^{D\cup D'}[\Cref{const_new_decompose1}N,\Cref{const_new_decompose1}^2N]\neq 0) \big \} \\
			\le    &C(d,k) \mathbb{P}\big( A_0\xleftrightarrow{(D)}\partial B(\tfrac{N}{\Cref{const_new_decompose3}}) \big)\sum_{\gamma_k \in \Upsilon_k} \prod_{\{j_1,...,j_m\}\in \gamma_k} \mathbb{P}\big(A_{j_1} \xleftrightarrow{(D')}...A_{j_m}\xleftrightarrow{(D')} \partial B(\Cref{const_new_decompose1}^2\Cref{const_new_decompose3}N) \big).
		\end{split}
	\end{equation*}
Combined with (\ref{3.42}), it confirms this lemma.
	\end{proof}

The following consequence of Lemma \ref{lemma_decompose_moresets} will be useful for later proofs.

\begin{lemma}\label{lemma_final}
	For any $3\le d\le 6$, there exists $\Cl\label{const_twopoint_set}(d)>0$ such that 
	\begin{equation}\label{final4.24}
		 \mathbb{P}^{D}\big( A\xleftrightarrow{\ge 0} x,y \big)\lesssim r^{-\frac{d}{2}+1+3\varsigma_d(r)}\mathbb{P}^{D}\big( A\xleftrightarrow{\ge 0} x \big)
	\end{equation}
	if either of the following two conditions (where $N\ge 1 $ and $r:=|x-y|$) holds:	\begin{enumerate}
		\item[(a)] $A \subset \widetilde{B}(N)$, $M\ge \Cref{const_twopoint_set}N^{1+\varsigma_d(N)}$, $x,y \in \partial B(M)$ and $D\subset \widetilde{B}(N)\cup [\widetilde{B}(2M)]^c$;

		\item[(b)] $A \subset [\widetilde{B}(N)]^c$, $M\le \Cref{const_twopoint_set}^{-1}N^{1-\varsigma_d(N)}$, $x,y\in B(M)$ and $D\subset [\widetilde{B}(N)]^c \cup \widetilde{B}(\frac{M}{2}) $.  
		
	\end{enumerate}

\end{lemma}
\begin{proof}
	Unless otherwise stated, the argument presented in this proof works for both Item (a) and Item (b). By the isomorphism theorem, it suffices to show that 
	\begin{equation}\label{final_4.25}
		 \mathbb{P}\big( A\xleftrightarrow{(D)} x,y \big)\lesssim r^{-(\frac{d}{2}-1)+3\varsigma_d(r)}\mathbb{P}\big( A\xleftrightarrow{(D)} x \big).
	\end{equation}
	By Lemma \ref{lemma_decompose_moresets}, we know that $\mathbb{P}\big( A\xleftrightarrow{(D)} x,y \big) $ is bounded from above by 
	\begin{equation*}
		\begin{split}
			&C \mathbb{P}\big( x \xleftrightarrow{(D)} \partial B_x(cr^{1-\varsigma_d(r)}) \big)  \Big[ \mathbb{P}\big( A\xleftrightarrow{(D)} y\xleftrightarrow{(D)}  B_x(c'r^{1-\varsigma_d(r)})  \big)\\
			 &\ \ \ \  \ \ \ \ \ \ \  \ \ \ \ \ \ \  \ \ \ \ \ \ \  \ \ \ + \mathbb{P}\big( A  \xleftrightarrow{(D)}     B_x(c'r^{1-\varsigma_d(r)})\big) \mathbb{P}\big( y  \xleftrightarrow{(D)}     B_x(c'r^{1-\varsigma_d(r)})\big) \Big]\\
			 \lesssim & \theta_d(cr^{1-\varsigma_d(r)}) \Big[ \mathbb{P}\big( A\xleftrightarrow{(D)} y \big) + \mathbb{P}\big( A  \xleftrightarrow{(D)}     B_x(c'r^{1-\varsigma_d(r)})\big) \mathbb{P}\big( y  \xleftrightarrow{(D)}     B_x(c'r^{1-\varsigma_d(r)})\big)  \Big]. 
		\end{split}
	\end{equation*}	
 Note that $\theta_d(cr^{1-\varsigma_d(r)})\lesssim r^{- (\frac{d}{2}-1) +3\varsigma_d(r)}$ (by (\ref{one_arm_low}) and (\ref{one_arm_6})) and that $\mathbb{P}\big( A\xleftrightarrow{(D)} y \big)\asymp \mathbb{P}\big( A\xleftrightarrow{(D)} x \big)$ (for Item (a) this follows from Proposition \ref{prop_pointtoset_speed}, while for Item (b) it follows from Proposition \ref{prop_pointtoset_harnack} instead). Moreover, by (\ref{app3}) we have  
 \begin{equation}
 	\begin{split}
 		\mathbb{P}\big( A  \xleftrightarrow{(D)}     B_x(c'r^{1-\varsigma_d(r)})\big)\lesssim r^{ (\frac{d}{2}-1)- \varsigma_d(r)}  	\mathbb{P}\big( A  \xleftrightarrow{(D)}    x\big).  
 	\end{split}
 \end{equation}
 Meanwhile, it follows from Corollary \ref{coro_subharmonic} that  
  \begin{equation}
 	\begin{split}
 			\mathbb{P}\big( y  \xleftrightarrow{(D)}     B_x(c'r^{1-\varsigma_d(r)})\big)  \lesssim  M^{2-d}r^{\frac{d}{2}-1- \varsigma_d(r)}   
 			\overset{r\lesssim M}{\lesssim } &r^{- (\frac{d}{2}-1) -\varsigma_d(r)}. 
 	\end{split}
 \end{equation}
To sum up, we obtain (\ref{final_4.25}) and thus establish the desired bound (\ref{final4.24}).
\end{proof}

\subsection{Upper bounds in Proposition \ref{prop_relation_pointandboundary}}\label{subsection_upper_prop_relation}	
\textbf{(a)} We use the second moment method to prove the upper bounds. Arbitrarily take $x\in \partial B(M)$. Let $\mathfrak{X}:=\sum_{z\in \partial B(M)}\mathbbm{1}_{z\xleftrightarrow{\ge 0} A}$. By Proposition \ref{prop_pointtoset_speed} and $|\partial B(M)|\asymp M^{d-1}$, we have
	\begin{equation}\label{new3.38}
		\mathbb{\mathbb{E}}^D[\mathfrak{X}] = \sum\nolimits_{z\in \partial B(M)}\mathbb{P}^D(A\xleftrightarrow{\ge 0} z) \asymp \mathbb{P}^D(A\xleftrightarrow{\ge 0}x ) \cdot  M^{d-1}. 
	\end{equation}
Next, we estimate the second moment 
	\begin{equation}\label{new3.43}
		\mathbb{\mathbb{E}}^D[\mathfrak{X}^2] = \sum\nolimits_{z_1,z_2\in \partial B(M)}\mathbb{P}^D( A\xleftrightarrow{\ge 0}z_1, z_2). 
	\end{equation}
	\noindent \textbf{Case 1:} $3\le d\le 6$. Let $r=r(z_1,z_2):=|z_1-z_2|$. By Lemma \ref{lemma_final} we have 
		\begin{equation}\label{new3.42}
	\begin{split}
			\mathbb{\mathbb{E}}^D[\mathfrak{X}^2] \lesssim & \mathbb{P}^D(A\xleftrightarrow{\ge 0} x)\cdot M^{3\varsigma_d(M)}\sum\nolimits_{z_1,z_2\in \partial B(M)} |z_1-z_2|^{-\frac{d}{2}+1}\\
			\overset{(\ref*{|x|-a_aneqd-1})}{\lesssim } & \mathbb{P}^D(A\xleftrightarrow{\ge 0} x)\cdot M^{\frac{3d}{2}-1+3\varsigma_d(M)}.
	\end{split}
	\end{equation}

\noindent \textbf{Case 2:} $d\ge 7$. This case requires a different approach. For any $z_1,z_2\in \partial B(M)$ and $v_1,v_2,v_3\in \widetilde{\mathbb{Z}}^d$, we define the event
\begin{equation}\label{newdef_4.38}
\begin{split}
	\mathsf{F}^{z_1,z_2}_{v_1,v_2,v_3}:=&\big\{\exists \widetilde{\ell}\in \widetilde{\mathcal{L}}_{1/2}^{D}\ \text{such that}\ v_1,v_2,v_3\in \mathrm{ran}(\widetilde{\ell}) \big\}\\
	&\circ \big\{ v_1\xleftrightarrow{(D)}z_1 \big\}
	\circ \big\{v_2\xleftrightarrow{(D)}z_2 \big\}\circ \big\{ v_3\xleftrightarrow{(D)}A \big\}.
\end{split}
	\end{equation} 
By the tree expansion argument (see e.g. \cite[Section 3.4]{cai2024high}), on $\{A\xleftrightarrow{(D)} z_1,z_2\}$, there exists $\widetilde{\ell}\in \widetilde{\mathcal{L}}_{1/2}^{D}$ connected to $A$, $z_1$ and $z_2$ disjointly. I.e., 
\begin{equation}\label{newineq_4.39}
	\{A\xleftrightarrow{(D)} z_1,z_2\} \subset \cup_{v_1,v_2,v_3\in \widetilde{\mathbb{Z}}^d} \mathsf{F}^{z_1,z_2}_{v_1,v_2,v_3}.
\end{equation}
For brevity, for $q\in \{1,-1\}$, we denote $\mathbf{B}_q^{\mathrm{in}}:=\widetilde{B}(M^{(\frac{2}{d-4})^q})$ and  $\mathbf{B}_q^{\mathrm{out}}:=[\widetilde{B}(M^{(\frac{2}{d-4})^q})]^c$. Moreover, for any $\lambda>0$, let $\widehat{\mathbf{B}}^{\mathrm{in}}(\lambda):=\widetilde{B}(\lambda N)$ and $\widehat{\mathbf{B}}^{\mathrm{out}}:=[\widetilde{B}(\lambda N)]^c$. For any $\diamond  \in \{\mathrm{in}, \mathrm{out}\}$, $q\in \{1,-1\}$ and $\lambda>0$, we define the events 
\begin{equation}\label{improve4.43}
	\mathsf{F}^{z_1,z_2}_{(1)}(\diamond,q):= \cup_{v_1\in \mathbf{B}_q^{\diamond},v_2,v_3\in \widetilde{\mathbb{Z}}^d}\mathsf{F}^{z_1,z_2}_{v_1,v_2,v_3},
\end{equation}
\begin{equation}
	\mathsf{F}^{z_1,z_2}_{(2)}(\diamond,q) := \cup_{v_2\in \mathbf{B}_q^{\diamond},v_1,v_3\in \widetilde{\mathbb{Z}}^d}\mathsf{F}^{z_1,z_2}_{v_1,v_2,v_3},
\end{equation}
\begin{equation}
	\mathsf{F}^{z_1,z_2}_{(3)}(\diamond,q,\lambda):= \cup_{v_1,v_2\in (\mathbf{B}_q^{\diamond})^c,v_3\in [\widehat{\mathbf{B}}^{\diamond}(\lambda)]^c}\mathsf{F}^{z_1,z_2}_{v_1,v_2,v_3},
\end{equation}
\begin{equation}\label{improve4.46}
	\mathsf{F}^{z_1,z_2}_{(4)}(\diamond,q,\lambda):= \cup_{v_1,v_2\in(\mathbf{B}_q^{\diamond})^c,v_3\in \widehat{\mathbf{B}}^{\diamond}(\lambda)}  \mathsf{F}^{z_1,z_2}_{v_1,v_2,v_3}.
\end{equation}
These notations will also be employed in the proof of Item (b). Let $C_*:=\Cref{const_pointtoset_speed}\vee \Cref{const_5.6}$ (recall $\Cref{const_pointtoset_speed}$ and $\Cref{const_5.6}$ in Proposition \ref{prop_pointtoset_speed} and (\ref{app2}) respectively). For Item (a), we denote $\mathsf{F}^{z_1,z_2}_{(j)}:=\mathsf{F}^{z_1,z_2}_{(j)}(\mathrm{in},1)$ for $j\in \{1,2\}$, and denote $\mathsf{F}^{z_1,z_2}_{(j)}:=\mathsf{F}^{z_1,z_2}_{(j)}(\mathrm{in},1,C_*)$ for $j\in \{3,4\}$. Let $\mathbb{I}_j:=\sum_{z_1,z_2\in \partial B(M)}\mathbb{P}(\mathsf{F}^{z_1,z_2}_{(j)})$ for $1\le j\le 4$. Thus, by (\ref{newineq_4.39}) we have
\begin{equation}\label{newadd_3.63}
	\mathbb{E}^{D}[\mathfrak{X}^2]\le \sum\nolimits_{1\le j\le 4}  \mathbb{I}_j.
\end{equation}

In what follows, we estimate $\mathbb{I}_j$ for $1\le j\le 4$ separartely. For any $j\in \{1,2\}$, since $\mathsf{F}^{z_1,z_2}_{(j)}\subset \{z_j\xleftrightarrow{(D)}B(M^{\frac{2}{d-4}})\} \circ \{z_{3-j}\xleftrightarrow{(D)} A \}$ (see Figure \ref{pic2}), we have 
\begin{equation}\label{newadd_3.64}
	\begin{split}
		\mathbb{I}_j \overset{(\text{BKR})}{\le} &\sum_{z_1,z_2\in \partial B(M)}\mathbb{P}\big(z_j\xleftrightarrow{(D)} \partial B(M^{\frac{2}{d-4}})\big) \cdot  \mathbb{P}(z_{3-j}\xleftrightarrow{(D)} A)\\
		\overset{\text{Corollary}\ \ref*{coro_subharmonic},\ \text{Proposition}\ \ref*{prop_pointtoset_speed}}{\lesssim}  
		& \mathbb{P}^{D}(A\xleftrightarrow{\ge 0} x) \cdot M^{d+2}.
	\end{split}
\end{equation}

For $\mathbb{I}_3$, by the union bound, the BKR inequality and Lemma \ref{lemma_revise_2.3}, we have 
\begin{equation}\label{newadd_3.65}
	\begin{split}
		\mathbb{I}_3  	\lesssim & \sum_{z_1,z_2\in \partial B(M)}\sum_{w_1,w_2\in \mathbb{Z}^d,w_3\in [B(C_*N)]^c}|w_1-w_2|^{2-d}|w_2-w_3|^{2-d}\\
&\cdot |w_3-w_1|^{2-d} |w_1-z_1|^{2-d}|w_2-z_2|^{2-d}\mathbb{P}^{D}(w_3\xleftrightarrow{\ge 0} A)\\
\overset{(\ref*{addnew3.25}),(\ref*{|x|-a_aneqd-1})}{\lesssim}&  \mathbb{P}^{D}(A\xleftrightarrow{\ge 0} x) M^{d-1}\sum_{z_1\in \partial B(M)}\sum_{w_1,w_2\in \mathbb{Z}^d,w_3\in [B(C_*N)]^c}|w_1-w_2|^{2-d}\\
&\cdot |w_2-w_3|^{2-d}|w_3-w_1|^{2-d}|w_1-z_1|^{2-d}|w_3|^{2-d}\\
\overset{(\ref*{z1z2z3})}{\lesssim} & \mathbb{P}^{D}(A\xleftrightarrow{\ge 0} x) M^{d-1}  \sum\nolimits_{z_1\in \partial B(M)} |z_1|^{4-d} \lesssim  \mathbb{P}^{D}(A\xleftrightarrow{\ge 0} x)\cdot   M^{d+2}.
\end{split}
\end{equation}
Note that the requirement $w_3\in [B(C_*N)]^c$ in the sum is necessary to guarantee that Proposition \ref{prop_pointtoset_speed} applies to the second transfer in (\ref{newadd_3.65}).

	  \begin{figure}[h]
	\centering
	\includegraphics[width=0.45\textwidth]{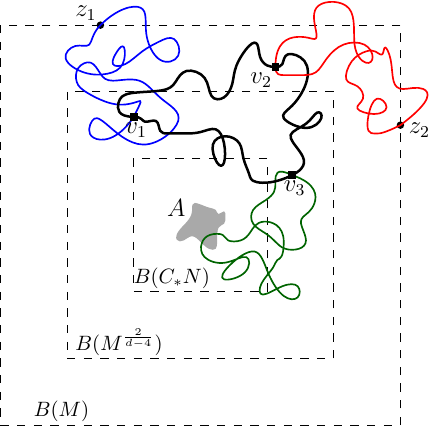}
	\caption{In this illustration, the grey region represents the target set $A$. The black loop represents the loop $\widetilde{\ell}$ involved in the event $\mathsf{F}_{v_1,v_2,v_3}^{z_1,z_2}$ (recalling (\ref{newdef_4.38})). The blue, red and green clusters are three loop clusters that consist of disjoint collections of loops and certify $\big\{ v_1\xleftrightarrow{(D)}z_1 \big\} $, $\big\{ v_2\xleftrightarrow{(D)}z_2 \big\}$ and $\big\{ v_3\xleftrightarrow{(D)}A \big\}$ respectively. In this example, $v_1$ is contained in $B(M^{\frac{2}{d-4}})$ and hence, $\mathsf{F}^{z_1,z_2}_{(1)}$ happens, which further implies that $ \big\{ z_1\xleftrightarrow{(D)}B(M^{\frac{2}{d-4}})\big\}
	$ (certified by the blue cluster) and $\big\{z_2\xleftrightarrow{(D)} A \big\}$ (certified by the union of the black loop and the red and green clusters) happen disjointly. In addition, the event $\mathsf{F}^{z_1,z_2}_{(3)}$ also occurs in this example because $v_3\in [\widetilde{B}(C_*N)]^c$. }\label{pic2}
\end{figure}


For $\mathbb{I}_4$, we recall the constant $\Cref{const_prepare_loop_decompose}$ in Lemma \ref{lemma_prepare_loop_decompose}, and employ the notations $\mathsf{G}_l^{D}$, $\{\widetilde{\eta}_i^{\mathrm{F}}\}_{i=1}^{l}$, $\{\widetilde{\eta}_i^{\mathrm{B}}\}_{i=1}^{l}$, $\Omega_l$, $\mathfrak{L}^{D}$ and $\overline{\mathfrak{L}}^{D}$ for the crossings from $\partial B(\frac{1}{10} M^{\frac{2}{d-4}})$ to $\partial B( \Cref{const_prepare_loop_decompose}C_*N)$. In addition, for any $\omega_l\in \Omega_l$, we abbreviate (recalling (\ref{def_widehat_G}))
\begin{equation}
	\widehat{\mathsf{G}}^{D}(\omega_l):= \widehat{\mathsf{G}}^{D}[\omega_l;  \Cref{const_prepare_loop_decompose}C_*N,\tfrac{1}{10} M^{\frac{2}{d-4}} ]. 
\end{equation}
When $\mathsf{F}^{z_1,z_2}_{(4)}$ happens, the loop involved in the event $\mathsf{F}^{z_1,z_2}_{v_1,v_2,v_3}$ (see (\ref{newdef_4.38})) must contain a crossing from $\partial B(\frac{1}{10} M^{\frac{2}{d-4}})$ to $\partial B( \Cref{const_prepare_loop_decompose}C_*N)$  (see Figure \ref{pic3}). Thus, we have $\mathsf{F}^{z_1,z_2}_{(4)}\subset \cup_{l\ge 1}\mathsf{G}_l^{D}$. Moreover, on $\mathsf{F}^{z_1,z_2}_{(4)}\cap \mathsf{G}_l^{D}$, there exist $i_1,i_2\in \{1,...,l\}$ such that $\mathrm{ran}(\widetilde{\eta}_{i_1}^{\mathrm{F}})\cap [\widetilde{B}(M^{\frac{2}{d-4}})]^c\xleftrightarrow{(D)}z_1$, $\mathrm{ran}(\widetilde{\eta}_{i_2}^{\mathrm{F}})\cap [\widetilde{B}(M^{\frac{2}{d-4}})]^c \xleftrightarrow{(D)}z_2$ and $\mathfrak{L}^{D}\cap \widetilde{B}(C_*N) \xleftrightarrow{\overline{\mathfrak{L}}^{D}} A$ happen disjointly (we denote this event by $\hat{\mathsf{F}}^{z_1,z_2}_{i_1,i_2}$). This further implies that 
\begin{equation}\label{newadd_3.66}
	\begin{split}
		\mathbb{P}(\mathsf{F}^{z_1,z_2}_{(4)}) 		\le  \sum\nolimits_{l\ge 1}\sum\nolimits_{\omega_l \in \Omega_l} \mathbb{P}\big(\widehat{\mathsf{G}
		}^{D}(\omega_l)\big) \sum\nolimits_{1\le i_1,i_2\le l} \mathbb{P}\big(\hat{\mathsf{F}}^{z_1,z_2}_{i_1,i_2}\mid \widehat{\mathsf{G}}^{D}(\omega_l)\big). 
	\end{split}
\end{equation}

	  \begin{figure}[h]
	\centering
	\includegraphics[width=0.5\textwidth]{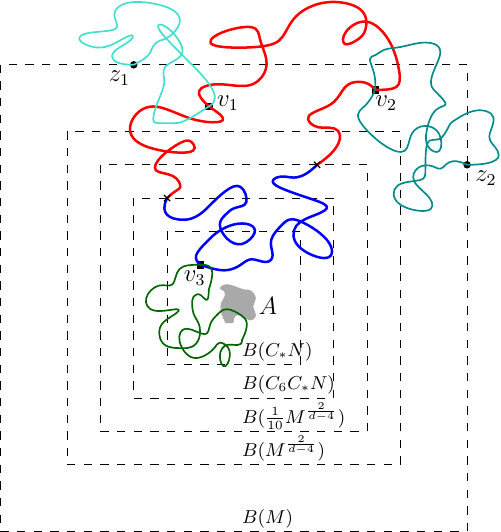}
	\caption{This is an illustration for the event $\mathsf{F}^{z_1,z_2}_{(4)}$. Since $v_1,v_2\in [\widetilde{B}(M^{\frac{2}{d-4}})]^c$ and $v_3\in \widetilde{B}(C_*N)$, the loop $\widetilde{\ell}$ involved in $\mathsf{F}_{v_1,v_2,v_3}^{z_1,z_2}$ must include forward and backward crossing paths (i.e., the red and blue paths) from $\partial B(\frac{1}{10}M^{\frac{2}{d-4}})$ to $\partial B(\Cref{const_prepare_loop_decompose} C_*N)$. The three green clusters (with different shades) consist of disjoint collections of loops and certify $ \{v_1\xleftrightarrow{(D)}z_1 \}$, $ \{ v_2\xleftrightarrow{(D)}z_2  \}$ and $ \{v_3\xleftrightarrow{(D)}A \} $ respectively. 
	 }\label{pic3}
\end{figure}

For any $1\le i_1,i_2\le l$, by the BKR inequality we have 
\begin{equation}\label{improve4.51}
	\begin{split}
		&\mathbb{P}\big(\hat{\mathsf{F}}^{z_1,z_2}_{i_1,i_2}\mid \widehat{\mathsf{G}}^{D}(\omega_l)\big)\\
		\le &  \sum_{w_1,w_2\in [B(M^{\frac{2}{d-4}})]^c} 	\mathbb{P}\big( \mathrm{ran}(\widetilde{\eta}_{i_j}^{\mathrm{F}})\cap \widetilde{B}_{w_j}(1)\neq \emptyset,\forall j\in \{1,2\}  \mid \widehat{\mathsf{G}}^{D}(\omega_l)\big) \\
	&	\cdot \mathbb{P}\big( \widetilde{B}_{w_1}(1)\xleftrightarrow{(D)}z_1 \big) \mathbb{P}\big( \widetilde{B}_{w_2}(1)\xleftrightarrow{(D)}z_2 \big)\mathbb{P}\Big(A\xleftrightarrow{\overline{\mathfrak{L}}^{D}}   \mathfrak{L}^{D} \cap \widetilde{B}(C_* N)    \mid \widehat{\mathsf{G}}^{D}(\omega_l)\Big)\\
		\lesssim & \sum_{w_1,w_2\in [B(M^{\frac{2}{d-4}})]^c} 	\mathbb{P}\big( \mathrm{ran}(\widetilde{\eta}_{i}^{\mathrm{F}})\cap \widetilde{B}_{w_j}(1)\neq \emptyset,\forall j\in \{1,2\}  \mid \widehat{\mathsf{G}}^{D}(\omega_l)\big)\\
		&\cdot |w_1-z_1|^{2-d} |w_2-z_2|^{2-d} (l+1)\mathbb{P}\big(A   \xleftrightarrow{(D)} \partial B(C_*N) \big),
					\end{split}
\end{equation}
where in the last inequality we used Lemma \ref{lemma_revise_B1}, (\ref{two_point}) and Lemma \ref{lemma_prepare_loop_decompose}. We next estimate the probability $\mathbb{P}\big( \mathrm{ran}(\widetilde{\eta}_{i_j}^{\mathrm{F}})\cap \widetilde{B}_{w_j}(1)\neq \emptyset,\forall j\in \{1,2\}  \mid \widehat{\mathsf{G}}^{D}(\omega_l)\big) $ in two cases depending on whether $i_1 = i_2$:
\begin{itemize}
	\item when $i_1=i_2$, it follows from (\ref{newineq_2.69}) that this probability is at most 
	\begin{equation}\label{improve_4.52}
\mathbb{J}_{w_1,w_2}:=C|w_1|^{2-d}|w_1-w_2|^{2-d}|w_2|^{2-d}\big(M^{\frac{2}{d-4}}\big)^{d-2}; 
	\end{equation}
	
	\item when $i_1\neq i_2$, conditioning on $\widehat{\mathsf{G}}^{D}(\omega_l)$, since $\widetilde{\eta}_{i_1}^{\mathrm{F}}$ and $\widetilde{\eta}_{i_2}^{\mathrm{F}}$ are conditionally independent, by (\ref{special_newineq_2.69}) we have	\begin{equation}\label{improve_4.53}
		\begin{split}
		&\mathbb{P}\big( \mathrm{ran}(\widetilde{\eta}_{i_j}^{\mathrm{F}})\cap \widetilde{B}_{w_j}(1)\neq \emptyset,\forall j\in \{1,2\}  \mid \widehat{\mathsf{G}}^{D}(\omega_l)\big) \\ 
		=&\prod\nolimits_{j=1,2}	\mathbb{P}\big( \mathrm{ran}(\widetilde{\eta}_{i_j}^{\mathrm{F}})\cap \widetilde{B}_{w_j}(1)\neq \emptyset \mid \widehat{\mathsf{G}}^{D}(\omega_l)\big)\\
		\lesssim &\mathbb{J}_{w_1,w_2}':=  |w_1|^{4-2d}|w_2|^{4-2d} \big(M^{\frac{2}{d-4}}\big)^{2d-4}.
		\end{split}
	\end{equation}

\end{itemize}
For any $w_1,w_2\in [B(M^{\frac{2}{d-4}})]^c$, since $|w_1-w_2|\le |w_1|+|w_2|\le 2(|w_1|\vee |w_2|)$ and $|w_1|\land |w_2|\gtrsim M^{\frac{2}{d-4}}$, one has
\begin{equation}
\begin{split}
		\frac{\mathbb{J}_{w_1,w_2}'}{\mathbb{J}_{w_1,w_2}}= \big(\frac{|w_1-w_2|\cdot M^{\frac{2}{d-4}}}{|w_1|\cdot |w_2|}\big)^{d-2} \lesssim 1. 
\end{split}
\end{equation}
Therefore, for all $1\le i_1,i_2\le l$, we have 
\begin{equation}\label{improve4.55}
	\mathbb{P}\big( \mathrm{ran}(\widetilde{\eta}_{i_j}^{\mathrm{F}})\cap \widetilde{B}_{w_j}(1)\neq \emptyset,\forall j\in \{1,2\}  \mid \widehat{\mathsf{G}}^{D}(\omega_l)\big) \lesssim \mathbb{J}_{w_1,w_2}. 
\end{equation}
Moreover, it follows from (\ref{app2}) that 
\begin{equation}\label{improve4.56}
	\mathbb{P}\big(A   \xleftrightarrow{(D)} \partial B(C_*N) \big) \lesssim \mathbb{P}^{D}\big(A   \xleftrightarrow{\ge 0} x \big)\cdot N^{-2}M^{d-2}.
\end{equation}
Thus, by summing over $z_1,z_2\in \partial B(M)$ on both sides of (\ref{improve4.51}), we obtain 
\begin{equation*}
	\begin{split}
		&\sum\nolimits_{z_1,z_2\in \partial B(M)}\mathbb{P}\big(\hat{\mathsf{F}}^{z_1,z_2}_{i_1,i_2}\mid \widehat{\mathsf{G}}^{D}(\omega_l)\big)\\
	\overset{(\ref*{improve4.55}),(\ref*{improve4.56})}{	\lesssim} & \mathbb{P}^{D}\big(A   \xleftrightarrow{\ge 0}   x \big)\cdot (l+1) N^{-2}M^{d-2+\frac{2(d-2)}{d-4}} \sum\nolimits_{w_1,w_2\in \mathbb{Z}^d} |w_1|^{2-d}|w_1-w_2|^{2-d}|w_2|^{2-d}\\
		&\cdot  \sum\nolimits_{z_1,z_2\in \partial B(M)} |w_1-z_1|^{2-d}|w_2-z_2|^{2-d} \\
		\overset{(\ref*{|x|-a_aneqd-1})}{\lesssim }	&	 \mathbb{P}^{D}\big(A   \xleftrightarrow{\ge 0}   x \big)\cdot (l+1) N^{-2}M^{d-1+\frac{2(d-2)}{d-4}} \sum\nolimits_{w_1,w_2\in \mathbb{Z}^d} |w_1|^{2-d}|w_1-w_2|^{2-d}|w_2|^{2-d}\\
		&\cdot  \sum\nolimits_{z_1\in \partial B(M)} |w_1-z_1|^{2-d}\\
		\overset{(\ref*{z1z2})}{\lesssim }	&\mathbb{P}^{D}\big(A   \xleftrightarrow{\ge 0}   x \big)\cdot (l+1) N^{-2}M^{d-1+\frac{2(d-2)}{d-4}}  \sum\nolimits_{z_1\in \partial B(M)} |z_1|^{2-d}\\
		\lesssim & \mathbb{P}^{D}\big(A   \xleftrightarrow{\ge 0}   x \big)\cdot (l+1) N^{-2}M^{d+\frac{2(d-2)}{d-4}}.
		\end{split}
\end{equation*}
Combined with (\ref{newadd_3.66}), it yields that 
\begin{equation}\label{newadd_3.73}
	\begin{split}
	\mathbb{I}_4\lesssim  	&\mathbb{P}^{D}\big(A   \xleftrightarrow{\ge 0}   x \big) \cdot  N^{-2}M^{d+\frac{2(d-2)}{d-4}} \sum\nolimits_{l\ge 1}   l^3 \mathbb{P}\big(\mathsf{G}_l^{D}\big)  	\\
	\overset{\text{Lemma}\ \ref*{lemma_Gl} }{\lesssim} &\mathbb{P}^{D}\big(A   \xleftrightarrow{\ge 0}   x \big)\cdot   N^{-2}M^{d+\frac{2(d-2)}{d-4}} \big( \frac{M^{\frac{2}{d-4}}}{N}\big)^{2-d}
	\lesssim  \mathbb{P}^{D}\big(A   \xleftrightarrow{\ge 0}   x \big)\cdot M^{d+2},
		\end{split}
\end{equation}
where the last inequality is ensured by $M\gtrsim  N^{\frac{d-4}{2}}$. Combining (\ref{newadd_3.63}), (\ref{newadd_3.64}), (\ref{newadd_3.65}) and (\ref{newadd_3.73}), we have: for all $d\ge 7$, 
\begin{equation}\label{newadd_3.74}
	\mathbb{E}^{D}[\mathfrak{X}^2]\lesssim \mathbb{P}^{D}\big(A   \xleftrightarrow{\ge 0}   x \big)\cdot M^{d+2}. 
	\end{equation}

Now we consider all dimensions $d\ge 3$ collectively. By applying the Paley-Zygmund inequality, and using (\ref{new3.38}), (\ref{new3.42}) and (\ref{newadd_3.74}), we derive the desired bound in Item (a):
\begin{equation}\label{newaddineq_4.51}
	\begin{split}
\mathbb{P}^D\big(A\xleftrightarrow{\ge 0} \partial B(M)\big)= \mathbb{P}^{D}(\mathfrak{X}>0)
\gtrsim  & \frac{\big[\mathbb{P}^D(A\xleftrightarrow{\ge 0}x ) \cdot  M^{d-1}\big]^2}{ \mathbb{P}^{D}\big(A   \xleftrightarrow{\ge 0}   x \big)\cdot M^{[(\frac{3d}{2}-1)\boxdot (d+2)]+3\varsigma_d(N)}} \\
=&\mathbb{P}^{D}\big(A   \xleftrightarrow{\ge 0}   x \big)\cdot  M^{[(\frac{d}{2}-1)]\boxdot  (d-4)]-3\varsigma_d(N)}.
	\end{split}
\end{equation}

\textbf{(b)} We also use the second moment method as in the proof for Item (a). Likewise, we arbitrarily take $x\in \partial B(M)$ and denote $\mathfrak{X}:=\sum_{z\in \partial B(M)}\mathbbm{1}_{z\xleftrightarrow{\ge 0} A}$. The estimate for $\mathbb{E}^{D}[\mathfrak{X}]$ in (\ref{new3.38}) also holds using Proposition \ref{prop_pointtoset_speed}. Then we also bound $\mathbb{E}^{D}[\mathfrak{X}^2]$ separately when $3\le d\le 6$ and when $d\ge 7$.

\noindent \textbf{Case 1:} $3\le d\le 6$. Similar to (\ref{new3.42}), by Lemma \ref{lemma_final} we have
\begin{equation}\label{newaddineq_4.53}
	\mathbb{E}^{D}[\mathfrak{X}^2] \lesssim \mathbb{P}^D(A\xleftrightarrow{\ge 0} x)\cdot M^{\frac{3d}{2}-1+3\varsigma_d(M)}.
\end{equation}

\noindent \textbf{Case 2:} $d\geq 7$. Recall the notations in (\ref{newdef_4.38}) and (\ref{improve4.43})--(\ref{improve4.46}). Also recall $\cref{const_pointtoset_harnack}$ and $\cref{const_5.7}$ in Proposition \ref{prop_pointtoset_harnack} and (\ref{app3}). Let $c_*:= \cref{const_pointtoset_harnack}\land \cref{const_5.7}$. We denote $\mathsf{F}^{z_1,z_2}_{(j),*}:= \mathsf{F}^{z_1,z_2}_{(j)}(\mathrm{out},-1)$ for $j\in \{1,2\}$, and denote $\mathsf{F}^{z_1,z_2}_{(j),*}:=\mathsf{F}^{z_1,z_2}_{(j)}(\mathrm{out},-1,c_*)$ for $j\in \{3,4\}$. As in (\ref{newadd_3.63}), 
\begin{equation}\label{newaddineq_4.54}
	\mathbb{E}^{D}[\mathfrak{X}^2]\le \sum\nolimits_{1\le j\le 4}  \mathbb{I}_j^*:=\sum\nolimits_{1\le j\le 4} \Big(\sum\nolimits_{z_1,z_2\in \partial B(M)}\mathbb{P}(\mathsf{F}^{z_1,z_2}_{(j),*}) \Big).
\end{equation}

For $\mathbb{I}_1^*$ and $\mathbb{I}_2^*$, since $\mathsf{F}^{z_1,z_2}_{(j),*}\subset \{z_j\xleftrightarrow{(D)} \partial B(M^{\frac{d-4}{2}}) \}\circ \{z_{3-j}\xleftrightarrow{(D)} A \}$ for $j\in \{1,2\}$ and $\big\{ z \xleftrightarrow{} \partial B(M^{\frac{d-4}{2}}) \big\}\subset \big\{ z \xleftrightarrow{} \partial B_z(\frac{1}{2}M^{\frac{d-4}{2}}) \big\} $ for $z\in \partial B(M)$, one has 
\begin{equation}\label{newaddineq_4.59}
	\begin{split}
		\max\{\mathbb{I}_1^*, \mathbb{I}_2^* \} \overset{(\text{BKR})}{\lesssim} & |\partial B(M)|\cdot \theta_d(\tfrac{1}{2}M^{\frac{d-4}{2}}) \sum\nolimits_{z\in \partial B(M)}\mathbb{P}^{D}(A\xleftrightarrow{\ge 0}z) \\
		\lesssim  &\mathbb{P}^{D}(A\xleftrightarrow{\ge 0}x)\cdot M^{d+2}. 
	\end{split}
\end{equation}
For $\mathbb{I}_3^*$, parallel to (\ref{newadd_3.65}), by applying the BKR inequality and using Lemma \ref{lemma_revise_2.3} and Proposition \ref{prop_pointtoset_harnack}, we have 
\begin{equation}\label{newaddineq_4.60}
	\begin{split}
		\mathbb{I}_3^* \lesssim    &   \mathbb{P}^{D}(A\xleftrightarrow{\ge 0} x) \sum_{z_1,z_2\in \partial B(M)}\sum_{w_1,w_2\in \mathbb{Z}^d,w_3\in B(c_* N)}|w_1-w_2|^{2-d}|w_2-w_3|^{2-d}\\
&\cdot |w_3-w_1|^{2-d}|w_1-z_1|^{2-d}|w_2-z_2|^{2-d}\\
\overset{(\ref*{z1z2z3})}{\lesssim } &  \mathbb{P}^{D}(A\xleftrightarrow{\ge 0} x) \sum_{z_1,z_2\in \partial B(M)} |z_1-z_2|^{4-d}\overset{(\ref*{|x|-a_aneqd-1})}{\lesssim } \mathbb{P}^{D}(A\xleftrightarrow{\ge 0} x)\cdot M^{d+2}. 
	\end{split}
\end{equation}

For $\mathbb{I}_4^*$, we recall $\Cref{const_prepare_loop_decompose}$ in Lemma \ref{lemma_prepare_loop_decompose} and employ the notations $\mathsf{G}_l^{D}$, $\{\widetilde{\eta}_i^{\mathrm{F}}\}_{i=1}^{l}$, $\{\widetilde{\eta}_i^{\mathrm{B}}\}_{i=1}^{l}$, $\Omega_l$, $\mathfrak{L}^D$, $\overline{\mathfrak{L}}^D$ and $\widehat{\mathsf{G}}^{D}(\omega_l)$ for the crossings from $\partial B(10 M^{\frac{d-4}{2}} )$ to $\partial B( \Cref{const_prepare_loop_decompose}^{-1} c_*N)$. By similar arguments as in the proof of (\ref{newadd_3.66}), one has $\hat{\mathsf{F}}^{z_1,z_2}_{(4)}\cap \mathsf{G}_l^{D}\subset \cup_{1\le i_1,i_2\le l}\hat{\mathsf{F}}^{z_1,z_2}_{i_1,i_2,*}$, where $\hat{\mathsf{F}}^{z_1,z_2}_{i_1,i_2,*}$ is the event that $\mathrm{ran}(\widetilde{\eta}_{i_1}^{\mathrm{B}})\cap \widetilde{B}(M^{\frac{d-4}{2}})\xleftrightarrow{(D)} z_1$, $\mathrm{ran}(\widetilde{\eta}_{i_2}^{\mathrm{B}})\cap \widetilde{B}(M^{\frac{d-4}{2}}) \xleftrightarrow{(D)} z_2$ and $A \xleftrightarrow{\overline{\mathfrak{L}}^D} \mathfrak{L}^{D}\cap [\widetilde{B}(c_*N)]^c $ happen disjointly. Thus, we have 
\begin{equation}\label{newaddineq_4.61}
	\begin{split}
		\mathbb{P}(\mathsf{F}^{z_1,z_2}_{(4),*})\le  \sum\nolimits_{l\ge 1}\sum\nolimits_{\omega_l \in \Omega_l} \mathbb{P}\big( \widehat{\mathsf{G}}^{D}(\omega_l)\big) \sum\nolimits_{1\le i_1,i_2\le l} \mathbb{P}\big(\hat{\mathsf{F}}^{z_1,z_2}_{i_1,i_2,*}\mid \widehat{\mathsf{G}}^{D}(\omega_l)\big). 
	\end{split}
\end{equation}
In addition, by the BKR inequality, $\mathbb{P}\big(\hat{\mathsf{F}}^{z_1,z_2}_{i_1,i_2,*}\mid \widehat{\mathsf{G}}^{D}(\omega_l)\big)$ is bounded from above by 
\begin{equation}\label{improve4.66}
	\begin{split}
		  & \sum_{w_1,w_2\in B(M^{\frac{d-4}{2}})} 	\mathbb{P}\big( \mathrm{ran}(\widetilde{\eta}_{i_j}^{\mathrm{B}})\cap \widetilde{B}_{w_j}(1)\neq \emptyset,\forall j\in \{1,2\}  \mid \widehat{\mathsf{G}}^{D}(\omega_l)\big) \\
		& \mathbb{P}\big( \widetilde{B}_{w_1}(1)\xleftrightarrow{(D)}z_1 \big) \mathbb{P}\big(\widetilde{B}_{w_2}(1)\xleftrightarrow{(D)}z_2 \big)\mathbb{P}\Big(A\xleftrightarrow{\overline{\mathfrak{L}}^{D}}  \mathfrak{L}^{D}\cap [\widetilde{B}(c_*  N)]^c    \mid \widehat{\mathsf{G}}^{D}(\omega_l)\Big)\\
		\lesssim & \sum_{w_1,w_2\in B(M^{\frac{d-4}{2}})}  	\mathbb{P}\big( \mathrm{ran}(\widetilde{\eta}_{i_j}^{\mathrm{B}})\cap \widetilde{B}_{w_j}(1)\neq \emptyset,\forall j\in \{1,2\}  \mid \widehat{\mathsf{G}}^{D}(\omega_l)\big)\\
		& \cdot  |w_1-z_1|^{2-d}|w_2-z_2|^{2-d} \cdot (l+1)\mathbb{P}^{D}\big(A\xleftrightarrow{\ge0} \partial B(c_* N) \big),
	\end{split}
\end{equation}
where in the last inequality we used Lemma \ref{lemma_revise_B1}, (\ref{two_point}) and Lemma \ref{lemma_prepare_loop_decompose}. Similar to (\ref{improve_4.52}) and (\ref{improve_4.53}), it follows from (\ref{newnewineq_2.73}) that when $i_1= i_2$ (resp. $i_1\neq i_2$), the probability $\mathbb{P}\big( \mathrm{ran}(\widetilde{\eta}_{i_j}^{\mathrm{B}})\cap \widetilde{B}_{w_j}(1)\neq \emptyset,\forall j\in \{1,2\}  \mid \widehat{\mathsf{G}}^{D}(\omega_l)\big)$ is bounded from above by $CN^{2-d} |w_1-w_2|^{2-d}$ (resp. $CN^{2(2-d)}$). Combined with the fact that $N^{2-d} |w_1-w_2|^{2-d}\gtrsim N^{2(2-d)}$ (since $w_1,w_2\in \widetilde{B}(M^{\frac{d-4}{2}})\subset \widetilde{B}(N)$), it implies that 
\begin{equation}\label{improve4.67}
	\mathbb{P}\big( \mathrm{ran}(\widetilde{\eta}_{i_j}^{\mathrm{B}})\cap \widetilde{B}_{w_j}(1)\neq \emptyset,\forall j\in \{1,2\}  \mid \widehat{\mathsf{G}}^{D}(\omega_l)\big) \lesssim N^{2-d} |w_1-w_2|^{2-d}. 
\end{equation}
Meanwhile, by (\ref{app3}) we have 
\begin{equation}\label{improve4.68}
	\mathbb{P}^{D}\big(A\xleftrightarrow{\ge0} \partial B(c_*N) \big)\lesssim \mathbb{P}^{D}\big(A\xleftrightarrow{\ge0} x \big)\cdot  N^{d-4}. 
\end{equation}
Plugging (\ref{improve4.67}) and (\ref{improve4.68}) into (\ref{improve4.66}), , we get 
\begin{equation*} 
	\begin{split}
		&\sum\nolimits_{z_1,z_2\in \partial B(M)}\mathbb{P}\big(\hat{\mathsf{F}}^{z_1,z_2}_{i_1,i_2,*}\mid  \widehat{\mathsf{G}}^{D}(\omega_l) \big) \\
		\lesssim & \mathbb{P}^{D}\big(A\xleftrightarrow{\ge0} x \big)\cdot (l+1) N^{-2}\sum\nolimits_{z_1,z_2\in \partial B(M)} \sum\nolimits_{w_1,w_2\in \mathbb{Z}^d}|w_1-w_2|^{2-d}\\
		&\cdot |w_1-z_1|^{2-d}|w_2-z_2|^{2-d}\\
		 \overset{(\ref*{2-d_2-d}),(\ref*{2-d_4-d})}{\lesssim}  &\mathbb{P}^{D}\big(A\xleftrightarrow{\ge0} x \big)\cdot (l+1) N^{-2}\sum\nolimits_{z_1,z_2\in \partial B(M)}|z_1-z_2|^{6-d}\\
		 \overset{(\ref*{|x|-a_aneqd})}{\lesssim} &\mathbb{P}^{D}\big(A\xleftrightarrow{\ge0} x \big)\cdot (l+1) N^{-2}M^{d+4}. 
	\end{split}
\end{equation*}
Combined with (\ref{newaddineq_4.61}), it yields that 
\begin{equation}\label{newaddineq_4.64}
	\begin{split}
		\mathbb{I}_4^* \lesssim &\mathbb{P}^{D}\big(A\xleftrightarrow{\ge0} x \big)\cdot   N^{-2}M^{d+4}   \sum\nolimits_{l\ge 1}  l^3\mathbb{P}\big( \mathsf{G}_l^{D}\big)  \\
		\lesssim &\mathbb{P}^{D}\big(A\xleftrightarrow{\ge0} x \big)\cdot M^{d+2},
			\end{split}
\end{equation}
where in the last inequality we used Lemma \ref{lemma_Gl} and the condition $N\gtrsim  M^{\frac{d-4}{2}}$.

By (\ref{newaddineq_4.54}), (\ref{newaddineq_4.59}), (\ref{newaddineq_4.60}) and (\ref{newaddineq_4.64}), we establish that for all $d\ge 7$, 
\begin{equation}\label{newaddineq_4.65}
	\mathbb{E}^{D}[\mathfrak{X}^2]\lesssim \mathbb{P}^{D}\big(A   \xleftrightarrow{\ge 0}   x \big)\cdot M^{d+2}.
\end{equation}
Parallel to (\ref{newaddineq_4.51}), by applying the Paley-Zygmund inequality and using (\ref{new3.38}), (\ref{newaddineq_4.53}) and (\ref{newaddineq_4.65}), we confirm the desired bound for Item (b) and thus complete the proof of this proposition. \qed

\section{Quasi-multiplicativity}\label{section_proof_theorem1}

The aim of this section is to establish quasi-multiplicativity for the connecting probability of the GFF, as stated in Theorem \ref{prop_new_QM}.

\subsection{Proof of the lower bounds in Theorem \ref{prop_new_QM}}

Before starting the proof, we need some preparations as follows. For any $N\ge 1$ and $A_1,A_2,D_1,D_2\subset \widetilde{\mathbb{Z}}^d$ satisfying the conditions in Theorem \ref{prop_new_QM}, we define 
			\begin{equation}\label{fix5.1}
	\mathcal{C}^{\mathrm{in}}:= \big\{v\in \widetilde{B}(d^{-2}N): v\xleftrightarrow{\widetilde{E}^{\ge 0}\cap \widetilde{B}(d^{-2}N)} A_1 \ \text{or}\ v\xleftrightarrow{ \widetilde{E}^{\le 0} \cap \widetilde{B}(d^{-2}N)} A_1 \big\}
\end{equation}
as the sign cluster containing $A_1$ obtained through exploration within  $\widetilde{B}(d^{-2}N)$. Similarly, we also define its counterpart for $A_2$ as follows: 
\begin{equation}\label{fix5.2}
	\mathcal{C}^{\mathrm{out}}:= \big\{v\in [\widetilde{B}(10N)]^c: v\xleftrightarrow{\widetilde{E}^{\ge 0}\cap [\widetilde{B}(10N)]^c} A_2\ \text{or}\  v\xleftrightarrow{\widetilde{E}^{\le 0}\cap [\widetilde{B}(10N)]^c} A_2 \big\}.
\end{equation}
Based on these two clusters, we denote the following harmonic averages:
\begin{equation}\label{fix5.3}
\mathcal{H}^{\mathrm{in}}:=|\partial \mathcal{B}(d^{-1}N)|^{-1}\sum\nolimits_{y\in \partial \mathcal{B}(d^{-1}N)}\mathcal{H}_{y}(	\mathcal{C}^{\mathrm{in}},	\mathcal{C}^{\mathrm{in}}\cup D_1\cup D_2 ),
\end{equation}
\begin{equation}\label{fix5.4}
	\mathcal{H}^{\mathrm{out}}:=|\partial \mathcal{B}(5N)|^{-1}\sum\nolimits_{y\in \partial \mathcal{B}(5N)}\mathcal{H}_{y}(\mathcal{C}^{\mathrm{out}},\mathcal{C}^{\mathrm{out}}\cup D_1\cup D_2).
\end{equation}
For any $a>0$, we define the events
\begin{equation}\label{AA_xn_in}
	\mathsf{F}_{a}^{\mathrm{in}}:=  \big\{\mathcal{H}^{\mathrm{in}}  \ge a N^{-[(\frac{d}{2}-1)\boxdot  (d-4)]-5\varsigma_d(N)} \big\}\cap \big\{A_1\xleftrightarrow{\le 0} \partial B(d^{-2}N)\big\}^c, 
\end{equation}
\begin{equation}\label{AA_xn_out}
\mathsf{F}_{a}^{\mathrm{out}}:=  \big\{\mathcal{H}^{\mathrm{out}} \ge a N^{-[(\frac{d}{2}-1)\boxdot(d-4)]-5\varsigma_d(N)}  \big\}\cap \big\{A_2 \xleftrightarrow{\le 0} B(10N)\big\}^c. 
\end{equation}
Note that $\mathsf{F}_{a}^{\mathrm{in}}$ (resp. $\mathsf{F}_{a}^{\mathrm{out}}$) is measurable with respect to $\mathcal{C}^{\mathrm{in}}$ (resp. $\mathcal{C}^{\mathrm{out}}$).

\begin{lemma}\label{lemma_newHr}
For any $d\ge 3$, there exists $\cl\label{const_newHr}(d)>0$ such that 
	\begin{equation}\label{late5.7}
		\mathbb{P}^{D_1\cup D_2}\big(\mathsf{F}_{\cref{const_newHr}}^{\mathrm{in}}  \big)\gtrsim   \mathbb{P}^{D_1\cup D_2}\big(A_1 \xleftrightarrow{\ge 0} \partial B(N)   \big),	\end{equation}
	\begin{equation}\label{late5.8}
		\mathbb{P}^{D_1\cup D_2}\big(\mathsf{F}_{\cref{const_newHr}}^{\mathrm{out}}  \big)\gtrsim  \mathbb{P}^{D_1\cup D_2}\big(A_2\xleftrightarrow{\ge 0}   B(N)   \big). 
	\end{equation}
\end{lemma}
\begin{proof}
	Since $\{A_1\xleftrightarrow{\ge 0} \partial B(N)\} \cap \{A_1\xleftrightarrow{\le 0} \partial B(d^{-2}N)\}^c\subset \{\mathcal{H}^{\mathrm{in}}>0\}$, one has 
	\begin{equation}\label{newlyimprove_5.9}
		\begin{split}
			\mathbb{P}^{D_1\cup D_2}\big(\mathsf{F}_{\cref{const_newHr}}^{\mathrm{in}}  \big) \ge & \mathbb{P}^{D_1\cup D_2}\big(\{A_1\xleftrightarrow{\ge 0} \partial B(N)\}\cap \{A_1\xleftrightarrow{\le 0} \partial B(d^{-2}N)\}^c  \big) \\
			&-\mathbb{P}^{D_1\cup D_2}\big(\mathsf{F}', A_1\xleftrightarrow{\ge 0} \partial B(N)  \big),
		\end{split}
	\end{equation}
	where $\mathsf{F}':=  \big\{0<\mathcal{H}^{\mathrm{in}} \le \cref{const_newHr} N^{-[(\frac{d}{2}-1)\boxdot (d-4)]-5\varsigma_d(N)} \big\}\cap \{A_1\xleftrightarrow{\le 0} \partial B(d^{-2}N)\}^c$. For the first term on the right-hand side of (\ref{newlyimprove_5.9}), by the FKG inequality and using the symmetry of $\widetilde{\phi}_\cdot$ (see (\ref{coro2.1_1})), we have 
	\begin{equation}\label{3.126}
	\begin{split}
			&\mathbb{P}^{D_1\cup D_2}\big(\{A_1\xleftrightarrow{\ge 0} \partial B(N)\}\cap \{A_1\xleftrightarrow{\le 0} \partial B(d^{-2}N)\}^c  \big) \\
		\overset{}{\ge} & \mathbb{P}^{D_1\cup D_2}\big(A_1\xleftrightarrow{\ge 0} \partial B(N) \big)\cdot \big[1- \mathbb{P}^{D_1\cup D_2}\big(A_1\xleftrightarrow{\ge 0} \partial B(d^{-2}N) \big) \big] \\
			\overset{(\ref*{crossing_low})\text{--}(\ref*{crossing_high})}{\ge} & \tfrac{1}{2}  \mathbb{P}^{D_1\cup D_2}\big(A_1\xleftrightarrow{\ge 0} \partial B(N) \big),
		\end{split}
	\end{equation}
	where in the last line we used the assumption that $A_1\subset \widetilde{B}(\cref{const_new_QM2}N)$ with a sufficiently small $\cref{const_new_QM2}$. For the second term, by Lemma \ref{lemma_upper_boundarytoset} (with $j=1$), we have 
	\begin{equation}\label{late5.11}
		\begin{split}
			\mathbb{P}^{D_1\cup D_2}\big(\mathsf{F}', A_1\xleftrightarrow{\ge 0} \partial B(N)  \big) \lesssim & N^{d-2} \theta_d(\tfrac{N}{4}) \mathbb{E}^{D_1\cup D_2} \big[ \mathcal{H}^{\mathrm{in}}\cdot \mathbbm{1}_{\mathsf{F}'}  \big]\\
			\overset{(\text{definition of}\ \mathsf{F}')}{ \lesssim} &  \cref{const_newHr} N^{-4\varsigma_d(N)}\mathbb{P}^{D_1\cup D_2}\big(\mathcal{H}^{\mathrm{in}}>0\big). 
		\end{split}
	\end{equation}
	Moreover, since $\{\mathcal{H}^{\mathrm{in}}>0\}\subset \{A_1\xleftrightarrow{\ge 0} \partial B(d^{-2}N)\}$, one has 
	\begin{equation}\label{late5.12}
		\begin{split}
			\mathbb{P}^{D_1\cup D_2}\big(\mathcal{H}^{\mathrm{in}}>0\big) \le &\mathbb{P}^{D_1\cup D_2}\big(A_1\xleftrightarrow{\ge 0} \partial B(d^{-2}N)\big)\\
			\overset{\text{Corollary}\ \ref*{corollary_compare_boundtoset}}{\lesssim } &N^{4\varsigma_d(N)}\mathbb{P}^{D_1\cup D_2}\big(A_1\xleftrightarrow{\ge 0} \partial B(N)\big). 
		\end{split}
	\end{equation}
By plugging (\ref{late5.12}) into (\ref{late5.11}) and taking a sufficiently small $\cref{const_newHr}$, we have
	\begin{equation}\label{newlyimprove_5.13}
		\mathbb{P}^{D_1\cup D_2}\big(\mathsf{F}', A_1\xleftrightarrow{\ge 0} \partial B(N)  \big)\le  \tfrac{1}{4}  \mathbb{P}^{D_1\cup D_2}\big(A_1\xleftrightarrow{\ge 0} \partial B(N) \big).
	\end{equation}
	Combining (\ref{newlyimprove_5.9}), (\ref{3.126}) and (\ref{newlyimprove_5.13}), we obtain (\ref{late5.7}). As the proof of (\ref{late5.8}) employs the same approach as above, we omit further details.
\end{proof}

Now we are ready to prove the lower bounds in Theorem \ref{prop_new_QM}. 
\begin{proof}[Proof of Theorem \ref{prop_new_QM}: lower bounds]
	It follows from Corollary \ref{lemma_boxtobox_hm} that
		\begin{equation}\label{3.130}
		\begin{split}
			\mathbb{P}^{D_1\cup D_2}\big(A_1\xleftrightarrow{\ge 0} A_2 \big)
			\ge &\mathbb{P}^{D_1\cup D_2}\big(A_1\xleftrightarrow{\ge 0} A_2 , \mathsf{F}_{\cref{const_newHr}}^{\mathrm{in}},\mathsf{F}_{\cref{const_newHr}}^{\mathrm{out}}\big)\\
			=& \mathbb{E}^{D_1\cup D_2}\Big[ \mathbbm{1}_{\mathsf{F}_{\cref{const_newHr}}^{\mathrm{in}}\cap \mathsf{F}_{\cref{const_newHr}}^{\mathrm{out}}} \cdot \big(1- e^{-cN^{d-2} \mathcal{H}^{\mathrm{in}} \mathcal{H}^{\mathrm{out}}} \big) \Big].
		\end{split}
	\end{equation}
	On $\mathsf{F}_{\cref{const_newHr}}^{\mathrm{in}}\cap \mathsf{F}_{\cref{const_newHr}}^{\mathrm{out}}$, one has $\mathcal{H}^{\mathrm{in}}\land \mathcal{H}^{\mathrm{out}}\ge \cref{const_newHr}N^{-[(\frac{d}{2}-1)\boxdot  (d-4)]-5\varsigma_d(N)}$, which implies that 
	\begin{equation}\label{late5.14}
		1- e^{-cN^{d-2} \mathcal{H}^{\mathrm{in}} \mathcal{H}^{\mathrm{out}}}  \gtrsim N^{[0 \boxdot (6-d) ]-10\varsigma_d(N)}. 
	\end{equation}
	Combining (\ref{3.130}) and (\ref{late5.14}), we get 
	\begin{equation}\label{late5.15}
		\mathbb{P}^{D_1\cup D_2}\big(A_1\xleftrightarrow{\ge 0} A_2 \big)
			\gtrsim N^{[0\boxdot (6-d) ]-10\varsigma_d(N)} \mathbb{P}^{D_1\cup D_2}\big(\mathsf{F}_{\cref{const_newHr}}^{\mathrm{in}}\cap  \mathsf{F}_{\cref{const_newHr}}^{\mathrm{out}}  \big). 
	\end{equation}
	Moreover, by the FKG inequality, Lemma \ref{lemma_newHr} and Proposition \ref{prop_boundtoset_zero}, we have 
	\begin{equation}\label{3.131}
		\begin{split}
			\mathbb{P}^{D_1\cup D_2}\big(\mathsf{F}_{\cref{const_newHr}}^{\mathrm{in}}\cap  \mathsf{F}_{\cref{const_newHr}}^{\mathrm{out}}  \big)
			\overset{(\text{FKG})}{\ge} &\mathbb{P}^{D_1\cup D_2}\big(\mathsf{F}_{\cref{const_newHr}}^{\mathrm{in}} \big)\mathbb{P}^{D_1\cup D_2}\big(  \mathsf{F}_{\cref{const_newHr}}^{\mathrm{out}}  \big)\\
			\overset{\text{Lemma}\ \ref*{lemma_newHr}}{\gtrsim} & \mathbb{P}^{D_1\cup D_2}\big(A_1 \xleftrightarrow{\ge 0} \partial B(N)   \big)  \mathbb{P}^{D_1\cup D_2}\big(A_2\xleftrightarrow{\ge 0}   B(N)   \big)\\
			\overset{\text{Proposition}\ \ref*{prop_boundtoset_zero}}{\gtrsim} & \mathbb{P}^{D_1}\big(A_1 \xleftrightarrow{\ge 0} \partial B(N)   \big)  \mathbb{P}^{ D_2}\big(A_2\xleftrightarrow{\ge 0}   B(N)   \big).			\end{split}
	\end{equation}
	By (\ref{late5.15}) and (\ref{3.131}), we confirm the lower bounds in Theorem \ref{prop_new_QM}
\end{proof}

\subsection{Proof of the upper bounds in Theorem \ref{prop_new_QM}}\label{section5.2}

For $3\le d\le 6$, it follows from Lemma \ref{lemma_decompose_moresets} and Corollary \ref{corollary_compare_boundtoset} that 
\begin{equation}
\begin{split}
		\mathbb{P}\big(A_1\xleftrightarrow{(D_1\cup D_2)} A_2 \big)\lesssim &\mathbb{P}\big(A_1\xleftrightarrow{(D_1)} \partial B(N) \big) \mathbb{P}\big(A_2\xleftrightarrow{(D_2)} \partial B(CN) \big) \\
		\lesssim  &N^{\varsigma_d(N)} \mathbb{P}\big(A_1\xleftrightarrow{(D_1)} \partial B(N) \big) \mathbb{P}\big(A_2\xleftrightarrow{(D_2)} \partial B(N) \big).
\end{split}
\end{equation}
This together with (\ref{newfinal_compare}) implies the upper bounds in Theorem \ref{prop_new_QM} for $3\le d\le 6$.

Next, we consider the case when $d\ge 7$. For brevity, we denote ${D}:=D_1\cup D_2$. By (\ref{newfinal_compare}), it suffices to estimate $\mathbb{P}(A_1\xleftrightarrow{(D)}A_2)$. In fact, on the event $\{A_1\xleftrightarrow{(D)}A_2\}$, there exists a finite sequence of distinct glued loops $\widetilde{\ell}_1,...,\widetilde{\ell}_k$ in $\widetilde{\mathcal{L}}_{1/2}^{D}$ such that $\mathrm{ran}(\widetilde{\ell}_1)\cap A_1\neq \emptyset$, $\mathrm{ran}(\widetilde{\ell}_k)\cap A_2\neq \emptyset$ and $\mathrm{ran}(\widetilde{\ell}_j)\cap \mathrm{ran}(\widetilde{\ell}_{j+1})\neq \emptyset$ for all $1\le j\le k-1$. Let $j_*$ be the smallest integer in $[1,k]$ such that $\widetilde{\ell}_{j_*}$ intersects $\partial \mathcal{B}(N)$ (such a $j_*$ must exist because $A_1$ and $A_2$ lie on opposite sides of $\partial \mathcal{B}(N)$). We denote $\widetilde{\mathcal{C}}_*:= \cup_{1\le j\le j_*-1}\mathrm{ran}(\widetilde{\ell}_j)$ and $\widetilde{\ell}_*:=\widetilde{\ell}_{j_*}$. By the minimality of $j_*$, we know that $\widetilde{\mathcal{C}}_*$ is disjoint from $\partial \mathcal{B}(N)$. Moreover, $\widetilde{\ell}_*$ is connected to $A_2$ by $\cup_{ j_*+1 \le j\le k}\mathrm{ran}(\widetilde{\ell}_j)$, without involving the loops contained in $\widetilde{\mathcal{C}}_*$. In what follows, we estimate the probability of $\{A_1 \xleftrightarrow{(D)} A_2\}$ under different restrictions on the loop $\widetilde{\ell}_*$.




\textbf{Case 1:} when $\mathrm{ran}(\widetilde{\ell}_*)\cap \partial B(N^{\frac{2}{d-4}})\neq \emptyset$. In this case, we have $\widetilde{\ell}_* \in \widetilde{\mathcal{L}}_{1/2}^{D}[N^{\frac{2}{d-4}},  N]$ (recalling $\widetilde{\mathcal{L}}_{1/2}^{D}[\cdot,\cdot]$ in (\ref{lateradd_4.14})). Thus, by Lemma \ref{lemma_new_decomposition}, the probability of this case is bounded from above by 
\begin{equation}\label{lateimprove_5.18}
	\begin{split}
		&C\Big( \frac{N}{N^{\frac{2}{d-4}}}\Big)^{d-2} \mathbb{P}\big(A_1\xleftrightarrow{(D)} \partial B(cN^{\frac{2}{d-4}}) \big)\mathbb{P}\big(A_2 \xleftrightarrow{(D)} \partial B(CN) \big)\\
		\lesssim &  N^{6-d}\mathbb{P}^{D}\big(A_1\xleftrightarrow{\ge 0} \partial B(N) \big)\mathbb{P}^{D}\big(A_2 \xleftrightarrow{\ge 0} \partial B(N) \big),
	\end{split}
\end{equation}
where in the second line we used (\ref{newfinal_compare}) and Corollary \ref{corollary_compare_boundtoset}.

\textbf{Case 2:} when $\mathrm{ran}(\widetilde{\ell}_*)\cap \partial B(N^{\frac{d-4}{2}})\neq \emptyset$. Similar to Case 1, we have $\widetilde{\ell}_* \in \widetilde{\mathcal{L}}_{1/2}^{D}[N, N^{\frac{d-4}{2}}]$. Therefore, by Lemma \ref{lemma_new_decomposition}, (\ref{newfinal_compare}) and Corollary \ref{corollary_compare_boundtoset}, the probability of this case is bounded from above by 
\begin{equation}\label{lateimprove_5.19}
	\begin{split}
		& C\Big( \frac{N^{\frac{d-4}{2}}}{N}\Big)^{d-2} \mathbb{P}\big(A_1\xleftrightarrow{(D)} \partial B(cN) \big)\mathbb{P}\big(A_2 \xleftrightarrow{(D)} \partial B(CN^{\frac{d-4}{2}}) \big)\\
\lesssim &  N^{6-d}\mathbb{P}^{D}\big(A_1\xleftrightarrow{\ge 0} \partial B(N) \big)\mathbb{P}^{D}\big(A_2 \xleftrightarrow{\ge 0} \partial B(N) \big). 
	\end{split}
\end{equation}

Now it remains to consider the following case.

\textbf{Case 3:} when $\mathrm{ran}(\widetilde{\ell}_*)\subset \widetilde{B}(N^{\frac{d-4}{2}})\setminus \widetilde{B}(N^{\frac{2}{d-4}})$. In this case, there exist $w_1\in \mathcal{B}(N)\setminus B(N^{\frac{2}{d-4}}),w_2\in B(N^{\frac{d-4}{2}})$ such that $\mathrm{ran}(\widetilde{\ell}_*)\cap \widetilde{B}_{w_j}(1)\neq \emptyset$ for all $j\in \{1,2\}$, and that $\widetilde{B}_{w_1}(1) \xleftrightarrow{D\cup \partial \mathcal{B}(N)} A_1$ (certified by the cluster $\mathcal{C}_*$) and $\widetilde{B}_{w_2}(1) \xleftrightarrow{(D)} A_2$ happen disjointly without using the loop $\widetilde{\ell}_*$; in addition, $\mathrm{ran}(\widetilde{\ell}_*)\cap \partial \mathcal{B}(N)\neq \emptyset$ implies that there is $w_3\in \partial \mathcal{B}(N)$ such that $\widetilde{\ell}_*$ includes a sub-path starting from $w_3$ and hitting $\widetilde{B}_{w_1}(1)$ before $\partial \mathcal{B}(N+1)$ (we denote this event, involving $w_1,w_2, w_3$ as mentioned above, by $\mathsf{F}_{w_1,w_2,w_3}$). Referring to Section \ref{subsection_loopsoup}, the total loop measure of loops that intersect both $\widetilde{B}_{w_1}(1)$ and $\widetilde{B}_{w_2}(1)$, and include a sub-path starting from $w_3$ and hitting $\widetilde{B}_{w_1}(1)$ before $\partial \mathcal{B}(N+1)$, is bounded from above by 
\begin{equation}
	C|w_1-w_2|^{2-d}|w_2-w_3|^{2-d}\widetilde{\mathbb{P}}_{w_3}\big(\tau_{\widetilde{B}_{w_1}(1)}<\tau_{\partial \mathcal{B}(N+1)} \big). 
\end{equation}
For any $w_1\in \mathcal{B}(N)$ and $w_3\in \partial \mathcal{B}(N)$, note that $|w_1-w_3|\gtrsim  N-|w_1|$. Therefore, by the strong Markov property and the potential theory for simple random walks (see e.g. \cite[Proposition 1.5.10]{lawler2012intersections}), we have 
\begin{equation}\label{later5.39}
\begin{split}
	&\widetilde{\mathbb{P}}_{w_3}\big(\tau_{\widetilde{B}_{w_1}(1)}<\tau_{\partial \mathcal{B}(N+1)} \big)\\
	\le &\widetilde{\mathbb{P}}_{w_3}\big(\tau_{\partial \mathcal{B}_{w_3}(c(N-|w_1|))}<\tau_{\partial \mathcal{B}(N+1)} \big)\max_{z\in \partial \mathcal{B}_{w_3}(c(N-|w_1|))}\widetilde{\mathbb{P}}_{z}(\tau_{w_1}<\infty)\\
	\lesssim &  \big(N-|w_1|\big)^{-1}|w_1-w_3|^{2-d}. 
	\end{split}
\end{equation}
Therefore, by the BKR inequality and (\ref{newfinal_compare}), we know that the probability of this case is at most (letting $x_N$ be an arbitrarily point in $\partial B(N)$) 
\begin{equation}\label{latefix5.37}
	\begin{split}
		&C\sum_{w_1\in \mathcal{B}(N)\setminus B(N^{\frac{2}{d-4}}),w_2\in B(N^{\frac{d-4}{2}}),w_3\in \partial \mathcal{B}(N)}\big(N-|w_1|\big)^{-1}|w_1-w_3|^{2-d}	\\
		& \cdot  |w_1-w_2|^{2-d}|w_2-w_3|^{2-d}\mathbb{P}^{D}\big( \widetilde{B}_{w_2}(1)\xleftrightarrow{\ge 0} A_2\big)  \mathbb{P}^{D\cup \partial \mathcal{B}(N)}\big( \widetilde{B}_{w_1}(1)\xleftrightarrow{\ge 0} A_1\big) \\
		\lesssim & \sum_{w_1\in \mathcal{B}(N)\setminus B(N^{\frac{2}{d-4}}),w_3\in \partial \mathcal{B}(N)} \big(N-|w_1|\big)^{-1} |w_1-w_3|^{6-2d}\mathbb{P}^{D}\big( x_N\xleftrightarrow{\ge 0} A_2\big)\\
		&\ \ \ \ \ \ \ \ \ \ \ \ \ \ \ \ \ \ \ \ \ \ \ \ \ \ \ \ \ \ \ \  \cdot\mathbb{P}^{D\cup \partial \mathcal{B}(N)}\big( \widetilde{B}_{w_1}(1)\xleftrightarrow{\ge 0} A_1\big),
				\end{split}
\end{equation}
where in the last inequality we used (\ref{2-d_2-d}), Lemma \ref{lemma_revise_B1} and Proposition \ref{prop_pointtoset_harnack}. Moreover, for any $w_1\in \mathcal{B}(N)\setminus B(N^{\frac{2}{d-4}})$, by using (\ref{ineq_compare}), Lemma \ref{lemma_revise_B1}, Lemma \ref{lemma_subharminic} and Proposition \ref{prop_pointtoset_speed}, we have
\begin{equation*}
	\begin{split}
	 	&\mathbb{P}^{D\cup \partial \mathcal{B}(N)}\big( \widetilde{B}_{w_1}(1)\xleftrightarrow{\ge 0} A_1\big)\\
		\overset{(\ref*{ineq_compare})}{\le} &  2\mathbb{P}^{D\cup \partial \mathcal{B}(N+10)}\big( \widetilde{B}_{w_1}(1)\xleftrightarrow{\ge 0} A_1\big) \\
	\overset{\text{Lemma}\ \ref*{lemma_revise_B1}}{	\lesssim } & \mathbb{P}^{D\cup \partial \mathcal{B}(N+10)}\big( w_1\xleftrightarrow{\ge 0} A_1\big) \\
	\overset{\text{Lemma}\ \ref*{lemma_subharminic} }{\lesssim } & \sum\nolimits_{z\in \partial  \mathcal{B}(\frac{1}{2}N^{\frac{2}{d-4}})}\mathbb{P}_{w_1}\big(\tau_{\mathcal{B}(\frac{1}{2}N^{\frac{2}{d-4}})}=\tau_{z}<\tau_{\partial \mathcal{B}(N+10)} \big) \cdot  \mathbb{P}^{D\cup \partial \mathcal{B}(N+10)}\big( z\xleftrightarrow{\ge 0} A_1\big)\\
	\overset{(\ref*{addnew3.25}),(\ref*{ineq_compare})}{\lesssim }&\mathbb{P}_{w_1}\big( \tau_{\mathcal{B}(\frac{1}{2}N^{\frac{2}{d-4}})}<\tau_{\partial \mathcal{B}(N+10)} \big)\cdot \mathbb{P}^{D}\big( x_N\xleftrightarrow{\ge 0} A_1\big)\cdot \Big( \frac{N^{\frac{2}{d-4}}}{N}\Big)^{2-d}.  
	\end{split}
	\end{equation*} 
Meanwhile, applying the potential theory for simple random walks, one has 
\begin{equation*}\label{latefix_5.39}
\begin{split}
		\mathbb{P}_{w_1}\big( \tau_{\mathcal{B}(\frac{1}{2}N^{\frac{2}{d-4}})}<\tau_{\partial \mathcal{B}(N+10)} \big)  \lesssim &  \big(|w_1|^{2-d}-N^{2-d}\big) N^{\frac{2(d-2)}{d-4}} \\
		\lesssim&  \big(N- |w_1| \big)|w_1|^{1-d}N^{\frac{2(d-2)}{d-4}}. 
\end{split}
\end{equation*}
Combining these two estimates, we get 
\begin{equation}\label{submitnew_6.24}
	\mathbb{P}^{D\cup \partial \mathcal{B}(N)}\big( \widetilde{B}_{w_1}(1)\xleftrightarrow{\ge 0} A_1\big) \lesssim \big(N- |w_1| \big)|w_1|^{1-d} N^{d-2}\mathbb{P}^{D}\big( x_N\xleftrightarrow{\ge 0} A_1\big). 
\end{equation}
This implies that the right-hand side of (\ref{latefix5.37}) is at most  
\begin{equation}\label{newfix_5.43}
	\begin{split}
		& CN^{d-2}\mathbb{P}^{D}\big( x_N\xleftrightarrow{\ge 0} A_1\big)\mathbb{P}^{D}\big( x_N\xleftrightarrow{\ge 0} A_2\big)\\
		&   \cdot \sum\nolimits_{w_1\in \mathcal{B}(N)}  |w_1|^{1-d} \sum\nolimits_{w_3\in \partial \mathcal{B}(N)} |w_1-w_3|^{6-2d}\\
		\overset{}{\lesssim } & N^{6-d}\mathbb{P}^{D}\big( A_1\xleftrightarrow{\ge 0} \partial B(N)\big) \mathbb{P}^{D}\big( A_2\xleftrightarrow{\ge 0} \partial B(N)\big),
	\end{split}
\end{equation}
where in the last inequality we used Proposition \ref{prop_relation_pointandboundary} and Lemma \ref{lemma_1-d_6-2d}.

Putting (\ref{lateimprove_5.18}), (\ref{lateimprove_5.19}) and (\ref{newfix_5.43}) together and using (\ref{newfinal_compare}), we obtain that 
\begin{equation}
	\mathbb{P}^{D}(A_1\xleftrightarrow{\ge 0}A_2)\lesssim N^{6-d}\mathbb{P}^{D}\big(A_1\xleftrightarrow{\ge 0} \partial B(N)\big)\mathbb{P}^{D}\big( A_2\xleftrightarrow{\ge 0} \partial B(N)\big) 
\end{equation}
for all $d>6$. To sum up, we conclude the proof of Theorem \ref{prop_new_QM}. \qed

\section{Decay rate of the volume}\label{section_volume}

In this section, we present the proof of Theorem \ref{thm_volume}, which is mainly inspired by \cite[Section 5]{borgs1999uniform}. For any $N\ge n\ge 1$, let $\mathfrak{C}[n,N]$ be the collection of clusters of $\widetilde{E}^{\ge 0}$ intersecting both $B(n)$ and $\partial B(N)$. Recall that for any $A\subset \widetilde{\mathbb{Z}}^d$, the volume of $A$ (denoted by $\mathrm{vol}(A)$) is the number of points in $\mathbb{Z}^d$ contained in $A$. Let $\mathcal{C}[n,N]$ be the cluster in $\mathfrak{C}[n,N]$ that maximizes $\mathrm{vol}(\mathcal{C}[n,N]\cap B(n))$ (we break the tie in a predetermined manner). 


\begin{lemma}\label{lemma_expected_max_volume}
	For $3\le d\le 5$, there exist $\cl\label{const_expected_max_volume1},\cl\label{const_expected_max_volume2},\cl\label{const_expected_max_volume3}>0$ such that for any $N\ge 1$,
	\begin{equation}
		\mathbb{P}\big(\mathrm{vol}(\mathcal{C}[\cref{const_expected_max_volume1} N,N]\cap B(\cref{const_expected_max_volume1}  N))\ge \cref{const_expected_max_volume2}N^{\frac{d}{2}+1}  \big)\ge \cref{const_expected_max_volume3}.   \end{equation}
\end{lemma}
\begin{proof}
	We prove this lemma using the second moment method. Recalling (\ref{crossing_low}), we can take a small $\cref{const_expected_max_volume1}$ such that for all sufficiently large $N$,
	\begin{equation}\label{final_revision_7.2}
		\rho_d(\cref{const_expected_max_volume1} N,N)\le \tfrac{1}{2}.
	\end{equation}
	For simplicity, we denote $\hat{N}:=\cref{const_expected_max_volume1} N$, $\hat{\mathfrak{C}}:=\mathfrak{C}[\hat{N},N]$ and $ \hat{\mathcal{C}}:=\mathcal{C}[\hat{N},N]$. Let $\mathfrak{X}:=\sum_{x\in B(\hat{N})}\mathbbm{1}_{x\xleftrightarrow{\ge 0} \partial B(N)}$. By the lower bounds on $\theta_d$ as stated in (\ref{one_arm_low}), we have 
\begin{equation}\label{ineq_EX}
	\mathbb{E}[\mathfrak{X}]\gtrsim   \hat{N}^{d}\cdot N^{-\frac{d}{2}+1} \asymp N^{\frac{d}{2}+1}.
	\end{equation}
Next, we estimate the second moment of $\mathfrak{X}$:
\begin{equation}
	\mathbb{E}[\mathfrak{X}^2]= \sum\nolimits_{x_1,x_2\in B(\hat{N})} \mathbb{P}\big(\partial B(N)  \xleftrightarrow{\ge 0} x_1,x_2 \big). 
\end{equation}
For any $x_1,x_2\in B(\hat{N})$, let $r=r(x_1,x_2):=|x_1-x_2|$. By Lemma \ref{lemma_final} we have 
\begin{equation}\label{ineq_EX2}
	\begin{split}
		\mathbb{E}[\mathfrak{X}^2] \lesssim & \sum\nolimits_{x_1 \in B(\hat{N})} \mathbb{P}\big(x_1\xleftrightarrow{\ge 0} \partial B(N) \big)  \sum\nolimits_{x_2 \in B(\hat{N})}  |x_1-x_2|^{-\frac{d}{2}+1}\\
		\overset{(\ref*{|x|-a_aneqd})}{\lesssim } & |B(\hat{N})|\cdot \theta_d(\tfrac{1}{2}N)\cdot N^{\frac{d}{2}+1} \overset{(\ref*{one_arm_low})}{\lesssim} N^{d+2}. 
	\end{split}
\end{equation}
By applying the Paley-Zygmund inequality and using (\ref{ineq_EX}) and (\ref{ineq_EX2}), we have\begin{equation}\label{new4.6}
	\mathbb{P}\big( \mathfrak{X}\ge c_\diamond N^{\frac{d}{2}+1}\big) \ge  c_\dagger
\end{equation}
 for some constants $c_\diamond , c_\dagger>0$.

For any $k\ge 1$, let $\mathsf{G}_k:=\{|\hat{\mathfrak{C}}|\ge k \}$ be the event that there are at least $k$ clusters in $\hat{\mathfrak{C}}$. Let $C_\dagger:=\lceil  \log_2(c_\dagger^{-1}) \rceil+1$. By the BKR inequality and (\ref{final_revision_7.2}), one has  
\begin{equation}\label{new4.7}
	\mathbb{P}\big( \mathsf{G}_{C_\dagger}  \big) \overset{(\text{BKR})}{\le} [\rho_d(\hat{N},N) ]^{C_\dagger} \overset{(\ref*{final_revision_7.2})}{\le}  2^{-C_\dagger}< \tfrac{1}{2}c_\dagger.
\end{equation}
 On $\{\mathfrak{X}\ge c_\diamond N^{\frac{d}{2}+1}\}\cap [\mathsf{G}_{C_\dagger}]^c$, by the pigeonhole principle, we know that there exists a cluster $\mathcal{C}'\in \hat{\mathfrak{C}}$ satisfying $\mathrm{vol}(\mathcal{C}'\cap B(\hat{N}))\ge C_\dagger^{-1} c_\diamond  N^{\frac{d}{2}+1}$, which together with the maximality of $\hat{\mathcal{C}}$ implies that $\mathrm{vol}(\hat{\mathcal{C}}\cap B(\hat{N}))\ge C_\dagger^{-1}  c_\diamond N^{\frac{d}{2}+1}$. Thus, by (\ref{new4.6}) and (\ref{new4.7}), we obtain this lemma:
 \begin{equation}
\begin{split}
	&\mathbb{P}\big(\mathrm{vol}(\hat{\mathcal{C}}\cap B(\hat{N}))\ge C_\dagger^{-1} c_\diamond  N^{\frac{d}{2}+1} \big)\\
	\ge &\mathbb{P}\big( \mathfrak{X}\ge c_\diamond N^{\frac{d}{2}+1}\big) -  \mathbb{P}\big( \mathsf{G}_{C_\dagger} \big)
	\overset{(\ref*{new4.6}),(\ref*{new4.7}) }{\ge} \tfrac{1}{2}c_\dagger. \qedhere
\end{split}
\end{equation}

\end{proof}

Based on Lemma \ref{lemma_expected_max_volume}, the proof of Theorem \ref{thm_volume} becomes straightforward.

\begin{proof}[Proof of Theorem \ref{thm_volume}]
	As in the proof of Lemma \ref{lemma_expected_max_volume}, we denote $\hat{N}:=\cref{const_expected_max_volume1} N$ and $ \hat{\mathcal{C}}:=\mathcal{C}[\hat{N},N]$. For any $N,s\ge 1$, we have 
\begin{equation}\label{3.3}
	\begin{split}
		\mathbb{E}\big[\mathrm{vol}(\hat{\mathcal{C}}\cap B(\hat{N}))\big]\le& s+ \mathbb{E}\big[\mathrm{vol}(\hat{\mathcal{C}}\cap B(\hat{N}))\cdot \mathbbm{1}_{\mathrm{vol}(\hat{\mathcal{C}})\ge s}\big]\\
		\le &s+ \sum\nolimits_{x\in  B(\hat{N})}\mathbb{P}\big(\mathrm{vol}\big(\mathcal{C}^{\ge 0}(x)\big) \ge s \big)\\
		\le &s+ C\hat{N}^d \mathbb{P}\big(\mathrm{vol}\big(\mathcal{C}^{\ge 0}(\bm{0})\big) \ge s \big).
	\end{split}
\end{equation}
By taking $s=\frac{1}{2}\mathbb{E}\big[\mathrm{vol}(\hat{\mathcal{C}}\cap B(\hat{N}))\big]$ in (\ref{3.3}), we get 
\begin{equation}
	\mathbb{P}\big( \mathrm{vol}\big(\mathcal{C}^{\ge 0}(\bm{0})\big) \ge \tfrac{1}{2}\mathbb{E}\big[\mathrm{vol}(\hat{\mathcal{C}}\cap B(\hat{N}))\big] \big) \gtrsim N^{-d} \mathbb{E}\big[\mathrm{vol}(\hat{\mathcal{C}}\cap B(\hat{N}))\big]. 
\end{equation}
Combined with $\mathbb{E}\big[\mathrm{vol}(\hat{\mathcal{C}}\cap B(\hat{N}))\big]\gtrsim N^{\frac{d}{2}+1}$ (which follows from Lemma \ref{lemma_expected_max_volume}), it implies that for some constant $c_*>0$, 
\begin{equation}\label{3.15}
		\mathbb{P}\big( \mathrm{vol}\big(\mathcal{C}^{\ge 0}(\bm{0})\big) \ge c_*N^{\frac{d}{2}+1} \big) \gtrsim  N^{-\frac{d}{2}+1}. 
\end{equation}
By letting $M=c_*N^{\frac{d}{2}+1}$, we obtain Theorem \ref{thm_volume}. 
\end{proof}

As promised in Section \ref{section_intro}, we provide the proof of (\ref{volume_extend}) as follows. 
\begin{proof}[Proof of (\ref{volume_extend})]
	For any $M\ge 1$, on the event $\big\{\bm{0}\xleftrightarrow{\ge 0} \partial B(M^{\frac{2}{d+2}})\big\}^c$, we have $\mathcal{C}^{\ge 0}(\bm{0})\cap \mathbb{Z}^d \subset B(M^{\frac{2}{d+2}})$, and hence $\mathrm{vol}\big(\mathcal{C}^{\ge 0}(\bm{0})\big)= \mathrm{vol}\big(\mathcal{C}^{\ge 0}(\bm{0})\cap B(M^{\frac{2}{d+2}})\big)$. Applying the Markov's inequality, this implies that 
\begin{equation}\label{v_d_theta_d}
\begin{split}
		\nu_d(M)\le&  \theta_d(M^{\frac{2}{d+2}})+ \mathbb{P}\big( \mathrm{vol}\big(\mathcal{C}^{\ge 0}(\bm{0})\cap B(M^{\frac{2}{d+2}})\big) \ge M  \big)\\
		\le &\theta_d(M^{\frac{2}{d+2}})+ M^{-1}\mathbb{E}\Big[ \mathrm{vol}\big(\mathcal{C}^{\ge 0}(\bm{0})\cap B(M^{\frac{2}{d+2}})\big)\Big]. 
		\end{split}
\end{equation}
By the two-point function estimate in (\ref{two_point}), we have 
\begin{equation}\label{E_vol_M}
	\begin{split}
		\mathbb{E}\big[ \mathrm{vol}\big(\mathcal{C}^{\ge 0}(\bm{0})\cap B(M^{\frac{2}{d+2}})\big)\big] \lesssim & \sum\nolimits_{x\in B(M^{\frac{2}{d+2}})}  |x|^{2-d}\overset{(\ref*{|x|-a_aneqd})}{\lesssim} M^{\frac{4}{d+2}}.	\end{split}
\end{equation}
Plugging (\ref{E_vol_M}) into (\ref{v_d_theta_d}), we conclude the desired bound (\ref{volume_extend}).
\end{proof}


	\bibliographystyle{plain}
	\bibliography{ref}
	
\end{document}